\documentclass[12pt,twoside]{article}

\usepackage{amsmath}
\usepackage{amssymb}
\usepackage{amsthm}
\usepackage[dvips]{graphicx}

\newtheorem{thm}{Theorem}[section]
\newtheorem{defi}[thm]{Definition}
\newtheorem{lem}[thm]{Lemma}
\newtheorem{prop}[thm]{Proposition}
\newtheorem{cor}[thm]{Corollary}
\newtheorem{rem}[thm]{Remark}

\DeclareMathOperator{\Ker}{Ker}

\title{\textbf{\large Deformations of special Legendrian submanifolds\\
in Sasaki-Einstein manifolds}}
\author{\textsc{Takayuki Moriyama}}
\date{}

\begin{document}
\maketitle
\thispagestyle{empty}


\begin{quote}
\textbf{\fontsize{10pt}{15pt}\selectfont Abstract.}
\textrm{\fontsize{10pt}{15pt}\selectfont 
In this paper we study the deformation theory of submanifolds characterized by a system of differential forms 
and provide a criterion for deformations of such submanifolds to be unobstructed. 
We apply this deformation theory to special Legendrian submanifolds in Sasaki-Einstein manifolds. 
In general, special Legendrian deformations have the obstruction. 
However, we show that the deformation space of special Legendrian submanifolds is 
the intersection of two larger smooth deformation spaces of different types. 
We also prove that any special Legendrian submanifold admits smooth deformations, which are not special Legendrian deformations, given by harmonic $1$-forms. 
}
\end{quote}

\section{Introduction}
Calabi-Yau manifolds can be studied in many categories, 
complex geometry, K\"{a}hler geometry, algebraic geometry, etc. 
From the view point of calibrated geometry 
which was initiated by Harvey and Lawson~\cite{HL}, 
any Calabi-Yau structure induces a calibration 
given by the real part of a holomorphic volume form. 
Recall that a calibration is a closed differential form with comass $1$ 
and a calibrated submanifold is defined by a submanifold in which 
the calibration equals to a volume form. 
Special Lagrangian submanifolds in Calabi-Yau manifolds are calibrated submanifolds. 
McLean provided the deformation theory for calibrated submanifolds~\cite{Mc}. 
In particular, he showed that the moduli space of compact special Lagrangian submanifolds 
is smooth and the tangent space is given by the vector space of harmonic $1$-forms. 
An important point of his proof was that any special Lagrangian submanifold is characterized as a submanifold 
where the imaginary part of the holomorphic volume form and the symplectic K\"{a}hler form vanish. 

In this paper, we consider deformations of submanifolds characterized 
by {\it differential forms which are not necessarily closed}. 
We assume that $(M,g)$ is a smooth compact Riemannian manifold.  
We call $X$ a submanifold in $M$ 
if there exists an embedding $\iota:X\hookrightarrow M$. 
Let $X$ be a compact connected submanifold in $M$. 
We consider a normal deformation of a submanifold $X$, 
that is, an embedding $f:X\hookrightarrow M$ with a family $\{f_{t}\}_{t\in [0,1]}$ of embeddings 
$f_{t}:X\hookrightarrow M$ such that $f_{0}=\iota$, $f_1=f$ and $\frac{d}{d t}f_{t}\in \Gamma(NX_t)$ 
where $NX_t$ is the normal bundle of $X_t=f_t(X)$ in $(M, g)$. 
We call the set of normal deformations of $X$ {\it the moduli space of deformations of $X$}. 
Now we choose a system $\Phi=(\varphi_1,\dots,\varphi_m)\in \Gamma (\oplus_{i=1}^{m}\wedge^{k_{i}}T^*M)$ 
of smooth differential forms on $M$ such that the pull-back $\iota^{*}\Phi$ of $\Phi$ by $\iota:X\to M$ vanishes : 
\begin{equation*}
\iota^{*}\Phi=(\iota^{*}\varphi_1,\dots,\iota^{*}\varphi_m)=(0,\dots,0).
\end{equation*} 
We call a deformation $f=\{f_{t}\}_{t\in [0,1]}$ of $X$ a {\it $\Phi$-deformation} 
if it satisfies $f_{t}^{*}\Phi=(0,\dots,0)$ for each $t\in [0,1]$, and we denote by $\mathcal{M}_{X}(\Phi)$ 
the moduli space of $\Phi$-deformations of $X$. 
We can regard $X$ itself as an element of $\mathcal{M}_{X}(\Phi)$ and denote it by $0_X$. 
We take a positive integer $s$ and a real number $\alpha$ with $0<\alpha<1$ 
and consider a deformation of $X$ of $C^{s, \alpha}$-class, that is, 
a $C^{s,\alpha}$-embedding $f:X\hookrightarrow M$ with a family $\{f_{t}\}_{t\in [0,1]}$ of $C^{s,\alpha}$-embeddings 
$f_{t}:X\hookrightarrow M$ such that $f_{0}=\iota$, $f_1=f$, 
$\frac{d}{d t}f_{t}\in \Gamma(NX_t)$ and $f_{t}^{*}\Phi=0$.  
Let $\mathcal{M}^{s,\alpha}_{X}(\Phi)$ denote the moduli space of $\Phi$-deformations of $X$ of $C^{s,\alpha}$-class. 

A Sasaki-Einstein manifold is a $(2n+1)$-dimensional Riemannian manifold $(M,g)$ 
whose metric cone $(C(M),\overline{g})=(\mathbb{R}_{>0}\times M, dr^{2}+r^{2}g)$ 
is a Ricci-flat K\"{a}hler manifold where $r$ is the coordinate of $\mathbb{R}_{>0}$. 
We assume that $M$ is simply connected. 
Then the cone $C(M)$ is a complex $(n+1)$-dimensional Calabi-Yau manifold 
which admits a holomorphic $(n+1)$-form $\Omega$ and 
a K\"{a}hler form $\omega$ on $C(M)$ satisfying the Monge-Amp\`{e}re equation 
$\Omega\wedge \overline{\Omega}=c_{n+1}\omega^{n+1}$ for a constant $c_{n+1}$. 
An $n$-dimensional submanifold $X$ in a Sasaki-Einstein manifold $(M, g)$ is 
a {\it special Legendrian submanifold} if 
the cone $C(X)$ is a special Lagrangian submanifold in $C(M)$. 
We characterize a Sasaki-Einstein manifold as a Riemannian manifold with 
a contact $1$-form $\eta$ and a complex valued $n$-form $\psi$ such that 
$(\psi, \frac{1}{2}d\eta)$ is an almost transverse Calabi-Yau structure with 
$d\psi=(n+1)\sqrt{-1}\,\eta\wedge \psi$ (See Section 3.2 for the definition of almost transverse Calabi-Yau structures). 
Then $X$ is a special Legendrian submanifold 
if and only if $\iota^{*}\psi^{\rm Im}=0$ and $\iota^{*}\eta=0$ where $\iota$ is the inclusion $\iota:X\hookrightarrow M$ 
and $\psi^{\rm Im}$ is the imaginary part of $\psi$. 

Let $X$ be a special Legendrian submanifold and $\mathcal{M}_{X}$ the moduli space of special Legendrian deformations of $X$. 
Then $\mathcal{M}_{X}$ is equal to the moduli space $\mathcal{M}_{X}(\psi^{\rm Im}, \eta)$ of $(\psi^{\rm Im}, \eta)$-deformations of $X$. 
The infinitesimal deformation space of $X$ is 
given by the eigenspace ${\rm Ker}(\Delta_{0}-2(n+1))$ of the Laplace operator $\Delta_{0}$ on $\wedge^0_X$ with the eigenvalue $2(n+1)$. 
If the obstruction for $\mathcal{M}_{X}$ vanishes, then $\mathcal{M}_{X}$ is a smooth manifold and the infinitesimal deformation space is the tangent space of $\mathcal{M}_{X}$ at $0_X$. 
However, the obstruction for $\mathcal{M}_{X}$ does not vanish in general. 
Let $\omega^{T}$ be the $2$-form $\frac{1}{2}d\eta$ on $M$. 
We denote by $\mathcal{N}_{X}$ and $\mathcal{L}_{X}$ 
the moduli spaces $\mathcal{M}_{X}(\psi^{\rm Im}, \omega^{T})$ and $\mathcal{M}_{X}(\eta)$, respectively. 
We will see that the obstructions for $\mathcal{N}_{X}$ and $\mathcal{L}_{X}$ vanish, and $\mathcal{N}_{X}$ and $\mathcal{L}_{X}$ are infinite dimensional Banach manifolds. 
We fix an integer $s\ge 3$ and a real number $\alpha$ with $0<\alpha<1$ and denote by 
$\mathcal{N}_{X}^{s,\alpha}$ and $\mathcal{L}_{X}^{s,\alpha}$ 
the moduli space $\mathcal{M}_{X}^{s,\alpha}(\psi^{\rm Im}, \omega^{T})$ and $\mathcal{M}_{X}^{s,\alpha}(\eta)$, respectively. 
\begin{thm}\label{s1t1}
The moduli space $\mathcal{M}_{X}$ is the intersection $\mathcal{N}_{X}\cap\mathcal{L}_{X}$ 
where $\mathcal{N}_{X}^{s,\alpha}$ and $\mathcal{L}_{X}^{s,\alpha}$ are smooth. 
\end{thm}

The Reeb foliation on a Sasaki-Einstein manifold $(M,g)$ induces a line bundle $F$ on $M$. 
Then we regard the quotient bundle $NF=TM/F$ as the subbundle of $TM$ which is orthogonal to $F$. 
A $\Phi$-deformation $f=\{f_{t}\}_{t\in [0,1]}$ of $X$ is called a \textit{transverse $\Phi$-deformation} of $X$ 
if there exists a family $\{h_{t}\}_{t\in [0,1]}$ of diffeomorphisms of $X$ with $h_0={\rm id}_{X}$ and 
$\frac{d}{d t}h_t|_{t=0}=0$ such that $\frac{d}{d t}f_t \circ h_t\in \Gamma(NF|_{X_t})$ for each $t\in [0,1]$. 
We consider transverse $(\psi^{\rm Im}, \omega^{T})$-deformations of a special Legendrian submanifold $X$ 
and denote by $\mathcal{N}_{X}^{T}$ the moduli space of transverse $(\psi^{\rm Im}, \omega^{T})$-deformations of $X$. 
\begin{thm}\label{s1t2}
The moduli space $\mathcal{N}_{X}^{T}$ is smooth at $0_X$ 
and the tangent space $T_{0_X}\mathcal{N}_{X}^{T}$ 
is isomorphic to $H^{1}(X)$. 
\end{thm}

Let $(M,g)$ be a Sasaki manifold with an almost transverse Calabi-Yau structure $(\psi, \frac{1}{2}d\eta)$ such that 
$d\psi=\kappa \sqrt{-1}\,\eta\wedge \psi$ for a real constant $\kappa$. 
In the case $\kappa=n+1$, $(M,g)$ is a Sasaki-Einstein manifold. 
The metric cone $(C(M), \overline{g})$ is an almost Calabi-Yau manifold which admits a holomorphic $(n+1)$-form $\Omega$ and 
a K\"{a}hler form $\omega$ satisfying 
$\Omega\wedge \overline{\Omega}=r^{2(\kappa-n-1)}c_{n+1}\omega^{n+1}$ on $C(M)$. 
A special Lagrangian submanifold in $C(M)$ is defined by an $(n+1)$-dimensional submanifold where $\Omega^{\rm Im}$ and $\omega$ vanish. 
We define a special Legendrian submanifold in $M$ as an $n$-dimensional submanifold $X$ such that $C(X)$ is a special Lagrangian submanifold in $C(M)$. 
Then $X$ is a special Legendrian submanifold 
if and only if $\iota^{*}\psi^{\rm Im}=0$ and $\iota^{*}\eta=0$ where $\iota$ is the inclusion $\iota:X\hookrightarrow M$. 
Let $X$ be a special Legendrian submanifold and $\mathcal{M}_{X}$ the moduli space of special Legendrian deformations of $X$. 
Then we can extend Theorem~\ref{s1t1} and Theorem~\ref{s1t2} in this case (see Theorem~\ref{s5.3t1} and Theorem~\ref{s5.3t2}). 
We say that a special Legendrian submanifold $X$ is \textit{rigid} if 
any special Legendrian deformation of $X$ is induced by the group ${\rm Aut}(\eta, \psi)$ of diffeomorphisms of $M$ preserving $\eta$ and $\psi$. 
\begin{thm}\label{s1t3}
The infinitesimal deformation space of $X$ is  
isomorphic to the space ${\rm Ker}(\Delta_{0}-2\kappa)$. 
If $\kappa=0$, then $X$ is rigid and $\mathcal{M}_{X}$ is a $1$-dimensional manifold. 
If $\kappa<0$, then $X$ does not have any non-trivial deformation and $\mathcal{M}_{X}=\{0_X\}$. 
\end{thm}

A special Legendrian submanifold $X$ is rigid if and only if the special Lagrangian cone $C(X)$ is rigid, that is, 
any deformation of $C(X)$ is induced by the action of the group ${\rm Aut}(\Omega, \omega, r)$ 
of diffeomorphisms of $C(M)$ preserving $\Omega$, $\omega$ and $r$. 
A typical example of Sasaki-Einstein manifolds is 
the odd-dimensional unit sphere $S^{2n+1}$ with the standard metric, 
then the metric cone is the complex space $\mathbb{C}^{n+1}\backslash \{0\}$. 
Joyce introduced the rigidity of special Lagrangian cones in $\mathbb{C}^{n+1}\backslash \{0\}$~\cite{J3}. 
There exist some rigid special Lagrangian cones in $\mathbb{C}^{n+1}\backslash \{0\}$ 
and the corresponding rigid special Legendrian submanifolds in $S^{2n+1}$~\cite{H, J2, O2}. 
Special Legendrian submanifolds have also the aspect of minimal Legendrian submanifolds. 
We call that a minimal Legendrian submanifold is rigid 
if any minimal Legendrian deformation of $X$ is induced by the group ${\rm Aut}(\eta, g)$ of diffeomorphisms of $M$ preserving $\eta$ and $g$. 
Hence there exist two kinds of rigidity conditions for special Legendrian submanifolds. 
We show that these conditions are equivalent : 
\begin{thm}\label{s1t4}
Let $X$ be a special Legendrian submanifold. 
If $\kappa>0$, then $X$ is rigid as a special Legendrian submanifold if and only if 
it is rigid as a minimal Legendrian submanifold. 
\end{thm}

This paper is organized as follows. 
In Section 2, we provide the $\Phi$-deformation theory of 
submanifolds in a general Riemannian manifold and a criterion for $\Phi$-deformations to be unobstructed (see Proposition~\ref{s2.1p1}). 
Moreover we see some typical examples of such $\Phi$-deformations. 
In Section 3, we introduce almost transverse Calabi-Yau structures and 
prove that any Sasaki-Einstein manifold is characterized by such a structure. 
In Section 4, we show Theorem~\ref{s1t1} and Theorem~\ref{s1t2} (see Theorem~\ref{s4.3t1} and Theorem~\ref{s4.1.4t1}). 
In the last section, we introduce special Legendrian submanifolds in Sasaki manifolds with almost transverse Calabi-Yau structures 
and prove Theorem~\ref{s1t3} (see Theorem~\ref{s5.3p3}) and a generalization of Theorem~\ref{s1t1} and Theorem~\ref{s1t2} (see Theorem~\ref{s5.3t1} and Theorem~\ref{s5.3t2}). 
We study minimal Legendrian deformations of special Legendrian submanifolds and show Theorem~\ref{s1t4} (see Theorem~\ref{s5.4t1}). 

\section{Deformations of submanifolds}
In this section, we assume that $(M,g)$ is a smooth Riemannian manifold. 
We provide a criterion for the smoothness of moduli spaces of submanifolds and apply it to some examples. 

\subsection{Smoothness of moduli spaces of $\Phi$-deformations}
Let $X$ be a compact submanifold in $M$ with an embedding $\iota:X\hookrightarrow M$. 
We choose a system $\Phi=(\varphi_1,\dots,\varphi_m)\in \Gamma (\oplus_{i=1}^{m}\wedge^{k_{i}}T^*M)$ 
of smooth differential forms on $M$. 
Suppose the pull-back $\iota^{*}\Phi$ of $\Phi$ by $\iota:X\to M$ vanishes : 
\begin{equation*}\label{s2.1e1}
\iota^{*}\Phi=(\iota^{*}\varphi_1,\dots,\iota^{*}\varphi_m)=(0,\dots,0)
\end{equation*} 
We call an embedding $f:X\hookrightarrow M$ is a {\it $\Phi$-deformation} of $X$ 
if there exists a family $\{f_{t}\}_{t\in [0,1]}$ of embeddings $f_{t}:X\hookrightarrow M$ with $f_{0}=\iota$ and $f_1=f$ 
such that $f_{t}^{*}\Phi=(0,\dots,0)$ for each $t\in [0,1]$. 
We assume that such a family $\{f_{t}\}_{t\in [0,1]}$ is normal, that is, $\frac{d}{d t}f_{t}\in \Gamma(NX_t)$ for each $t\in [0,1]$. 
For simplicity, we write a deformation $f$ of $X$ with the family $\{f_{t}\}_{t\in [0,1]}$ by $f=\{f_{t}\}_{t\in [0,1]}$. 
We denote by $\mathcal{M}_{X}(\Phi)$ the moduli space of $\Phi$-deformations of $X$. 
It follows from the tubular neighbourhood theorem that 
there exists a neighbourhood of $X$ in $M$ which is identified with 
a neighbourhood $\mathcal{U}$ of the zero section of $NX$ by 
the exponential map. 
We define $V_{1}$ and $V_{2}$ as the vector spaces
\begin{eqnarray*}
V_{1}&=&\Gamma(NX), \\
V_{2}&=&\Gamma(\oplus_{i=1}^{m}\wedge^{k_{i}}T^*X)
\end{eqnarray*}
of all smooth sections of $NX$ and $\oplus_{i=1}^{m}\wedge^{k_{i}}T^*X$, respectively. 
We denote by $U$ the set $\{ v\in V_1 \mid v_{x}\in \mathcal{U}, x\in X \}$ : 
\begin{equation}\label{s2.1e2}
U=\{ v\in V_1 \mid v_{x}\in \mathcal{U}, x\in X \}.
\end{equation}
The exponential map induces the embedding ${\rm exp}_{v}:X\hookrightarrow M$ for each $v\in U$. 
Then we define the map $F:U\to V_{2}$ by 
\begin{equation*}
F(v)=\exp_{v}^{*}\Phi=(\exp_{v}^{*}\varphi_1,\dots,\exp_{v}^{*}\varphi_m)
\end{equation*}
for any $v\in U$. 
We can consider $X$ as an element of $\mathcal{M}_{X}(\Phi)$ 
since $\iota:X\to M$ can be the trivial deformation of $X$. 
Hence we denote by $0_X$ the element $X$ of $\mathcal{M}_X(\Phi)$. 
If the inverse image $F^{-1}(0)$ of the origin of $V_2$ is smooth at $0_X$, 
then $F^{-1}(0)$ is identified with a neighbourhood of $0_X$ in $\mathcal{M}_{X}(\Phi)$ in the $C^1$ sense. 
Let $D_{1}$ be the linearization of $F$ at $0$ : 
\begin{equation*}
D_{1}=d_{0}F:V_{1}\to V_{2}.
\end{equation*}
Then the infinitesimal deformation space of $X$ is given by $\Ker D_{1}$. 
Let $D_1^*$ be the formal adjoint operator of $D_1$. 
\begin{lem}\label{s2.1l1}
The equation $D_1^*\circ F(v)=0$ is a partial differential equation whose order is at most two, 
and quasi-linear if the order is two. 
\end{lem}
\begin{proof}
We assume that the dimension of $X$ is $n$ and the rank of $NX$ is $\ell$. 
Let $U_X$ be an open set in $X$ and $(x_1,\dots,x_n)$ a coordinate of $U_X$. 
Taking a local frame $\{v_1,\dots,v_{\ell}\}$ of $NX$ over $U_X$, 
then we have a trivialization $U_X\times \mathbb{R}^{\ell}$ of $NX|_{U_X}$ by 
the correspondence of $(x,y_1,\dots,y_{\ell})\in U_X\times \mathbb{R}^{\ell}$ to $\sum_{j=1}^{\ell}y_jv_j(x)\in NX|_{U_X}$. 
By the exponential map, the tubular neighbourhood $\mathcal{U}\cap NX|_{U_X}$ of $NX|_{U_X}$ is isomorphic to an open set $U_M$ in $M$, 
and $(U_M, x_1,\dots,x_n, y_1,\dots,y_{\ell})$ is a local coordinate of $M$. 
Let $v$ be an element of the set $U$ of (\ref{s2.1e2}). If $v$ is expressed by $v=\sum_{j=1}^{\ell}f_jv_j$ on $U_X$, 
then $\exp_{v}^*dx_i=dx_i$ and $\exp_{v}^*dy_j=\sum_{i=1}^{n}\frac{\partial f_j}{\partial x_i}dx_i$ on $U_X$ for each $i=1,\dots,n, j=1,\dots,\ell$. 
Since any differential form on $M$ is generated by wedge products of $dx_i$ and $dy_j$ on $U_M$, 
the equation $F(v)=\exp_{v}^{*}\Phi=0$ is a partial differential equation whose order is at most one. 
Hence the order of the equation $D_1^*\circ F(v)=0$ is at most two. 
The first order term of $F(v)=0$ may contain 
\begin{equation}\label{s2.1e3}
(\frac{\partial f_1}{\partial x_1})^{m_{11}}(\frac{\partial f_1}{\partial x_2})^{m_{12}}\cdots (\frac{\partial f_{\ell}}{\partial x_{n}})^{m_{\ell n}}
\end{equation}
for integers $m_{11}, m_{12}, \dots, m_{\ell n}$ with $m_{ij}=0$ or $1$ such that $0\le m_{1i}+\dots+ m_{\ell i}\le 1$ for each $i$. 
If the order of $D_1^*\circ F(v)=0$ is two, then the second order term is generated by the partial derivative of (\ref{s2.1e3}), 
and hence $D_1^*\circ F(v)=0$ is quasi-linear. 
\end{proof}
We fix an integer $s\ge 3$ and a real number $\alpha$ with $0<\alpha<1$. 
Then we set the Banach spaces 
\begin{eqnarray*}
V_{1}^{s,\alpha}&=& C^{s,\alpha}(NX), \\
V_{2}^{s-1,\alpha}&=&C^{s-1,\alpha}(\oplus_{i=1}^{m}\wedge^{k_{i}}T^*X)
\end{eqnarray*}
with respect to the H\"{o}lder norm $\parallel \cdot \parallel_{C^{s,\alpha}}$ 
induced by the Riemannian metric $\iota^{*}g$ on $X$. 
We define $U^{s,\alpha}$ as the set $\{ v\in V_1^{s,\alpha}\mid v_{x}\in \mathcal{U}, x\in X \}$. 
Then we can extend the map $F$ to the smooth map $F^{s,\alpha}:U^{s,\alpha}\to V_{2}^{s-1,\alpha}$. 
We define $D_{1}^{s,\alpha}$ as the linearization $d_{0}F^{s,\alpha}$ of $F^{s,\alpha}$. 
Let $\mathcal{M}^{s,\alpha}_{X}(\Phi)$ denote 
the moduli space of $\Phi$-deformations of $C^{s,\alpha}$-class. 

\begin{prop}\label{s2.1p1} 
Suppose that there exist a vector space $V_{3}$ of smooth sections of a vector bundle $E$ on $X$ 
and a differential operator $D_{2}$ with the differential complex 
\begin{equation*}
0\to V_{1}\stackrel{D_{1}}{\longrightarrow}
V_{2}\stackrel{D_{2}}{\longrightarrow}
V_{3}\to 0. 
\end{equation*}
Let $D_2^*$ be a formal adjoint operator of $D_2$. 
If $P_2=D_{1}\circ D_{1}^{*}+D_{2}^{*}\circ D_{2}$ is elliptic 
and ${\rm Im}(F)\subset {\rm Im}(D_1)$, 
then the moduli space $\mathcal{M}_{X}^{s,\alpha}(\Phi)$ is smooth at $0_X$ 
and the tangent space $T_{0_X}\mathcal{M}_{X}^{s,\alpha}(\Phi)$ is given by $\Ker D_{1}^{s,\alpha}$. 
Moreover, if $P_1=D_{1}^{*}\circ D_{1}$ is also elliptic, 
then the moduli space $\mathcal{M}_{X}(\Phi)$ is smooth at $0_X$ 
and the tangent space $T_{0_X}\mathcal{M}_{X}(\Phi)$ is given by ${\rm Ker}(D_{1})$. 
\end{prop}
\begin{proof}
If $P_{2}$ is elliptic, then the map $D_{1}^{s,\alpha}:V_{1}^{s,\alpha}\to V_{2}^{s-1,\alpha}$ 
has a closed image ${\rm Im}(D_{1}^{s,\alpha})$ and 
there exists a right inverse of 
the map $D_{1}^{s,\alpha}:V_{1}^{s,\alpha}\to {\rm Im}(D_{1}^{s,\alpha})$. 
It follows from the assumption ${\rm Im}(F)\subset {\rm Im}(D_1)$ 
that ${\rm Im}(F^{s,\alpha})\subset {\rm Im}(D_1^{s,\alpha})$. 
Then we obtain the smooth map $F^{s,\alpha}:U^{s,\alpha}\to {\rm Im}(D_{1}^{s,\alpha})$ such that 
the derivative $d_{0}F:V_{1}^{s,\alpha}\to {\rm Im}(D_{1}^{s,\alpha})$ has the right inverse $D_1^*\circ G$ where $G$ is Green's operator of $P_2$. 
The implicit function theorem implies that 
there exist a neighbourhood $W$ of $0$ in ${\rm Ker}(D_{1}^{s,\alpha})$ and 
a smooth map $\varphi: W\to W^{\perp}$ where $W^{\perp}$ is the orthogonal complement of $W$ in $V_1^{s,\alpha}$ 
such that $\varphi(0)=0$ and $F(u,\varphi(u))=0$. 
Hence, $(F^{s,\alpha})^{-1}(0)$ has a manifold structure at $0_X$ 
whose tangent space is ${\rm Ker}(D_{1}^{s,\alpha})$. 
Moreover, if $P_{1}$ is elliptic, then ${\rm Ker}(D_{1}^{s,\alpha})$ coincides with 
the finite dimensional vector space ${\rm Ker}(D_{1})$ by the elliptic regularity method. 
We will show that the map $\varphi$ provides a manifold structure of $F^{-1}(0)$ at $0_X$. 
If we take the Taylor expansion 
\begin{equation}\label{s2.1e4}
F(v)=D_1(v)+R(v)
\end{equation}
of $F$ at $0$ for $v\in V_1^{s,\alpha}$ where $R(v)$ is the higher term with respect to $v$, 
then $D_1^*\circ F(v)$ is given by 
\[
D_1^*\circ F(v)=P_1(v)+D_1^*(R(v))
\]
and $\frac{D_1^*(R(v))}{\|v\|}\to 0$ as $v\to 0$ in $V_{1}^{s,\alpha}$. 
Hence $D_1^*\circ F(v)=0$ is a second order elliptic partial differential equation for any sufficiently small $v\in V_1^{s,\alpha}$. 
It follows form Lemma~\ref{s2.1l1} that $D_1^*\circ F(v)=0$ is quasi-linear elliptic, 
and the solution $v$ is $C^{\infty}$-class by Morrey's elliptic regularity results~\cite{Mo}. 
Now an element $(u,\varphi(u))$ of $V_1^{s,\alpha}$ for $u\in W$ is a solution of $D_1^*\circ F(v)=0$. 
Hence there exists a sufficiently small $\varepsilon >0$ such that $\varphi(u)$ is $C^{\infty}$-class 
for any $u\in W$ with $\|u\|_{C^{s,\alpha}}<\varepsilon$ 
since $\frac{D_1^*(R(u,\varphi(u)))}{\|u\|}\to 0$ as $u\to 0$ in $W$. 
Thus the map $\varphi$ provides a manifold structure of $F^{-1}(0)$ at $0_X$ whose tangent space is ${\rm Ker}(D_{1})$, 
and we finish the proof. 
\end{proof}

\begin{rem}
{\rm 
Under the assumption that $P_2$ is elliptic and ${\rm Im}(F)\subset {\rm Ker}(D_2)$, 
we take the Taylor expansion of $F(v)$ as in (\ref{s2.1e4}), and define 
the map $\Theta : V_1 \to V_1$ by 
\[
\Theta(v)=v+D_1^*\circ G(R(v))
\]
 for $v\in V_1$. 
Then, the condition $F(v)=0$ is equivalent to $D_1\circ \Theta (v)=0$ and $H_2(R(v))=0$ where $H_2(R(v))$ means the harmonic part of $R(v)$ with respect to $P_2$. 
The map $\Theta$ is so called {\it Kuranishi map}, and we can consider ${\rm Im}(F)\subset {\rm Im}(D_1)$ as the integrability condition 
since ${\rm Im}(F)\subset {\rm Im}(D_1)$ is equivalent to $H_2(R(v))=0$. 
}
\end {rem}

\subsection{Examples of $\Phi$-deformations}
In this section, we see some examples of $\Phi$-deformations and the moduli spaces. 
Let $X$ be a submanifold in $M$. 
We denote by $\wedge^k$ the vector space of all smooth differential $k$-forms on $X$. 

\subsubsection{Special Lagrangian submanifolds in Calabi-Yau manifolds}
We assume that $(M, \Omega, \omega)$ is a Calabi-Yau manifold 
of dimension $2n$ where $(\Omega, \omega)$ is a Calabi-Yau structure on $M$, 
that is, $\Omega$ is a holomorphic $n$-form and 
$\omega$ is a K\"{a}hler form satisfying the equation 
$\Omega\wedge \overline{\Omega}=c_{n}\omega^{n}$ for 
$c_{n}=\frac{1}{n!}(-1)^{\frac{(n-1)n}{2}}(\frac{2}{\sqrt{-1}})^{n}$. 
We call $X$ a {\it special Lagrangian submanifold} in $M$ 
if $X$ is the calibrated submanifold with respect to 
the real part $\Omega^{\rm Re}$ of $\Omega$. 
It is well known that $X$ is a special Lagrangian submanifold 
if and only if $X$ is an $n$-dimensional submanifold such that 
$\iota^{*}\Omega^{\rm Im}=\iota^{*}\omega=0$ where $\iota$ is the inclusion $\iota:X\hookrightarrow M$. 
Hence special Lagrangian deformations are $(\Omega^{\rm Im}, \omega)$-deformations and 
the moduli space $\mathcal{M}_{X}$ of special Lagrangian deformations is $\mathcal{M}_{X}(\Omega^{\rm Im}, \omega)$. 
The following result is provided by McLean : 
\begin{prop}(Theorem 3.6 \cite{Mc})\label{s2.2p1}
Let $X$ be a compact special Lagrangian submanifold in $M$. 
Then $\mathcal{M}_{X}$ is a smooth manifold of dimension $\dim(H^{1}(X))$. 
\end{prop}
\begin{proof}
We take the set $U$ as in $(\ref{s2.1e2})$ and define the map $F:U\to \wedge^{0}\oplus \wedge^{2}$ by 
\begin{equation*}\label{s2.2e2}
F(v)=(*\exp_{v}^{*}\Omega^{\rm Im},\exp_{v}^{*}\omega)
\end{equation*}
for $v\in U$ where $*$ means the Hodge star operator 
with respect to the metric $\iota^{*}g$ on $X$. 
Then we can regard $F^{-1}(0)$ as a set of special Lagrangian submanifolds 
in $M$ which is near $X$ in $C^1$ topology. 
To see the infinitesimal deformation of $X$, 
we consider the linearization $d_{0}F$ of $F$ at the origin $0\in U$. 
It follows that 
\begin{eqnarray*}\label{s2.2e3}
d_{0}F(v)&=&(*\iota^{*}L_{\tilde{v}}\Omega^{\rm Im}, \iota^{*}L_{\tilde{v}}\omega )\\
&=&(*d\iota^{*}(i_{\tilde{v}}\Omega^{\rm Im}), d\iota^{*}(i_{\tilde{v}}\omega))\\
&=&(d^{*}\iota^{*}(i_{\tilde{v}}\omega), d\iota^{*}(i_{\tilde{v}}\omega))
\end{eqnarray*}
for $v\in U$ where $\tilde{v}$ is an extension of $v$ to $M$.  
In the last equation, we use that 
$\iota^{*}(i_{\tilde{v}}\Omega^{\rm Im})=-*\iota^{*}(i_{\tilde{v}}\omega)$. 
We remark that the differential form $\iota^{*}(i_{\tilde{v}}\omega)=i_{v}\omega$ on $X$ 
is independent of any extension of $v$. 
Under the identification $\Gamma(NX)\simeq \wedge^{1}$ given by $v\mapsto i_{v}\omega$, 
we identify $d_{0}F$ with the map 
$D_{1}:\wedge^{1}\to\wedge^{0}\oplus \wedge^{2}$ defined by  
\begin{equation*}\label{s2.2e4}
D_{1}(\alpha)=(d^{*}\alpha, d\alpha)
\end{equation*}
for $\alpha\in\wedge^1$. Then it turns out that 
\begin{equation*}
{\rm Ker}(D_{1})=\{\alpha\in\wedge^{1} \mid d^{*}\alpha=d\alpha=0\}=\mathcal{H}^{1}(X)
\end{equation*}
where $\mathcal{H}^{k}(X)$ means the set of harmonic $k$-forms on $X$. 
We provide the differential complex 
\begin{equation*}\label{s2.2e5}
0\to \wedge^{1}\stackrel{D_{1}}{\longrightarrow}
\wedge^{0}\oplus \wedge^{2}\stackrel{D_{2}}{\longrightarrow}
\wedge^{3}\to 0 
\end{equation*}
where the operator $D_{2}$ is given by 
\begin{equation*}
D_{2}(f,\beta)=d\beta
\end{equation*}
for $(f,\beta)\in\wedge^{0}\oplus \wedge^{2}$. 
Since the dual operators $D_{1}^{*}$ and $D_{2}^{*}$ are given by 
$D_{1}^{*}(f,\beta)=df+d^{*}\beta$ and $D_{2}^{*}(\gamma)=(0,d^{*}\gamma)$,  
we have 
\begin{eqnarray*}
P_{1}(\alpha)&=&\Delta_{1}\alpha,\\
P_{2}(f,\beta)&=&(\Delta_{0}f, \Delta_{2}\beta)
\end{eqnarray*}
where $\Delta_{i}$ is the ordinary Laplace operator on 
$\wedge^{i}$ for $i=0,1$ and $2$. 
Thus $P_{1}$ and $P_{2}$ are elliptic. 
The image ${\rm Im}(F)$ of the map $F$ is included in $d^{*}\wedge^{1}\oplus d\wedge^{1}$ 
since $[f_{t}^{*}\Omega^{\rm Im}]=[\iota^{*}\Omega^{\rm Im}]=0$ and 
$[f_{t}^{*}\omega]=[\iota^{*}\omega]=0$. 
It is clear that $d^{*}\wedge^{1}\oplus d\wedge^{1}$ is perpendicular 
to $\Ker{P_{2}}=\mathcal{H}^{0}(X)\oplus \mathcal{H}^{2}(X)$ and 
${\rm Im}(D^{*}_{2})=\{0\}\oplus d^{*}\wedge^{3}$. 
Therefore ${\rm Im}(F)$ is also perpendicular to 
$\Ker{P_{2}}\oplus {\rm Im}(D^{*}_{2})$ and we obtain ${\rm Im}(F)\subset {\rm Im}(D_{1})$ by the Hodge decomposition 
$\wedge^{0}\oplus \wedge^{2}=\Ker{P_{2}}\oplus {\rm Im}(D_{1})\oplus {\rm Im}(D_{2}^{*})$. 
It follows from Proposition~\ref{s2.1p1} that $\mathcal{M}_{X}$ is smooth at $0_X$. 
We can show that $\mathcal{M}_{X}$ is smooth at any point by repeating the argument 
for each special Lagrangian submanifold. Hence we finish the proof. 
\end{proof}

\subsubsection{Coassociative submanifolds in $G_{2}$ manifolds}
We assume that $(M, g, \varphi)$ is a $G_{2}$ manifold 
where $\varphi$ is an associative $3$-form on $M$. 
We call $X$ a {\it coassociative submanifold} in $M$ 
if $X$ is calibrated submanifold with respect to 
the Hodge dual $*\varphi$ of $\varphi$ 
where $*$ is the Hodge star operator with respect to the metric $g$ on $M$. 
An $n$-dimensional submanifold $\iota:X\hookrightarrow M$ is a coassociative submanifold 
if and only if $\iota^{*}\varphi=0$. 
Hence coassociative deformations are $\varphi$-deformations and 
the moduli space $\mathcal{M}_{X}$ of coassociative deformations is $\mathcal{M}_{X}(\varphi)$. 
\begin{prop}(Theorem 4.5. \cite{Mc})\label{s2.3p1}
Let $X$ be a compact coassociative submanifold in $M$. 
Then $\mathcal{M}_{X}$ is a smooth manifold of dimension $\dim(H^{2}_{-}(X))$. 
\end{prop}
\begin{proof}
We take the set $U$ as in $(\ref{s2.1e2})$ and define the map $F:U\to \wedge^{3}$ by 
\begin{equation*}\label{s2.3e2}
F(v)=\exp_{v}^{*}\varphi 
\end{equation*}
for $v\in U$, then we can regard $F^{-1}(0)$ as a set of coassociative submanifolds 
in $M$ which is near to $X$. 
To see the first order deformation of $X$, 
we consider the linearization $d_{0}F$ of $F$ at the origin $0\in U$. 
It follows that for $v\in U$ 
\begin{equation*}\label{s2.3e3}
d_{0}F(v)=\iota^{*}L_{\tilde{v}}\varphi=d\iota^{*}(i_{\tilde{v}}\varphi)
\end{equation*}
where $\tilde{v}$ is an extension of $v$ to $M$. 
Let $\wedge^{2}_{-}$ be the set of anti-self dual $2$-forms on $X$. 
Under the identification $\Gamma(NX)\simeq \wedge^{2}_{-}$ given by $v\mapsto i_{v}\varphi$, 
we can consider $d_{0}F$ as the map 
$D_1:\wedge^{2}_{-}\to \wedge^{3}$ given by 
\begin{equation*}\label{s2.3e5}
D_1(\alpha)=d\alpha 
\end{equation*}
for $\alpha\in\wedge^2_{-}$. 
Then it turns out that 
\begin{equation*}
{\rm Ker}(D_{1})=\{\alpha\in\wedge^{2}_{-} \mid d\alpha=0\}=\mathcal{H}^{2}_{-}(X)
\end{equation*}
where $\mathcal{H}^{2}_{-}(X)$ is the set of harmonic anti-self dual $2$-forms on $X$. 
Now we provide a complex as follows 
\begin{equation*}\label{s2.3e6}
0\to \wedge^{2}_{-}\stackrel{D_{1}}{\longrightarrow}
\wedge^{3}\stackrel{D_{2}}{\longrightarrow}
\wedge^{4}\to 0 \tag{$\sharp$}
\end{equation*}
where the operator $D_{2}$ is given by 
\begin{equation*}
D_{2}(\beta)=d\beta
\end{equation*}
for $\beta\in\wedge^{3}$. It is easy to see that  
\begin{eqnarray*}
P_{1}(\alpha)&=&d^{*}d\alpha,\\
P_{2}(\beta)&=&(d(d^{*})_{-}+d^{*}d)\beta
\end{eqnarray*}
where $(d^{*})_{-}$ is the composition $p_{-}\circ d^{*}$ of $d^{*}$ and 
the projection $p:\wedge^{2}\to\wedge^{2}_{-}$.  
The complex (\ref{s2.3e6}) is isomorphic to the elliptic complex 
\begin{equation*}\label{s2.3e8}
0\to \wedge^{2}_{-}\stackrel{d^{*}}{\longrightarrow}
\wedge^{1}\stackrel{d^{*}}{\longrightarrow}
\wedge^{0}\to 0  
\end{equation*}
by the Hodge star operator with respect to the metric $\iota^*g$ on $X$. 
Hence the operators $P_{1}$ and $P_{2}$ are elliptic. 
It follows that the image ${\rm Im}(F)$ is included in ${\rm Im}(D_1)=d\wedge^{2}_{-}$ 
from  $f_{t}^{*}\varphi\in d\wedge^{2}$ and $d\wedge^{2}=d\wedge^{2}_{-}$. 
Proposition~\ref{s2.1p1} implies that $\mathcal{M}_{X}$ is smooth at $0_X$. 
We can show that $\mathcal{M}_{X}$ is smooth at any element by repeating the argument 
for each coassociative submanifold. 
Hence we finish the proof. 
\end{proof}

\subsubsection{Special Legendrian submanifolds in contact Calabi-Yau manifolds}
Let $M$ be a $(2n+1)$-dimensional manifold. 
A pair $(\psi, \eta)$ of two differential forms 
is called a {\it contact Calabi-Yau structure} on $M$ 
if $\eta$ is a contact $1$-form and $\psi$ is a $d$-closed complex valued $n$-form on $M$ 
such that $(\psi, \frac{1}{2}d\eta)$ is an almost transverse Calabi-Yau structure 
with respect to the Reeb foliation (we refer to Section 3.2 for the definition of almost transverse Calabi-Yau structures). 
On a contact Calabi-Yau manifold $(M, \psi, \eta)$, 
a calibrated submanifold $X$ with respect to the calibration $\psi^{\rm Re}$ 
is called a {\it special Legendrian submanifold}. 
Then $X$ is a special Legendrian submanifold if and only if 
$\iota^{*}\psi^{\rm Im}=\iota^{*}\eta=0$. 
Hence special Legendrian deformations are $(\psi^{\rm Im}, \eta)$-deformations and 
the moduli space $\mathcal{M}_{X}$ of special Legendrian deformations is $\mathcal{M}_{X}(\psi^{\rm Im}, \eta)$. 
Tomassini and Vezzoni showed the following result : 
\begin{prop}(Theorem 4.5. \cite{TV})\label{s2.4p1}
Let $X$ be a compact special Legendrian submanifold in $M$. 
Then $\mathcal{M}_{X}$ is a smooth manifold of dimension $\dim(H^{0}(X))$. 
\end{prop}
\begin{proof}
We remark the moduli space $\mathcal{M}_{X}$ is $\mathcal{M}_{X}(\psi^{\rm Im}, \eta, \frac{1}{2}d\eta)$ 
since $\iota^{*}\eta=0$ is equal to $\iota^{*}\eta=\iota^{*}(\frac{1}{2}d\eta)=0$. 
We take the set $U$ as in $(\ref{s2.1e2})$ and define the map $F:U\to \wedge^{0}\oplus \wedge^{1}\oplus \wedge^{2}$ by 
\begin{equation*}\label{s2.4e1}
F(v)=(*\exp_{v}^{*}\psi^{\rm Im},\exp_{v}^{*}\eta, \exp_{v}^{*}\frac{1}{2}d\eta)
\end{equation*}
for $v\in U$, then we can regard $F^{-1}(0)$ as a set of special Legendrian submanifolds 
in $M$ which is near $X$. 
It follows that for $v\in U$ 
\begin{eqnarray*}\label{s2.4e2}
d_{0}F(v)&=&(*\iota^{*}L_{\tilde{v}}\psi^{\rm Im}, 
\iota^{*}L_{\tilde{v}}\eta, \frac{1}{2}\iota^{*}L_{\tilde{v}}d\eta)\\
&=&(*d\iota^{*}(i_{\tilde{v}}\psi^{\rm Im}), 
d\iota^{*}(i_{\tilde{v}}\eta+i_{\tilde{v}}d\eta), \frac{1}{2}d\iota^{*}(i_{\tilde{v}}d\eta))\\
&=&(\frac{1}{2}d^{*}\iota^{*}(i_{\tilde{v}}d\eta), 
d\iota^{*}(i_{\tilde{v}}\eta)+\iota^{*}(i_{\tilde{v}}d\eta), \frac{1}{2}d\iota^{*}(i_{\tilde{v}}d\eta))
\end{eqnarray*}
where $\tilde{v}$ is an extension of $v$ to $M$. 
In the last equation, we use that 
$\iota^{*}(i_{\tilde{v}}\psi^{\rm Im})=-\frac{1}{2}*\iota^{*}(i_{\tilde{v}}d\eta)$. 
There exists the identification $\Gamma(NX)\simeq \wedge^{0}\oplus\wedge^{1}$ 
given by $v\mapsto (i_{v}\eta, \frac{1}{2}i_{v}d\eta)$. 
Under the above identification, 
we can consider $d_{0}F$ as the map 
$D_1:\wedge^{0}\oplus\wedge^{1}
\to\wedge^{0}\oplus \wedge^{1}\oplus \wedge^{2}$ given by  
\begin{equation*}\label{s2.4e4}
D_1(f,\alpha)=(d^{*}\alpha, df+2\alpha, d\alpha)
\end{equation*}
for $(f,\alpha)\in\wedge^{0}\oplus\wedge^{1}$. Then it turns out that 
\begin{eqnarray*}
{\rm Ker}(D_{1})&=&\{(f,\alpha)\in\wedge^{0}\oplus\wedge^{1} 
\mid d^{*}\alpha=df+2\alpha=0\} \\
&=&\{(f,-\frac{1}{2}df)\in\wedge^{0}\oplus\wedge^{1} \mid \Delta_{0}f=0 \} \\
&=&\{(f,0)\in\wedge^{0}\oplus\wedge^{1} \mid f\in H^{0}(X) \}
\simeq H^{0}(X). 
\end{eqnarray*}
Now we provide a complex as follows 
\begin{equation*}\label{s2.4e5}
0\to \wedge^{0}\oplus\wedge^{1}\stackrel{D_{1}}{\longrightarrow}
\wedge^{0}\oplus \wedge^{1}\oplus \wedge^{2}\stackrel{D_{2}}{\longrightarrow}
\wedge^{2}\oplus \wedge^{3}\to 0 
\end{equation*}
where the operator $D_{2}$ is given by 
\begin{equation*}
D_{2}(f,\alpha,\beta)=(d\alpha-2\beta, d\beta)
\end{equation*}
for $(f, \alpha, \beta)\in\wedge^{0}\oplus\wedge^{1}\oplus \wedge^{2}$.
Since $D_{1}^{*}(f,\alpha,\beta)=(d^{*}\alpha, df+2\alpha+d^{*}\beta)$ 
and $D_{2}^{*}(\beta,\gamma)=(0,d^{*}\beta, -2\beta+d^{*}\gamma)$, we have 
\begin{eqnarray*}
P_{1}(f,\alpha)&=&(\Delta_{0}f+2d^{*}\alpha, (\Delta_{1}+4)\alpha+2df),\\
P_{2}(f,\alpha, \beta)&=&(\Delta_{0}f+2d^{*}\alpha, 
(\Delta_{1}+4)\alpha+2df, (\Delta_{2}+4)\beta).
\end{eqnarray*}
Hence $P_{1}$ and $P_{2}$ are elliptic. 
The image ${\rm Im}(F)$ of the map $F$ is included in 
\begin{equation*}
\{(d^{*}h,\alpha, \frac{1}{2}d\alpha)\in \wedge^{0}\oplus \wedge^{1}\oplus \wedge^{2} 
\mid h\in\wedge^{0}, \alpha\in \wedge^{1}\}
\end{equation*}
which is perpendicular to the kernel 
\begin{equation*}
\Ker{D_{1}^{*}}=\{(f,-\frac{1}{2}d^{*}\beta, \beta)\in \wedge^{0}\oplus 
\wedge^{1}\oplus \wedge^{2} \mid f\in H^{0}(X)\}
\end{equation*} 
of the operator $D_{1}^{*}$. 
It follows from $\Ker{D_{1}^{*}}=\Ker{P_{2}}\oplus {\rm Im}(D_{2}^{*})$ that 
${\rm Im}(F)\perp\Ker{P_{2}}\oplus {\rm Im}(D_{2}^{*})$. 
Hence we obtain ${\rm Im}(F)\subset {\rm Im}(D_{1})$ by the Hodge decomposition 
$\wedge^{0}\oplus \wedge^{2}=\Ker{P_{2}}\oplus {\rm Im}(D_{1})\oplus {\rm Im}(D_{2}^{*})$. 
It follows from Proposition~\ref{s2.1p1} that $\mathcal{M}_{X}$ is smooth at $0_X$. 
We can show that $\mathcal{M}_{X}$ is smooth by repeating the argument 
for any special Legendrian submanifold. Hence we finish the proof. 
\end{proof}

\subsubsection{Legendrian submanifolds in contact manifolds}
Let $(M,\eta)$ be a $(2n+1)$-dimensional contact manifold 
with a contact $1$-form $\eta$. 
A Legendrian submanifold is defined by a submanifold $\iota:X\hookrightarrow M$ such that $\iota^{*}\eta=0$. 
Hence the space $\mathcal{M}_{X}(\eta)$ is the moduli space $\mathcal{M}_{X}$ of Legendrian deformations of $X$. 
The following result is well known as a consequence of 
the Darboux-Weinstein's neighborhood theorem for Legendrian submanifolds in contact geometry. 
\begin{prop}\label{s2.2.4p1}
Let $X$ be a compact Legendrian submanifold in $M$. 
Then $\mathcal{M}_{X}^{s,\alpha}$ is a smooth manifold. 
The tangent space $T_{0}\mathcal{M}_{X}^{s,\alpha}$ 
is isomorphic to the graph 
$\{(f,df)\in C^{s,\alpha}(\wedge^{0}\oplus\wedge^{1}) \}$ of the exterior derivative $d$. 
\end{prop}
\begin{proof}
We remark that $\mathcal{M}_{X}=\mathcal{M}_{X}(\eta, \frac{1}{2}d\eta)$ 
since $\iota^{*}\eta=0$ is equal to $\iota^{*}\eta=\iota^{*}\frac{1}{2}d\eta=0$. 
We take the set $U$ as in $(\ref{s2.1e2})$ and define the map $F:U\to \wedge^{1}\oplus\wedge^{2}$ by 
\begin{equation*}
F(v)=(\exp_{v}^{*}\eta, \frac{1}{2}\exp_{v}^{*}d\eta)
\end{equation*}
for $v\in U$. 
It follows that 
\begin{equation*}
d_{0}F(v)=(d\iota^{*}(i_{\tilde{v}}\eta)+\iota^{*}(i_{\tilde{v}}d\eta),\ 
\frac{1}{2}d\iota^{*}(i_{\tilde{v}}d\eta))
\end{equation*}
for $v\in U$ where $\tilde{v}$ is an extension of $v$ to $M$.  
Under the identification $\Gamma(NX)\simeq \wedge^{0}\oplus\wedge^{1}$ 
given by $v\mapsto (i_{v}\eta, \frac{1}{2}i_{v}d\eta)$, we identify $d_{0}F$ with the map 
$D_1:\wedge^{0}\oplus\wedge^{1}\to\wedge^{1}\oplus \wedge^{2}$ defined by  
\begin{equation*}
D_1(f,\alpha)=(df+2\alpha, d\alpha) 
\end{equation*}
for $(f,\alpha)\in\wedge^0\oplus \wedge^1$. Then it turns out that 
\begin{eqnarray*}
{\rm Ker}(D_{1})&=&\{(f,\alpha)\in\wedge^{0}\oplus\wedge^{1} 
\mid df+2\alpha=0\} \\
&=&\{(f,-\frac{1}{2}df)\in\wedge^{0}\oplus\wedge^{1} \}.
\end{eqnarray*}
Now we provide a complex as follows 
\begin{equation*}\label{s3.1e4}
0\to \wedge^{0}\oplus\wedge^{1}\stackrel{D_{1}}{\longrightarrow}
\wedge^{1}\oplus \wedge^{2}\stackrel{D_{2}}{\longrightarrow} 
\wedge^{2}\oplus\wedge^{3}\to 0 
\end{equation*}
where the operator $D_{2}$ is given by 
\begin{equation*}
D_{2}(\alpha,\beta)=(d\alpha-2\beta, d\beta) 
\end{equation*}
for $(\alpha,\beta)\in \wedge^{1}\oplus \wedge^{2}$. 
It is easy to see that  
\begin{eqnarray*}
P_{1}(f,\alpha)&=&(\Delta_{0}f+2d^{*}\alpha, (dd^{*}+2)\alpha+df),\\
P_{2}(\alpha,\beta)&=&((\Delta_{0}+4)\alpha, (\Delta_{2}+4)\beta). 
\end{eqnarray*}
Hence $P_{2}$ is the elliptic operator with $\Ker{P_{2}}=\{0\}\oplus \{0\}$. 
The space ${\rm Im}(F)$ is perpendicular to ${\rm Im}(D^{*}_{2})$ since ${\rm Im}(F)\subset \Ker{(D_{2})}$. 
Hence we obtain ${\rm Im}(F)\subset {\rm Im}(D_{1})$ by the Hodge decomposition 
$\wedge^{0}\oplus \wedge^{1}={\rm Im}(D_{1})\oplus {\rm Im}(D_{2}^{*})$. 
Proposition~\ref{s2.1p1} implies that $\mathcal{M}_{X}^{s,\alpha}$ is smooth at $0_X$ 
with the tangent space ${\rm Ker}(D_{1}^{s,\alpha})$. 
We can show that $\mathcal{M}_{X}^{s,\alpha}$ is smooth by repeating the argument 
for any Legendrian submanifold. Hence we finish the proof. 
\end{proof}

\section{Sasaki-Einstein manifolds}
In this section, we assume that $(M,g)$ is a smooth Riemannian manifold 
of dimension $2n+1$. 

\subsection{Transverse differential forms} 
Let $\mathcal{F}$ be a foliation on $M$ of codimension $2n$ and 
$F$ the vector bundle induced by the foliation $\mathcal{F}$. 
A differential form $\varphi$ on $M$ is called {\it transverse} if 
\begin{equation*}\label{s3.1e1}
i_{v}\varphi=0
\end{equation*}
for any $v\in\Gamma(F)$. 
We denote by $\wedge_{T}^{k}$ the vector space of transverse differential $k$-forms 
on the foliated manifold $(M,\mathcal{F})$. 
A transverse $k$-form can be considered as the section of $\wedge^{k}Q^{*}$ 
where $Q$ is the quotient bundle $TM/F$. 
A differential form $\varphi$ on $M$ is called {\it basic} if 
\begin{equation*}\label{s3.1e2}
i_{v}\varphi=0,\quad  L_{v}\varphi=0
\end{equation*}
for any $v\in\Gamma(F)$. 
Let $\wedge_{B}^{k}$ be the vector space of basic differential $k$-forms on $(M,\mathcal{F})$. 
It is easy to see that for a basic form $\varphi$ the derivative $d\varphi$ is also basic. 
Thus the exterior derivative $d$ induces 
the operator 
\begin{equation*}\label{s3.1e3}
d_{B}=d|_{\wedge_{B}^{k}}:\wedge_{B}^{k}\to\wedge_{B}^{k+1}
\end{equation*}
by the restriction. 
The corresponding complex $(\wedge_{B}^{*},d_{B})$ associates 
the cohomology group $H_{B}^{*}(M)$ which is called the {\it basic de Rham cohomology group}. 
In general, the derivative $d\varphi$ of a transverse form $\varphi$ is not necessarily transverse. 
In fact, a transverse form $\varphi$ is basic if $d\varphi$ is transverse. 
On the space $\wedge_{T}^{k}$, there exists an orthogonal decomposition 
$d\wedge^{k}_{T}=\wedge^{k+1}_{T}\oplus \wedge^{k}_{T}\wedge F^{*}$ with respect to the metric $g$. 
Let $\pi_{T}$ denote the first projection from $\wedge^{k+1}_{T}\oplus \wedge^{k}_{T}\wedge F^{*}$ 
to $\wedge^{k+1}_{T}$. 
We define a map 
\begin{equation*}\label{s3.1e4}
d_{T}:\wedge^{k}_{T}\to \wedge^{k+1}_{T}
\end{equation*} 
by the composition $\pi_{T}\circ d|_{\wedge^{k}_{T}}$ of $\pi_{T}$ and 
the restriction $d|_{\wedge^{k}_{T}}$ of $d$ to $\wedge^{k}_{T}$. 
Then $d_{T}\varphi=d_{B}\varphi$ for a basic form $\varphi$. 

If there exists a complex structure $J$ of $Q$, 
then we have a decomposition 
\begin{equation*}
\wedge_{T}^{k}=\oplus_{p+q=k}\wedge_{T}^{p,q}
\end{equation*}
where $\wedge_{T}^{p,q}$ is the set of transverse $(p,q)$-forms on $(M,\mathcal{F})$. 
Moreover, if $J$ is a transverse complex structure  on $(M,\mathcal{F})$ 
(see the next subsection for the definition), 
then it gives rise to a decomposition 
$\wedge^{k}_{B}\otimes\mathbb{C}=\oplus_{p+q=k}\wedge^{p,q}_{B}$ 
and operators 
\begin{eqnarray*}
\partial_{B}:\wedge_{B}^{p,q}\to\wedge_{B}^{p+1,q}, \\ 
\overline{\partial}_{B}:\wedge_{B}^{p,q}\to\wedge_{B}^{p,q+1}
\end{eqnarray*}
in the same manner as complex geometry. 
We denote by $H_{B}^{p,*}(M)$ the cohomology of 
the complex $(\wedge_{B}^{p,*},\overline{\partial}_{B})$ 
which is called the {\it basic Dolbeault cohomology group}. 
On the space $\wedge_{T}^{p,q}$, we have an orthogonal decomposition 
$d\wedge^{p,q}_{T}=\wedge^{p+1,q}_{T}\oplus \wedge^{p,q+1}_{T}\oplus \wedge^{p,q}_{T}\wedge F^{*}$ 
since $\wedge^{p,q}_{T}$ is locally generated by basic forms. 
We denote by $\pi_{T}^{1,0}$ the first projection from 
$\wedge^{p+1,q}_{T}\oplus \wedge^{p,q+1}_{T}\oplus \wedge^{p,q}_{T}\wedge F^{*}$ 
to $\wedge^{p+1,q}_{T}$. Then we define a map 
\begin{equation*}\label{s3.1e6}
\partial_{T}:\wedge^{p,q}_{T}\to \wedge^{p+1,q}_{T}
\end{equation*} 
by the composition $\pi_{T}^{1,0}\circ d|_{\wedge^{p,q}_{T}}$. 
In the same manner, we can define $\overline{\partial}_{T}:\wedge^{p,q}_{T}\to \wedge^{p,q+1}_{T}$. 
Then $d_{T}=\partial_{T}+\overline{\partial}_{T}$ and 
$\partial_{T}\varphi=\partial_{B}\varphi$ for a basic form $\varphi$.

\subsection{Transverse complex structures}
Let $\mathcal{F}$ be a foliation of codimension $2n$ on $M$. 
Then there exists a system $\{U_{i},f_{i},\gamma_{ij}\}$ 
consisting of an open covering $\{U_{i}\}_{i\in \Lambda}$ of $M$, 
submersions $f_{i}:U_{i}\to \mathbb{C}^{n}$ 
and diffeomorphisms $\gamma_{ij}:f_{i}(U_{i}\cap U_{j})\to f_{j}(U_{i}\cap U_{j})$ 
for $U_{i}\cap U_{j}\neq \phi$ satisfying $f_{j}=\gamma_{ij}\circ f_{i}$ 
such that any leaf of $\mathcal{F}$ is given by each fiber of $f_{i}$. 
We denote by $M^{T}$ the transverse manifold $\sqcup_{i}f_{i}(U_{i})$. 
The foliation $\mathcal{F}$ is a \textit{transverse holomorphic foliation} 
(resp. a \textit{transverse K\"{a}hler foliation}) if 
there exist a system $\{U_{i},f_{i},\gamma_{ij}\}$ and a complex structure $J_{i}$ (resp. K\"{a}hler structure $(g_{i}, J_{i})$) on each $f_{i}(U_{i})$ such that 
$\gamma_{ij}$ is bi-holomorphic (resp. preserving the K\"{a}hler structure). 
Thus any transverse holomorphic foliation $\mathcal{F}$ induces a complex structure $J^{T}=\{J_{i}\}_{i\in \Lambda}$ on $M^{T}$. 

In order to characterize transverse structures on $(M,\mathcal{F})$, 
we consider the quotient bundle $Q=TM/F$ 
where $F$ is the line bundle associated by the foliation $\mathcal{F}$. 
We define an action of $\Gamma(F)$ 
to any section $u$ of $Q$ as follows : 
\begin{equation*}\label{s3.2e1}
L_{v}u=\pi(L_{v}\widetilde{u})
\end{equation*}
for $v\in\Gamma(F)$ where $\pi$ is the quotient map $TM\to Q$ 
and $\widetilde{u}\in\Gamma(TM)$ is a lift of $u$ by $\pi$. 
We remark that $L_{v}u$ is independent of the choice of the lift $\widetilde{u}\in\Gamma(TM)$ of $u$. 
The section $u$ of $\Gamma(Q)$ is called {\it basic} if $L_{v}u=0$ for any  $v\in\Gamma(F)$. 
We denote by $\Gamma_{B}(Q)$ the set of basic sections of $\Gamma(Q)$. 
The vector field $\widetilde{u}$ on $M$ is called  {\it foliated} if 
$L_{v}u\in\Gamma(F)$ for any  $v\in\Gamma(F)$. Let $\Gamma_{F}(TM)$ denote the set of foliated vector fields on $M$. 
Then there exists the exact sequence 
\begin{equation*}\label{s3.2e2}
0\to \Gamma(F)\stackrel{\iota}{\longrightarrow}\Gamma_{F}(TM)\stackrel{\pi}{\longrightarrow} \Gamma_{B}(Q)\to 0
\end{equation*}
where $\iota$ is the natural inclusion. 
In fact, any basic section $u$ of $\Gamma(Q)$ has a lift $\widetilde{u}$ 
by $\pi$ which is a foliated vector field. 
We can also define an action of $\Gamma(F)$ 
to any section $J\in \Gamma({\rm End}(Q))$ as follows : 
\begin{equation*}\label{s3.2e3}
(L_{v}J)(u)=L_{v}(J(u))-J(L_{v}u)
\end{equation*}
for $v\in\Gamma(F)$ and $u\in\Gamma(Q)$. 
A section $J\in \Gamma({\rm End}(Q))$ is called {\it basic} 
if $L_{v}J=0$ for any $v\in\Gamma(F)$. 
If $J$ is a complex structure of $Q$, i.e. $J^{2}=-{\rm id}_{Q}$, 
and basic as a section of ${\rm End}(Q)$, 
then a tensor $N_{J}\in \Gamma(\otimes^{2}Q^{*}\otimes Q)$ can be defined by 
\begin{equation*}\label{s3.2e4}
N_{J}(u,w)=[\widetilde{Ju},\widetilde{Jw}]_{Q}-[\widetilde{u},\widetilde{w}]_{Q}
-J[\widetilde{u},\widetilde{Jw}]_{Q}-J[\widetilde{Ju},\widetilde{w}]_{Q}
\end{equation*}
for $u,w\in\Gamma(Q)$, 
where $[u,w]_{Q}$ denotes $\pi[\widetilde{u},\widetilde{w}]$ 
for each lift $\widetilde{u}$ and $\widetilde{w}$. 
We call that $J$ is {\it integrable} if $N_{J}=0$. 
\begin{defi}\label{s3.2d1}
{\rm 
A complex structure $J$ of $Q$ is 
a \textit{transverse complex structure} on $(M,\mathcal{F})$ 
if $J$ is basic and integrable. 
}
\end{defi}
Any transverse holomorphic foliation $\mathcal{F}$ induces a transverse complex structure. 
\begin{lem}
If $\mathcal{F}$ is a transverse holomorphic foliation, 
then there exists a transverse complex structure $J_{\mathcal{F}}$ on $(M,\mathcal{F})$. 
\end{lem}
\begin{proof}
We can extend a vector field $u^{T}$ on $M^{T}$ to a foliated vector field $u_{i}$ 
on each $U_{i}$ such that $df_{i}(u_{i})=u^{T}$ since $U_{i}$ is diffeomorphic to $f_{i}(U_{i})\times V_{i}$ 
where $V_{i}$ is an open subset of $\mathbb{R}$. 
Then $\{\pi(u_{i})\}_{i}$ defines the section of $Q$ and it is basic. 
Hence we obtain the map $\sigma:\Gamma(TM^{T})\to\Gamma_{B}(Q)$ by $\sigma(u^{T})=\{\pi(u_{i})\}_{i}$. 
On the other hand, for any $u\in\Gamma_{B}(Q)$ 
the family $\{df_{i}(u)\}_{i}$ defines the vector field over $M^{T}$ 
where $df_{i}(u)$ is defined by $df_{i}(\widetilde{u})$ for a lift $\widetilde{u}$ of $u$. 
We define the map $df:\Gamma_{B}(Q)\to\Gamma(TM^{T})$ as $df(u)=\{df_{i}(u)\}_{i}$. 
Then $df$ is the inverse map of $\sigma$ and hence $\Gamma_{B}(Q)$ is isomorphic to $\Gamma(TM^{T})$. 
For the complex structure $J^{T}$ on $M^{T}$, 
we can define a section $J_{\mathcal{F}}$ of ${\rm End}(Q)$ as 
\begin{equation*}\label{s3.2e5}
J_{\mathcal{F}}(u)=\sigma (J^{T}(df(u)))
\end{equation*}
for $u\in\Gamma_{B}(Q)$. This section $J_{\mathcal{F}}$ is well-defined 
since any section of $Q$ is locally generated by basic sections. 
Then $J_{\mathcal{F}}$ is a complex structure of $Q$ and basic 
since $J_{\mathcal{F}}(u)=\sigma (J^{T}u^{T})$ is basic for any $u\in \Gamma_{B}(Q)$. 
The tensor $N_{J_{\mathcal{F}}}$ satisfies that 
\begin{equation*}\label{s3.2e6}
N_{J_{\mathcal{F}}}(u,w)=\sigma(N_{J^{T}}(df(u),df(v)))=0
\end{equation*}
for $u,w\in\Gamma_{B}(Q)$. It implies that $N_{J_{\mathcal{F}}}=0$. 
Hence a transversely holomorphic foliation $\mathcal{F}$ 
induces the transverse complex structure $J_{\mathcal{F}}$ on $(M,\mathcal{F})$. 
\end{proof}

The following result is Newlander-Nirenberg's theorem for a transverse complex structure on a foliated manifold : 

\begin{prop}\label{s3.2p1}
Let $J$ be a complex structure of $Q$. 
Then the following conditions are equivalent. 
\begin{enumerate}
\item[(i)] $J$ is a transverse complex structure on $(M,\mathcal{F})$. 
\item[(ii)] $\mathcal{F}$ is a transversely holomorphic foliation with $J_{\mathcal{F}}=J$. 
\item[(iii)] $J$ is basic and satisfies $d(\wedge_{B}^{1,0})\subset \wedge_{B}^{2,0}\oplus \wedge_{B}^{1,1}$. 
\item[(iv)]  $J$ is basic and satisfies $d(\wedge_{B}^{0,1})\subset \wedge_{B}^{1,1}\oplus \wedge_{B}^{0,2}$. 
\end{enumerate}
\end{prop}
\begin{proof}
A basic complex structure $J$ of $Q$ corresponds to an almost complex structure on $M^{T}$, 
and so $\wedge_{B}^{p,q}$ is isomorphic to $\wedge^{p+q}_{M^{T}}$. 
Hence the conditions (ii), (iii) and (iv) are equivalent 
by Newlander-Nirenberg's theorem for the complex manifold $(M^{T},J^{T})$. 
We already checked that the condition (ii) implies (i). 
Hence it suffices to show that the condition (i) implies (ii). 
If $J$ is a transverse complex structure on $(M,\mathcal{F})$, 
then the section $J(\sigma (u^{T}))$ of $Q$ is basic for any $u^{T}\in\Gamma(TM^{T})$. 
Hence we can define the section $J^{T}$ of ${\rm End}(TM^{T})$ as 
\begin{equation*}\label{s3.2e7}
J^{T}(u^{T})=df(J(\sigma (u^{T})))
\end{equation*}
for $u^{T}\in\Gamma(TM^{T})$. 
The section $J^{T}$ is an almost complex structure on $M^{T}$. 
Let $N_{J^{T}}$ be the Nijenhuis tensor of $J^{T}$. 
Then we obtain 
\begin{equation*}\label{s3.2e8}
N_{J^{T}}(u^{T},v^{T})=df(N_{J}(\sigma(u^{T}),\sigma(v^{T})))=0
\end{equation*}
for any $u^{T},v^{T}\in\Gamma(TM^{T})$. 
Hence $J^{T}$ is the complex structure on $M^{T}$. 
Then the foliation $\mathcal{F}$ is transversely holomorphic and $J_{\mathcal{F}}=J$ 
by the definition of $J_{\mathcal{F}}$. 
It completes the proof. 
\end{proof}

\begin{defi}\label{s3.2d2}
{\rm 
A nowhere vanishing complex $n$-form $\psi\in \Gamma (\wedge^{n}T^*M \otimes \mathbb{C})$ 
is called an \textit{almost transverse ${\rm SL}_{n}(\mathbb{C})$ structure} on $(M,\mathcal{F})$ 
if $\psi$ is transverse and  
\begin{equation*}\label{s3.2e9}
Q\otimes\mathbb{C}=\Ker_{\mathbb{C}}\psi/F\oplus \overline{\Ker_{\mathbb{C}}\psi/F}
\end{equation*} 
where $\Ker_{\mathbb{C}}\psi/F$ is the set $\{u\in Q\otimes\mathbb{C} \ | \ i_{u}\psi=0 \}$. 
}
\end{defi}
An almost transverse ${\rm SL}_{n}(\mathbb{C})$ structure $\psi$ 
induces a complex structure $J_{\psi}$ of $Q$ as follows  
\begin{equation*}\label{s3.2e10}
J_{\psi}(u)=
\left\{
  \begin{array}{cc}
      -\sqrt{-1}u &  {\rm for}\ u\in \Ker_{\mathbb{C}}\psi/F, \\
       \sqrt{-1}u &  {\rm for}\ u\in \overline{\Ker_{\mathbb{C}}\psi/F}. \\
  \end{array}
\right.
\end{equation*} 
Then $Q^{0,1}=\Ker_{\mathbb{C}}\psi/F$ and $Q^{1,0}=\overline{\Ker_{\mathbb{C}}\psi/F}$. 
Therefore $\psi$ is a transverse $(n,0)$-form on $(M,\mathcal{F})$. 
Note that the section $J_{\psi}\in {\rm End}(Q)$ is not necessarily 
integrable. 
However we have 

\begin{prop}\label{s3.2p2}
If $\psi$ satisfies $d\psi=A\wedge \psi$ for a complex valued $1$-form $A$, 
then $J_{\psi}$ is basic and integrable. 
\end{prop}
\begin{proof}
At first, we show that $J_{\psi}$ is basic, that is, $L_{v}J_{\psi}=0$ for $v\in\Gamma(F)$. 
Let $u$ be a section of $\Ker_{\mathbb{C}}\psi/F$. 
Then $L_{v}u$ is also the section of $\Ker_{\mathbb{C}}\psi/F$ 
since 
\begin{eqnarray*}
\psi(L_{v}u)&=&L_{v}(\psi(u))-(L_{v}\psi)(u)\\
&=&-(i_{v}d\psi)(u) \\
&=&-i_{v}(A\wedge\psi)(u) \\
&=&-(i_{v}A)\psi(u)+(A\wedge i_{v}\psi)(u) \\
&=&0. 
\end{eqnarray*}
It yields that 
\begin{equation*}
(L_{v}J_{\psi})u=L_{v}(J_{\psi}(u))-J_{\psi}(L_{v}u)=L_{v}(-\sqrt{-1}u)-(-\sqrt{-1}L_{v}u)=0
\end{equation*}
for any $u\in \Gamma(\Ker_{\mathbb{C}}\psi/F)$. In the same manner, 
we can prove $(L_{v}J_{\psi})u=0$ for any $u\in \Gamma(\overline{\Ker_{\mathbb{C}}\psi/F})$. 
Thus $L_{v}J_{\psi}=0$ for any $v\in\Gamma(F)$, and hence $J_{\psi}$ is basic.  

Secondary, we see that $J_{\psi}$ is integrable. 
It suffices to show that $d\wedge^{1,0}_{B}\subset \wedge^{1,1}_{B}\oplus\wedge^{2,0}_{B}$ 
by Proposition~\ref{s3.2p1}. 
Let $\alpha$ be an element of $\wedge^{1,0}_{B}$. 
Then $\alpha\wedge\psi=0$ since $\psi$ is the transverse $(n,0)$-form. 
Then the derivative $d\alpha$ does not have the basic $(0,2)$-part. 
In fact, $d\alpha\wedge\psi=d(\alpha\wedge\psi)+\alpha\wedge d\psi=\alpha\wedge A\wedge\psi=0$ 
since $\alpha\wedge\psi=0$. 
Therefore $d\wedge^{1,0}_{B}\subset \wedge^{1,1}_{B}\oplus\wedge^{2,0}_{B}$. 
Hence $J_{\psi}$ is integrable, and we finish the proof. 
\end{proof} 

If a transverse $2$-form $\omega^{T}$ on $M$ satisfies $(\omega^{T})^{n}\neq 0$, 
then we call the form $\omega^{T}$ 
an \textit{almost transverse symplectic structure} on $(M,\mathcal{F})$. 
We can consider the form $\omega^{T}$ as a tensor of $\wedge^{2}Q^{*}$. 
\begin{defi}\label{s3.2d3}
{\rm 
Let $\psi$ be an almost transverse ${\rm SL}(n,\mathbb{C})$ structure and 
$\omega^{T}$ an almost transverse symplectic structure on $(M,\mathcal{F})$. 
The pair $(\psi, \omega^{T})$ is called an \textit{almost transverse Calabi-Yau structure} on $(M,\mathcal{F})$ 
if $\psi$ and $\omega^{T}$ satisfy the following equations 
\begin{eqnarray*}
&&\psi\wedge \omega^{T}=\overline{\psi}\wedge \omega^{T} =0,\\
&&\psi\wedge \overline{\psi}=c_{n}(\omega^{T})^{n},\\
&&g^{T}=\omega^{T}(\cdot,J_{\psi}\cdot)
\end{eqnarray*}
where $c_{n}=\frac{1}{n!}(-1)^{\frac{n(n-1)}{2}}(\frac{2}{\sqrt{-1}})^{n}$. 
}
\end{defi}

If an almost transverse symplectic structure $\omega^{T}$ is $d$-closed on $M$, 
then $\omega^{T}$ is called a \textit{transverse symplectic structure} on $(M,\mathcal{F})$. 
A pair $(\omega^{T},J)$ is called 
a \textit{transverse K\"{a}hler structure} on $(M,\mathcal{F})$ 
if $\omega^{T}(\cdot,J\cdot)$ is a positive definite symmetric $2$-tensor of $Q$, 
and then $\omega^{T}(J\cdot,J\cdot)=\omega^{T}(\cdot,\cdot)$ holds. 
Then we define a fiber metric $g^{T}$ of $Q$ by 
$g^{T}(\cdot,\cdot)=\omega^{T}(\cdot,J\cdot)$ and 
call it a {\it transverse K\"{a}hler metric} on $(M,\mathcal{F})$. 
In \cite{M}, we introduced a {\it transverse Calabi-Yau structure} on $(M,\mathcal{F})$ 
as an almost transverse Calabi-Yau structure $(\psi, \omega^{T})$ such that $\psi$ and $\omega^{T}$ are $d$-closed. 

\subsection{Sasaki structures}
We will give a brief review of 
some elementary results in Sasakian geometry. 
For much of this material, 
we refer to~\cite{BG} and \cite{S}. 

\begin{defi}\label{s3.3d1}
{\rm 
A Riemannian manifold $(M,g)$ is a {\it Sasaki manifold} 
if the metric cone $(C(M),\overline{g})=(\mathbb{R}_{>0}\times M, dr^{2}+r^{2}g)$ 
is a K\"{a}hler manifold with a complex structure $I$. 
}
\end{defi}
We identify the manifold $M$ with the hypersurface $\{r=1\}$ of $C(M)$. 
Let $I$ and $\omega$ denote the complex structure 
and the K\"{a}hler form on the K\"{a}hler manifold $(C(M),\overline{g})$, respectively. 
The vector field $r\frac{\partial}{\partial r}$ is called 
the {\it Euler vector field} on $C(M)$. 
We define a vector field $\xi$ and a $1$-form $\eta$ on $C(M)$ by 
\begin{equation*}\label{s3.3e1}
\xi=I(r\frac{\partial}{\partial r}),\quad \eta(X)=\frac{1}{r^{2}}\overline{g}(\xi, X)
\end{equation*}
for any vector field $X$ on $C(M)$. 
The vector field $\xi$ is a Killing vector field, i.e. $L_{\xi}\overline{g}=0$, and 
$\xi+\sqrt{-1}\,I\xi=\xi-\sqrt{-1}\,r\frac{\partial}{\partial r}$ is 
a holomorphic vector field on $C(M)$. 
It follows from $L_{\xi}\eta=IL_{r\frac{\partial}{\partial r}}\eta=0$ that 
\begin{equation}\label{s3.3e2}
\eta(\xi)=1,\quad i_{\xi}d\eta=0
\end{equation}
where $i_{\xi}$ means the interior product relative to $\xi$. 
The form $\eta$ is expressed as 
\begin{equation*}\label{s3.3e3}
\eta=d^{c}\log r=\sqrt{-1}\,(\overline{\partial}-\partial)\log r
\end{equation*} 
where $d^{c}$ is the composition $-I\circ d$ of the exterior differentiation $d$ and 
the action of the complex structure $-I$ on differential forms. 
We define an action $\lambda$ of $\mathbb{R}_{>0}$ on $C(M)$ by 
\begin{equation*}\label{s3.3e4}
\lambda_{a}(r,x)=(ar,x)
\end{equation*}
for $a\in\mathbb{R}_{>0}$ and $(r,x)\in \mathbb{R}_{>0}\times M=C(M)$. 
If we put $a=e^{t}$ for $t\in\mathbb{R}$, 
then it follows from $L_{r\frac{\partial}{\partial r}}=\frac{d}{dt}\lambda_{e^{t}}^{*}|_{t=0}$ 
that $\{\lambda_{e^{t}}\}_{t\in\mathbb{R}}$ is one parameter group of transformations such that 
$r\frac{\partial}{\partial r}$ is the infinitesimal transformation. 
Then the K\"{a}hler form $\omega$ satisfies $\lambda_{a}^{*}\omega=a^{2}\omega$ for $a\in\mathbb{R}_{>0}$ and 
\begin{equation*}\label{s3.3e5}
L_{r\frac{\partial}{\partial r}}\omega=2\omega.
\end{equation*}
It implies that  
\begin{equation*}\label{s3.3e6}
\omega=\frac{1}{2}d(r^{2}\eta)=\frac{\sqrt{-1}}{2}\partial\overline{\partial}r^{2}.
\end{equation*}
Hence $\frac{1}{2}r^{2}$ is a K\"{a}hler potential on $C(M)$. 

The $1$-form $\eta$ induces the restriction $\eta|_{M}$ on $M\subset C(M)$. 
Since $L_{r\frac{\partial}{\partial r}}\eta=0$, 
the form $\eta$ is the extension of $\eta|_{M}$ to $C(M)$. 
The vector field $\xi$ is tangent to the hypersurface $\{r=c\}$ 
for each positive constant $c$. 
In particular, $\xi$ is considered as the vector field on $M$ 
and satisfies $g(\xi,\xi)=1$ and $L_{\xi}g=0$. 
Hence we shall not distinguish between $(\eta,\xi)$ on $C(M)$ and 
the restriction $(\eta|_{M},\xi|_{M})$ on $M$. 
Then the form $\eta$ is a contact $1$-form on $M$ : 
\begin{equation*}\label{s3.3e7}
\eta\wedge (d\eta)^{n}\neq 0
\end{equation*}
since $\omega$ is non-degenerate. 
The equation (\ref{s3.3e2}) implies that 
\begin{equation}\label{s3.3e8}
\eta(\xi)=1,\quad i_{\xi}d\eta=0
\end{equation}
on $M$. 
For a contact form $\eta$, a vector field $\xi$ on $M$ 
satisfying the equation (\ref{s3.3e8})  
is unique, and called the {\it Reeb vector field}. 
We define the contact subbundle $D\subset TM$ by $D=\ker\eta$. 
Then the tangent bundle $TM$ has the orthogonal decomposition 
\begin{equation*}\label{s3.3e9}
TM=D\oplus \langle \xi\rangle
\end{equation*}
where $\langle \xi\rangle$ is the line bundle generated by $\xi$. 
We define a section $\Psi$ of ${\rm End}(TM)$ by setting 
$\Psi |_{D}=I|_{D}$ and $\Psi |_{\langle \xi\rangle}=0$. 
One can see that 
\begin{eqnarray}
\Psi^{2}=-{\rm id}+\xi\otimes\eta, \label{s3.3e10}\\ 
d\eta(\Psi X,\Psi Y)=d\eta(X,Y) \label{s3.3e11}
\end{eqnarray}
for any $X,Y \in TM$. 
Then the Riemannian metric $g$ satisfies 
\begin{equation}\label{s3.3e12}
g(X,Y)=\frac{1}{2}d\eta(X, \Psi Y)+\eta(X)\eta(Y)
\end{equation}
for any $X,Y \in TM$. 
We denote by $\omega^{T}$ the $2$-form $\frac{1}{2}d\eta$ on $M$. 
Then $\omega^{T}$ is a symplectic structure on $D$ which is compatible with $\Psi$. 

We say a data $(\xi,\eta,\Psi,g)$ a \textit{contact metric structure} on $M$ 
if for a contact form $\eta$ and a Reeb vector field $\xi$, 
a section $\Psi$ of ${\rm End}(TM)$ and a Riemannian metric $g$ satisfy 
the equations (\ref{s3.3e10}), (\ref{s3.3e11}) and (\ref{s3.3e12}). 
Moreover, a contact metric structure $(\xi,\eta,\Psi,g)$ is called 
a \textit{K-contact structure} on $M$ if $\xi$ is a Killing vector field with respect to $g$. 
Then $\Psi$ defines an almost CR structure $(D,\Psi|_{D})$ on $M$. 
As we saw above, any Sasaki manifold $(M,g)$ has 
a K-contact structure $(\xi,\eta,\Psi,g)$ with 
the integrable CR structure $(D,\Psi|_{D}=I|_{D})$ on $M$. 
Conversely, if we have such a structure $(\xi,\eta,\Psi,g)$ on $M$, 
then $(\overline{g},\frac{1}{2}d(r^{2}\eta))$ is a K\"{a}hler structure 
on the cone $C(M)$, hence $(M,g)$ is a Sasaki manifold. 
We call a K-contact structure $(\xi,\eta,\Psi,g)$ with 
the integrable CR structure $(D,\Psi|_{D})$ a \textit{Sasaki structure} on $M$.

\subsection{The Reeb foliation}
Let $(\xi,\eta,\Psi,g)$ be a Sasaki structure on $M$. 
Then the Reeb vector field $\xi$ generates a foliation $\mathcal{F}_{\xi}$ of codimension $2n$ on $M$. 
The foliation $\mathcal{F}_{\xi}$ is called the {\it Reeb foliation}. 
We can consider $\Psi$ as a section of ${\rm End}(Q)$ since $\Psi|_{\langle \xi\rangle}=0$. 
Then $\Psi$ is a transverse complex structure on $(M,\mathcal{F}_{\xi})$ 
by the integrability of the CR structure $\Psi|_{D}$. 
Let $\omega^{T}$ be the $2$-form $\frac{1}{2}d\eta$ on $M$. 
Then the pair $(\Psi,\omega^{T})$ is a transverse K\"{a}hler structure 
with the transverse K\"{a}hler metric $g^{T}(\cdot,\cdot)=\omega^{T}(\cdot,\Psi\cdot)$ 
on $(M,\mathcal{F}_{\xi})$. 

We define ${\rm Ric}^{T}$ as the Ricci tensor of $g^{T}$ 
which is called the {\it transverse Ricci tensor}. 
The transverse Ricci form $\rho^{T}$ is defined 
by $\rho^{T}(\cdot,\cdot)={\rm Ric}^{T}(\cdot,\Psi\cdot)$. 
The form $\rho^{T}$ is a basic $d$-closed $(1,1)$-form on $(M,\mathcal{F}_{\xi})$ 
and defines a $(1,1)$-basic Dolbeault cohomology class $[\rho^{T}]\in H_{B}^{1,1}(M)$ 
as in the K\"{a}hler case. 
The basic class $[\frac{1}{2\pi}\rho^{T}]$ in $H_{B}^{1,1}(M)$ is called 
the {\it basic first Chern class} on $(M,\mathcal{F}_{\xi})$ 
and is denoted by $c_{1}^{B}(M)$ (for short, we write it $c_{1}^{B}$). 
We say the basic first Chern class is positive (resp. negative) 
if $c_{1}^{B}$ (resp. $-c_{1}^{B}$) is represented by a transverse K\"{a}hler form. 
This condition is expressed by $c_{1}^{B}>0$ (resp. $c_{1}^{B}<0$). 
We say that $(g^{T},\omega^{T})$ is 
a {\it transverse K\"{a}hler-Einstein structure} 
with Einstein constant $\kappa$ if 
$(g^{T},\omega^{T})$ is the transverse K\"{a}hler structure satisfying 
${\rm Ric}^{T}=\kappa g^{T}$ which is equivalent to $\rho^{T}=\kappa \omega^{T}$. 
If $M$ admits such a structure, 
then $2\pi c_{1}^{B}=\kappa[\omega^{T}]$, 
so the basic first Chern class has to be positive, zero or negative 
according to the sign of $\kappa$. 

On the cone $C(M)$, a foliation $\mathcal{F}_{\langle \xi,r\frac{\partial}{\partial r}\rangle}$ 
is induced by the vector bundle $\langle \xi,r\frac{\partial}{\partial r}\rangle$ 
generated by $\xi$ and $r\frac{\partial}{\partial r}$. 
Let $\widetilde{\phi}$ be a basic form  on 
$(C(M), \mathcal{F}_{\langle \xi, r\frac{\partial}{\partial r}\rangle})$. 
Then the restriction $\widetilde{\phi}|_{M}$ of $\widetilde{\phi}$ to $M$ is also basic on $(M,\mathcal{F}_{\xi})$. 
Conversely, for any basic form $\phi$ on $(M,\mathcal{F}_{\xi})$, 
the trivial extension $\widetilde{\phi}$ of $\phi$ to $C(M)=\mathbb{R}_{>0}\times M$ is 
a basic form on $(C(M),\mathcal{F}_{\langle \xi,r\frac{\partial}{\partial r}\rangle})$. 
In this paper, we identify a basic form $\phi$ on $(M,\mathcal{F}_{\xi})$
with the extension $\widetilde{\phi}$ on 
$(C(M), \mathcal{F}_{\langle \xi,r\frac{\partial}{\partial r}\rangle})$. 

\subsection{Sasaki-Einstein structures and almost transverse Calabi-Yau structures}
In this section, we assume that $M$ is compact. 
We provide the definition of Sasaki-Einstein manifolds. 
\begin{defi}\label{s3.5d1}
{\rm 
A Sasaki manifold $(M,g)$ is {\it Sasaki-Einstein} 
if the metric $g$ is Einstein. 
}
\end{defi}

Let $(\xi, \eta, \Psi, g)$ be a Sasaki structure on $M$. 
Then the Ricci tensor ${\rm Ric}$ of $g$ has following relations :
\begin{eqnarray*}
&&{\rm Ric}(u,\xi)=2n\eta(u),\ u\in TM \label{s2.5eq1}\\
&&{\rm Ric}(u,v)={\rm Ric}^{T}(u,v)-2g(u,v),\ u,v \in D \label{s2.5eq2}
\end{eqnarray*}
Thus the Einstein constant of a Sasaki-Einstein metric $g$ 
must be $2n$, that is, ${\rm Ric}=2ng$. 
It follows from the above equations that the Einstein condition ${\rm Ric}=2ng$ 
is equal to ${\rm Ric}^{T}=2(n+1)g^{T}$. 
Moreover, the cone metric $\overline{g}$ is Ricci-flat on $C(M)$ if and only if 
$g$ is Einstein with the Einstein constant $2n$ on $M$ (we refer to Lemma 11.1.5 in \cite{BG}). 
Hence we can characterize the Sasaki-Einstein condition as follows 
\begin{prop}\label{s3.5p1}
Let $(M,g)$ be a Sasaki manifold of dimension $2n+1$. 
Then the following conditions are equivalent each other. 
\begin{enumerate}
\item[(i)] $(M,g)$ is a Sasaki-Einstein manifold. 

\item[(ii)] $(C(M),\overline{g})$ is Ricci-flat, that is, ${\rm Ric}_{\overline{g}}=0$. 

\item[(iii)] $g^{T}$ is transverse K\"{a}hler-Einstein with ${\rm Ric}^{T}=2(n+1)g^{T}$. 
$\hfill\Box$
\end{enumerate}
\end{prop}
We remark that Sasaki-Einstein manifolds have finite fundamental groups by Mayer's theorem. 
From now on, we assume that $M$ is simply connected. 

\begin{defi}\label{s3.5d2}
{\rm 
A pair $(\Omega,\omega)\in\Gamma(\wedge^{n+1}T^*C(M) \otimes\mathbb{C}\oplus\wedge^{2}T^*C(M))$ is called 
a \textit{weighted Calabi-Yau structure} on $C(M)$ 
if $\Omega$ is a holomorphic section of $K_{C(M)}$ 
and $\omega$ is a K\"{a}hler form satisfying 
the equation 
\begin{equation}\label{s3.5e2}
\Omega\wedge \overline{\Omega}=c_{n+1}\omega^{n+1}
\end{equation}
and 
\begin{eqnarray*}
&&L_{r\frac{\partial}{\partial r}}\Omega=(n+1)\Omega, \label{s2.5eq8}\\
&&L_{r\frac{\partial}{\partial r}}\omega=2\omega. \label{s2.5eq9}
\end{eqnarray*} 
}
\end{defi}
If there exists a weighted Calabi-Yau structure $(\Omega,\omega)$ on $C(M)$, 
then it is unique up to change 
$\Omega\to e^{\sqrt{-1}\,\theta}\Omega$ of a phase $\theta\in \mathbb{R}$. 

\begin{lem}\label{s3.5l1}
A Riemannian metric $g$ on $M$ is Sasaki-Einstein 
if and only if there exists 
a weighted Calabi-Yau structure $(\Omega,\omega)$ on $C(M)$ 
such that $\overline{g}$ is the K\"{a}hler metric. $\hfill\Box$
\end{lem}
We characterize a Sasaki-Einstein manifold by 
a pair of two differential forms on the manifold. 
\begin{prop}\label{s3.5p2}
The Riemannian manifold $(M,g)$ is a Sasakian-Einstein manifold 
if and only if there exist a contact $1$-form $\eta$ and 
a complex valued $n$-form $\psi$ such that 
$(\psi, \frac{1}{2}d\eta)$ is an almost transverse Calabi-Yau structure on $(M, \mathcal{F})$, 
where $\mathcal{F}$ is the Reeb foliation induced by $\eta$, with 
\begin{equation*}\label{s3.5e3}
d\psi=(n+1)\sqrt{-1}\,\eta\wedge \psi.
\end{equation*} 
\end{prop}
\begin{proof}
If $(M,g)$ is a Sasakian-Einstein manifold, 
then there exits a weighted Calabi-Yau structure $(\Omega,\omega)$ on $C(M)$ 
with the K\"{a}hler metric $\overline{g}$. 
Then the K\"{a}hler form $\omega$ is given by 
\begin{equation*}\label{s3.5e4}
\omega=d(\frac{1}{2}r^{2}\eta)=rdr\wedge\eta+\frac{r^{2}}{2}d\eta. 
\end{equation*}
We define $\psi^{\prime}$ as the $n$-form 
\begin{equation*}\label{s3.5e5}
\psi^{\prime}=i_{r\frac{\partial}{\partial r}}\Omega
\end{equation*}
on $C(M)$. 
Then $\psi^{\prime}$ is a transversely $(n,0)$-form on 
$(C(M),\mathcal{F}_{\langle\xi,r\frac{\partial}{\partial r}\rangle})$ such that 
\begin{equation}\label{s3.5e6}
\Omega=(\frac{dr}{r}+\sqrt{-1}\eta)\wedge\psi^{\prime}
\end{equation}
since $\psi^{\prime}=i_{v}\Omega+i_{\overline{v}}\Omega=i_{v}\Omega$ for 
the holomorphic vector field $v=\frac{1}{2}(r\frac{\partial}{\partial r}-\sqrt{-1}\xi)$. 
The equation (\ref{s3.5e6}) implies that 
\begin{equation}\label{s3.5e8}
d\psi^{\prime}=L_{r\frac{\partial}{\partial r}}\Omega=(n+1)\Omega
=(n+1)(\frac{dr}{r}+\sqrt{-1}\eta)\wedge\psi^{\prime}.
\end{equation}
It is straightforward to 
\begin{eqnarray*}
\Omega\wedge \overline{\Omega}&=&2(-1)^{n}\sqrt{-1}\,r^{-1}dr\wedge\eta
\wedge\psi^{\prime}\wedge\overline{\psi^{\prime}},\\
\omega^{n+1}&=&(n+1)rdr\wedge\eta\wedge(\frac{1}{2}r^{2}d\eta)^{n}. 
\end{eqnarray*}
Hence it follows from the equation (\ref{s3.5e2}) that 
\begin{equation}\label{s3.5e9}
\psi^{\prime}\wedge\overline{\psi^{\prime}}=c_{n}r^{2(n+1)}(\frac{1}{2}d\eta)^{n}.
\end{equation}
Moreover, we obtain 
\begin{equation}\label{s3.5e10}
\psi^{\prime}\wedge d\eta =-2r^{-2}\sqrt{-1}\,i_{\xi}(\Omega\wedge\omega)=0
\end{equation}
since $\Omega\wedge\omega=\frac{r^{2}}{2}(\frac{dr}{r}+\sqrt{-1}\eta)\wedge d\eta\wedge\psi^{\prime}$. 
Let $\psi$ denote the pull-back of $\psi^{\prime}$ to $M$ by the inclusion $i:M\to C(M)$ : 
\begin{equation*}\label{s3.5e11}
\psi=i^{*}\psi^{\prime}.
\end{equation*}
Then $\psi$ is a transversely $(n,0)$-form on 
$(M,\mathcal{F}_{\xi})$ such that $d\psi=(n+1)\sqrt{-1}\eta\wedge\psi$ 
by taking the pull-back of the equation (\ref{s3.5e8}) by $i$. 
Moreover, it follows from the equations (\ref{s3.5e9}) and (\ref{s3.5e10}) that 
$\psi\wedge d\eta=\overline{\psi}\wedge d\eta=0$ and $\psi\wedge \overline{\psi}=c_{n}(\frac{1}{2}d\eta)^{n}$. 
Hence $(\psi,\frac{1}{2}d\eta)$ is an almost transverse Calabi-Yau structures. 

Conversely, we assume that there exist a contact form $\eta$ and a complex valued $n$-form $\psi$ 
such that $(\psi, \frac{1}{2}d\eta)$ is an almost transverse Calabi-Yau structure. 
Then the contact form $\eta$ induces the Reeb vector field $\xi$ and the Reeb foliation $\mathcal{F}_{\xi}$. 
It follows from Proposition~\ref{s3.2p2} that $\psi$ defines the transverse complex structure $J_{\psi}$ on $(M, \mathcal{F}_{\xi})$. 
Then $(\eta, \xi, J_{\psi}, g)$ is a Sasaki structure on $M$ and the metric cone $(C(M), \overline{g})$ has the K\"{a}hler form 
\begin{equation*}\label{s3.5e14}
\omega=d(\frac{1}{2}r^{2}\eta).
\end{equation*}
It is easy to see that $L_{r\frac{\partial}{\partial r}}\omega=2\omega$. 
We define $\psi^{\prime}$ as the $n$-form 
\begin{equation*}\label{s3.5e15}
\psi^{\prime}=r^{n+1}\psi
\end{equation*}
on $C(M)$, where we consider $\psi$ as an $n$-form on $C(M)$ by the trivial extension. 
Then  $\psi^{\prime}$ is a transversely $(n,0)$-form on 
$(C(M),\mathcal{F}_{\langle\xi,r\frac{\partial}{\partial r}\rangle})$ 
such that $i_{\overline{v}}\psi^{\prime}=0$ for 
the holomorphic vector field $v=\frac{1}{2}(r\frac{\partial}{\partial r}-\sqrt{-1}\xi)$. 
We define $\Omega$ as the $(n+1)$-form  
\begin{equation}\label{s3.5e16}
\Omega=(\frac{dr}{r}+\sqrt{-1}\eta)\wedge\psi^{\prime}
\end{equation}
on $C(M)$. 
Then $\Omega$ is a holomorphic $(n+1)$-form satisfying 
\begin{equation*}
L_{r\frac{\partial}{\partial r}}\Omega=di_{r\frac{\partial}{\partial r}}\Omega
=d\psi^{\prime}=(n+1)\Omega
\end{equation*}
since $i_{r\frac{\partial}{\partial r}}\Omega=\psi^{\prime}$ and 
$d\psi^{\prime}=(n+1)(\frac{dr}{r}+\sqrt{-1}\eta)\wedge\psi^{\prime}$. 
The equation (\ref{s3.5e16}) implies that 
\begin{eqnarray*}
\Omega\wedge \overline{\Omega}
&=&2(-1)^{n}\sqrt{-1}\,r^{-1}dr\wedge\eta
\wedge\psi^{\prime}\wedge\overline{\psi^{\prime}}\\
&=&2(-1)^{n}\sqrt{-1}\,r^{2n+1}dr\wedge\eta
\wedge\psi\wedge\overline{\psi}\\
&=&2(-1)^{n}\sqrt{-1}c_{n}r^{2n+1}dr\wedge\eta\wedge(\frac{1}{2}d\eta)^{n}\\
&=&(n+1)c_{n+1}rdr\wedge\eta\wedge(\frac{1}{2}r^{2}d\eta)^{n}\\
&=&c_{n+1}\omega^{n+1}. 
\end{eqnarray*}
Hence $(\Omega, \omega)$ is a weighted Calabi-Yau structure on $C(M)$, and we finish the proof. 
\end{proof}

\section{Deformations of special Legendrian submanifolds}
In this section, we assume that $(M,g)$ is a simply connected and compact Sasaki-Einstein manifold of dimension $(2n+1)$ 
and $(\xi, \eta, \Psi, g)$ is the Sasaki-Einstein structure on $M$. 
Let $(C(M),\overline{g})$ be the metric cone of $(M,g)$. 
We fix an almost transverse Calabi-Yau structure $(\psi, \frac{1}{2}d\eta)$ and 
the corresponding weighted Calabi-Yau structure $(\Omega, \omega)$ on $C(M)$ 
as in Proposition~\ref{s3.5p2}. 

\subsection{Special Legendrian submanifolds}
The holomorphic $(n+1)$-form $\Omega$ induces the calibration $\Omega^{\rm Re}$ on $C(M)$ and 
special Lagrangian submanifolds are defined by calibrated submanifolds. 
We consider such submanifolds of cone type. 
For any submanifold $X$ in $M$, 
the cone $C(X)=\mathbb{R}_{>0}\times X$ is a submanifold in $C(M)$. 
We identify $X$ with the hypersurface $\{1\}\times X$ in $C(X)$. 
Then $X$ can be considered as the link $C(X)\cap M$. 
\begin{defi}\label{s4.1.1d1}
{\rm 
A submanifold $X$ is a {\it special Legendrian submanifold} in $M$ 
if the cone $C(X)$ is a special Lagrangian submanifold in $C(M)$. 
}
\end{defi}

Any special Legendrian submanifold $X$ is a minimal submanifold in $M$, 
that is, the mean curvature vector field $H$ of $X$ vanishes. 
In fact, the mean curvature vector field $\widetilde{H}$ of the cone $C(X)$ satisfies that 
$\widetilde{H}_{(r,x)}=\frac{1}{r^{2}}H_{x}$ at $(r,x)\in \mathbb{R}_{>0}\times X=C(X)$. 
Hence $H=0$ since a special Lagrangian cone $C(X)$ is minimal. 

We also denote by $\eta$ the extension of the contact form $\eta$ to $C(M)$. 
We provide a characterization of special Lagrangian cones in $C(M)$. 
\begin{prop}\label{s4.1.1p2}
Let $\widetilde{X}$ be an $(n+1)$-dimensional closed submanifold in $C(M)$ with 
the inclusion $\tilde{\iota}:\widetilde{X}\hookrightarrow C(M)$. 
The submanifold $\widetilde{X}$ is a special Lagrangian cone if and only if 
$\tilde{\iota}^{*}\Omega^{\rm Im}=\tilde{\iota}^{*}\eta=0$. 
\end{prop}
\begin{proof}
We remark that a special Lagrangian submanifold in $C(M)$ is characterized by 
an $(n+1)$-dimensional submanifold $\widetilde{X}$ in $C(M)$ such that 
$\tilde{\iota}^{*}\Omega^{\rm Im}=0$ and $\tilde{\iota}^{*}\omega=0$. 
If $\widetilde{X}$ is a special Lagrangian cone, 
then the vector field $r\frac{\partial}{\partial r}$ is tangent to $\widetilde{X}$. 
The vector fields $\xi$ and $r\frac{\partial}{\partial r}$ span a symplectic subspace 
of $T_{p}C(M)$ with respect to $\omega_{p}$ at the each point $p\in C(M)$. 
We can obtain $\tilde{\iota}^{*}\eta=0$ 
since $\eta=i_{r\frac{\partial}{\partial r}}\omega$ and $\tilde{\iota}^{*}\omega=0$. 

Conversely, if an $(n+1)$-dimensional submanifold $\widetilde{X}$ satisfies 
$\tilde{\iota}^{*}\Omega^{\rm Im}=\tilde{\iota}^{*}\eta =0$, 
then $\widetilde{X}$ is a special Lagrangian submanifold 
since $\tilde{\iota}^{*}\omega=\frac{1}{2}d(r^{2}\tilde{\iota}^{*}\eta)=0$. 
In order to see that $\widetilde{X}$ is a cone, we consider the set 
\[
I_{p}=\{ a \in \mathbb{R}_{>0}\mid \lambda_{a}p\in \widetilde{X}\}
\]
for each $p\in \widetilde{X}$. 
Then $I_{p}$ is a closed subset of $\mathbb{R}_{>0}$ since $\widetilde{X}$ is closed. 
On the other hand, the vector field $r\frac{\partial}{\partial r}$ 
has to be tangent to $\widetilde{X}$ 
since $\widetilde{X}$ is Lagrangian and $\tilde{\iota}^{*}\eta=0$. 
The vector field $r\frac{\partial}{\partial r}$ is the infinitesimal transformation 
of the action $\lambda$. 
Therefore $I_{p}$ is open, 
and so $I_{p}=\mathbb{R}_{>0}$ for each point $p\in \widetilde{X}$. 
Hence $\widetilde{X}$ is a cone, and it completes the proof. 
\end{proof}

We denote by $\psi^{\rm Re}$ and $\psi^{\rm Im}$ 
the real part and the imaginary part of $\psi$, respectively. 
Then we have a characterization of special Legendrian submanifolds : 

\begin{prop}\label{s4.1.2p1}
An $n$-dimensional submanifold $X$ in $M$ is a special Legendrian submanifold 
if and only if $\iota^{*}\psi^{\rm Im}=\iota^{*}\eta=0$. 
\end{prop}
\begin{proof}
Let $(\Omega,\omega)$ be the Calabi-Yau structure on $C(M)$. 
As in the proof of Proposition~\ref{s3.5p2}, 
$\Omega$ is given by 
\begin{equation*}\label{s4.1e1}
\Omega=(\frac{dr}{r}+\sqrt{-1}\eta)\wedge r^{n+1}\psi
\end{equation*}
on $C(M)$. 
It implies that $\Omega^{\rm Im}=r^{n}dr\wedge \psi^{\rm Im}+r^{n+1}\eta\wedge\psi^{\rm Re}$. 
Thus the equation $\iota^{*}\psi^{\rm Im}=\iota^{*}\eta=0$ is equivalent to 
$\tilde{\iota}^{*}\Omega^{\rm Im}=\tilde{\iota}^{*}\eta=0$ 
where $\tilde{\iota}$ is the embedding $\tilde{\iota}:C(X)\to C(M)$. 
Hence the condition $\iota^{*}\psi^{\rm Im}=\iota^{*}\eta=0$ holds if and only if 
$X$ is a special Legendrian submanifold by Proposition~\ref{s4.1.1p2} 
and the definition of special Legendrian submanifolds. 
\end{proof}

For a real constant $\theta$, the $n$-form $e^{\sqrt{-1}\,\theta}\Omega$ induces 
a calibration $(e^{\sqrt{-1}\,\theta}\Omega)^{\rm Re}$ on $(C(M),\overline{g})$. 
Then calibrated submanifolds are called {\it $\theta$-special Lagrangian submanifolds} for a phase $\theta$. 
\begin{defi}\label{s4.1.1d12}
{\rm 
A submanifold $X$ in $M$ is {\it $\theta$-special Legendrian} if 
the cone $C(X)$ is a $\theta$-special Lagrangian submanifold in $C(M)$ for a phase $\theta$. 
}
\end{defi}
A $\theta$-special Legendrian submanifold is a special Legendrian submanifold 
in the Sasaki-Einstein manifold with the almost transverse Calabi-Yau structure $(e^{\sqrt{-1}\,\theta}\psi, \frac{1}{2}d\eta)$. 
Hence any $\theta$-special Legendrian submanifold is minimal. Moreover the converse is true, that is, 
\begin{prop}\label{s4.1.1p1}
A connected oriented submanifold $X$ in $M$ is minimal Legendrian 
if and only if $X$ is $\theta$-special Legendrian for a phase $\theta$. 
\end{prop}
\begin{proof} 
We only show that a minimal Legendrian submanifold is $\theta$-special Legendrian for a phase $\theta$. 
Let $X$ be a connected oriented Legendrian submanifold in $M$ and $H$ the mean curvature vector of $X$. 
Then there exists a $\mathbb{R}/\mathbb{Z}$-valued function $\theta$ on $X$ such that $*\iota^{*}\psi=e^{-\sqrt{-1}\,\theta}$ 
where $*$ is the Hodge operator with respect to the metric $\iota^*g$ on $X$. 
If we regard $\theta$ as the function on $C(X)$ by the trivial extension, 
then it follows from $\Omega=(\frac{dr}{r}+\sqrt{-1}\eta)\wedge r^{n}\psi$ that $\tilde{*}\tilde{\iota}^{*}\Omega=e^{-\sqrt{-1}\,\theta}$ 
where $\tilde{\iota}$ is the inclusion $\tilde{\iota}:C(X)\hookrightarrow C(M)$ and 
$\tilde{*}$ is the Hodge operator with respect to the metric $\tilde{\iota}^*\overline{g}$ on $C(X)$. 
Hence $d\theta=\tilde{\iota}^{*}(i_{\widetilde{H}}\omega)$ for the mean curvature vector $\widetilde{H}$ of $C(X)$ (Lemma 2.1.\cite{TY}). 
It follows from $d\theta(\frac{\partial}{\partial r})=0$ that $\widetilde{H}$ has no component of $\langle \xi\rangle_{\mathbb{R}}$. 
Since the transverse part of $\omega$ is $\frac{r^2}{2}d\eta$ and $\widetilde{H}=\frac{1}{r^2}H$, 
the equation $\tilde{\iota}^{*}(i_{\widetilde{H}}\omega)=\tilde{\iota}^{*}(i_{\widetilde{H}}\frac{r^2}{2}d\eta)=\iota^{*}(i_{H}\omega^{T})$ holds 
where $\omega^{T}$ is the transverse $2$-form $\frac{1}{2}d\eta$ on $(M, \mathcal{F}_{\xi})$. 
Thus we obtain 
\begin{equation}\label{s4.1e2}
d\theta=\iota^{*}(i_{H}\omega^{T})
\end{equation}
on $M$. Hence $\theta$ is constant if $X$ is minimal. 
Then $X$ is special Legendrian with respect to 
the almost transverse Calabi-Yau structure $(e^{\sqrt{-1}\,\theta}\psi, \frac{1}{2}d\eta)$.  
Therefore $X$ is the $\theta$-special Legendrian submanifold for the phase $\theta$. 
\end{proof}

\subsection{Infinitesimal deformations of special Legendrian submanifolds}
Let $X$ be a compact special Legendrian submanifold in $M$. 
We denote by $\mathcal{M}_{X}$ the moduli space of special Legendrian deformations of $X$. 
The following result is shown by Futaki, Hattori and Yamamoto 
by considering deformations of the special Lagrangian cone in $C(M)$ \cite{FHY}. 
We provide its different proof from the point of view of $\Phi$-deformations. 
\begin{prop}\label{s4.1.2p2}
The infinitesimal deformation space of $X$ is  
isomorphic to the space ${\rm Ker}(\Delta_{0}-2(n+1))$ of $2(n+1)$-eigenfunctions 
of the Laplace operator $\Delta_{0}$. 
\end{prop}
\begin{proof}
Proposition~\ref{s4.1.2p1} implies that a special Legendrian deformation is a $(\psi^{\rm Im}, \eta)$-deformation and 
$\mathcal{M}_{X}$ is the moduli space $\mathcal{M}_{X}(\psi^{\rm Im}, \eta)$. 
We take the set $U$ as in $(\ref{s2.1e2})$ and define the map $F:U\to \wedge^{0}\oplus \wedge^{1}$ by 
\begin{equation*}\label{s4.1.2e3}
F(v)=(*\exp_{v}^{*}\psi^{\rm Im},\exp_{v}^{*}\eta)
\end{equation*}
for $v\in U$, then we can regard $F^{-1}(0)$ as a set of special Legendrian submanifolds 
in $M$ which is near $X$ in $C^1$ sense. 
The linearization $d_{0}F$ of $F$ at $0\in U$ is given by 
\begin{eqnarray*}\label{s4.1.2e4}
d_{0}F(v)&=&(*\iota^{*}L_{\tilde{v}}\psi^{\rm Im}, \iota^{*}L_{\tilde{v}}\eta)\\
&=&(*d\iota^{*}(i_{\tilde{v}}\psi^{\rm Im})+(n+1)\iota^{*}(i_{\tilde{v}}\eta\psi^{\rm Re}), 
d\iota^{*}(i_{\tilde{v}}\eta)+\iota^{*}(i_{\tilde{v}}d\eta))\\
&=&(\frac{1}{2}d^{*}\iota^{*}(i_{\tilde{v}}d\eta)+(n+1)\iota^{*}(i_{\tilde{v}}\eta), 
d\iota^{*}(i_{\tilde{v}}\eta)+\iota^{*}(i_{\tilde{v}}d\eta))
\end{eqnarray*}
for $v\in U$ where $\tilde{v}$ is an extension of $v$ to $M$. 
In the last equation, we use that 
$\iota^{*}(i_{\tilde{v}}\psi^{\rm Im})=-\frac{1}{2}*\iota^{*}(i_{\tilde{v}}d\eta)$. 
Under the identification 
\[
\Gamma(NX)\simeq \wedge^{0}\oplus\wedge^{1}
\] 
given by $v\mapsto (i_{v}\eta, \frac{1}{2}i_{v}d\eta)$, we identify $d_{0}F$ with the map 
$d_{0}F:\wedge^{0}\oplus\wedge^{1}\to\wedge^{0}\oplus \wedge^{1}$ defined by  
\begin{equation*}\label{s4.1.2e5}
d_{0}F(f,\alpha)=(d^{*}\alpha+(n+1)f, df+2\alpha)
\end{equation*}
for $(f,\alpha)\in\wedge^{0}\oplus\wedge^{1}$. 
Then it turns out that 
\begin{eqnarray*}
{\rm Ker}(d_{0}F)&=&\{(f,\alpha)\in\wedge^{0}\oplus\wedge^{1} 
\mid d^{*}\alpha+(n+1)f=0, 2\alpha+df=0\} \\
&=&\{(f,-\frac{1}{2}df)\in\wedge^{0}\oplus\wedge^{1} \mid (\Delta_{0}-2(n+1))f=0 \}
\end{eqnarray*}
which is isomorphic to ${\rm Ker}(\Delta_{0}-2(n+1))$, 
and hence it completes the proof. 
\end{proof}
The obstruction space of special Legendrian deformations is the cokernel ${\rm Coker}(d_0F)$ 
as in the proof of Proposition~\ref{s4.1.2p2}. 
However, the space ${\rm Coker}(d_0F)$ is isomorphic to ${\rm Ker}(d_0F)$ 
since $d_{0}F:\wedge^{0}\oplus\wedge^{1}\to\wedge^{0}\oplus \wedge^{1}$ is self dual. 
Hence the obstruction space does not vanish whenever $X$ has non-trivial deformations. 
Thus there may exist some obstruction of special Legendrian deformations.

\subsection{The intersection of two deformation spaces}
Let $\omega^{T}$ be the $2$-form $\frac{1}{2}d\eta$ on $M$. 
We remark that $X$ is a special Legendrian submanifold if and only if 
$\iota^{*}\psi^{\rm Im}=\iota^{*}\eta=\iota^{*}\omega^{T}=0$. 
Let $X$ be a compact special Legendrian submanifold in $M$. 
We consider $(\psi^{\rm Im}, \omega^{T})$-deformations of $X$ and 
denote by $\mathcal{N}_{X}$ the moduli spaces $\mathcal{M}_{X}(\psi^{\rm Im}, \omega^{T})$. 
We fix an integer $s\ge 3$ and a real number $\alpha$ with $0<\alpha<1$ and set $\mathcal{N}_{X}^{s,\alpha}$ 
as the moduli space $\mathcal{M}_{X}^{s,\alpha}(\psi^{\rm Im}, \omega^{T})$ of $(\psi^{\rm Im}, \omega^{T})$-deformations of $C^{s, \alpha}$-class. 
Then we have 
\begin{prop}\label{s4.3p1}
The moduli space $\mathcal{N}_{X}^{s,\alpha}$ is smooth at $0_X$. 
The tangent space $T_{0_X}\mathcal{N}_{X}^{s,\alpha}$ 
is isomorphic to 
$\{(-\frac{1}{n+1}d^{*}\alpha,\alpha)\in C^{s,\alpha}(\wedge^{0}\oplus\wedge^{1}) \mid d\alpha=0 \}$. 
\end{prop}
\begin{proof}
We take the set $U$ as in $(\ref{s2.1e2})$ and define the map $G:U\to \wedge^{0}\oplus \wedge^{2}$ by 
\begin{equation*}\label{s4.1.3e2}
G(v)=(*\exp_{v}^{*}\psi^{\rm Im}, \exp_{v}^{*}\omega^{T})
\end{equation*}
for $v\in U$. It follows that for $v\in U$ 
\begin{eqnarray*}\label{s4.1.3e3}
d_{0}G(v)&=&(*(d\iota^{*}(i_{\tilde{v}}\psi^{\rm Im})+
(n+1)\iota^{*}(i_{\tilde{v}}\eta\psi^{\rm Re})-(n+1)\iota^{*}(\eta\wedge i_{\tilde{v}}\psi^{\rm Re}),\  
d\iota^{*}(i_{\tilde{v}}\omega^{T})) \\
&=&(d^{*}\iota^{*}(i_{\tilde{v}}\omega^{T})+
(n+1)\iota^{*}(i_{\tilde{v}}\eta),\ d\iota^{*}(i_{\tilde{v}}\omega^{T}))
\end{eqnarray*}
where $\tilde{v}$ is an extension of $v$ to $M$.  
In the last equation, we use that 
$\iota^{*}(i_{\tilde{v}}\psi^{\rm Im})=-\frac{1}{2}*\iota^{*}(i_{\tilde{v}}d\eta)$ and 
$\iota^{*}\psi^{\rm Re}={\rm vol}(X)$. 
Under the identification $\Gamma(NX)\simeq \wedge^{0}\oplus\wedge^{1}$ given 
by $v\mapsto (i_{v}\eta, i_{v}\omega^{T})$, we identify $d_{0}G$ with the map 
$D_1:\wedge^{0}\oplus\wedge^{1}\to\wedge^{0}\oplus \wedge^{2}$ defined by  
\begin{equation*}\label{s4.1.3e4}
D_1(f,\alpha)=(d^{*}\alpha+(n+1)f, d\alpha) 
\end{equation*}
for $(f,\alpha)\in\wedge^0\oplus \wedge^1$. Then it turns out that 
\begin{eqnarray*}
{\rm Ker}(D_{1})&=&\{(f,\alpha)\in\wedge^{0}\oplus\wedge^{1} 
\mid d^{*}\alpha+(n+1)f=0, d\alpha=0\} \\
&=&\{(-\frac{1}{n+1}d^{*}\alpha,\alpha)\in\wedge^{0}\oplus\wedge^{1} \mid d\alpha=0 \}.
\end{eqnarray*}
Now we provide a complex as follows 
\begin{equation*}\label{s4.1.3e5}
0\to \wedge^{0}\oplus\wedge^{1}\stackrel{D_{1}}{\longrightarrow}
\wedge^{0}\oplus \wedge^{2}\stackrel{D_{2}}{\longrightarrow} 
\wedge^{3}\to 0 
\end{equation*}
where the operator $D_{2}$ is given by 
\begin{equation*}
D_{2}(f,\beta)=d\beta
\end{equation*}
for $(f,\beta)\in\wedge^{0}\oplus \wedge^{2}$. 
It is easy to see that  
\begin{eqnarray*}
P_{1}(f,\alpha)&=&((n+1)^{2}f+(n+1)d^{*}\alpha, \Delta_{1}\alpha+(n+1)df),\\
P_{2}(f,\beta)&=&((\Delta_{0}+(n+1)^{2})f, \Delta_{2}\beta).
\end{eqnarray*}
Hence $P_{2}$ is the elliptic operator. 
It follows from ${\rm Im}(G)\subset \wedge^{0}\oplus d\wedge^{1}$  
that ${\rm Im}(G)$ is perpendicular 
to $\Ker{P_{2}}(=\{0\}\oplus \mathcal{H}^{2}(X))$ and ${\rm Im}(D^{*}_{2})$. 
Hence we obtain ${\rm Im}(G)\subset {\rm Im}(D_{1})$ by the Hodge decomposition 
$\wedge^{0}\oplus \wedge^{2}=\Ker{P_{2}}\oplus {\rm Im}(D_{1})\oplus {\rm Im}(D_{2}^{*})$. 
Proposition~\ref{s2.1p1} implies that $\mathcal{N}_{X}^{s,\alpha}$ is smooth at $0_X$ 
with the tangent space ${\rm Ker}(D_{1}^{s,\alpha})$. 
Hence we finish the proof. 
\end{proof}

Let $\mathcal{L}_{X}$ be the moduli space of Legendrian deformations of $X$ of $C^{\infty}$-class 
and $\mathcal{L}_{X}^{s,\alpha}$ that of $C^{s, \alpha}$-class. 
Then we have the following 
\begin{thm}\label{s4.3t1}
The moduli space $\mathcal{M}_{X}$ is the intersection $\mathcal{N}_{X}\cap\mathcal{L}_{X}$ 
where $\mathcal{N}_{X}^{s,\alpha}$ and $\mathcal{L}_{X}^{s,\alpha}$ are smooth. 
\end{thm}
\begin{proof}
The moduli space $\mathcal{M}_{X}$ is the intersection $\mathcal{N}_{X}\cap \mathcal{L}_{X}$
since $\mathcal{M}_{X}(\psi^{\rm Im}, \eta)=\mathcal{M}_{X}(\psi^{\rm Im}, \eta, \omega^{T})=\mathcal{M}_{X}(\psi^{\rm Im}, \omega^{T})\cap\mathcal{M}_{X}(\eta)$. 
The space $\mathcal{L}_{X}^{s,\alpha}$ is smooth by Proposition~\ref{s2.2.4p1}. 
It follows from Proposition~\ref{s4.3p1} that $\mathcal{N}_{X}^{s,\alpha}$ is smooth at any point of $\mathcal{N}_{X}\cap\mathcal{L}_{X}$. 
Hence it completes the proof. 
\end{proof}

\subsection{Transverse deformations of special Legendrian submanifolds}
We provide a definition of transverse deformations of submanifolds in a general foliated manifold. 
Let $(M',g')$ be a Riemannian manifold and $\mathcal{F}$ a Riemannian foliation on $M'$. 
Then $\mathcal{F}$ induces the vector bundle $F$ on $M'$. 
We regard the quotient bundle $NF=TM'/F$ as the subbundle of $TM'$ which is orthogonal to $F$ with respect to $g'$. 
Let $X$ be a compact submanifold in $M'$. 
A normal deformation $f=\{f_{t}\}_{t\in [0,1]}$ of $X$ is called a \textit{transverse deformation} of $X$ 
if there exists a family $\{h_{t}\}_{t\in [0,1]}$ of diffeomorphisms of $X$ with $h_0={\rm id}_{X}$ and 
$\frac{d}{d t}h_t|_{t=0}=0$ such that $\frac{d}{d t}f_t \circ h_t\in \Gamma(NF|_{X_t})$ for each $t\in [0,1]$. 
We say that $f$ is a {\it transverse $\Phi$-deformation} of $X$ if $f$ is a transverse and $\Phi$-deformation of $X$, 
and we define $\mathcal{M}_{X}^{T}(\Phi)$ as the moduli space of transverse $\Phi$-deformations of $X$. 

We consider a transverse deformation of special Legendrian submanifolds in the Sasaki-Einstein manifold $(M,g)$. 
We assume that $\mathcal{F}$ is the Reeb foliation on $M$ and $X$ is a compact special Legendrian submanifold in $M$. 
It follows from $NF= \ker \eta$ that a transverse $(\psi^{\rm Im}, \omega^{T})$-deformation of $X$ is given by a $(\psi^{\rm Im}, \omega^{T})$-deformation $\{f_{t}\}_{t\in [0,1]}$ 
such that $\frac{d}{d t}f_t \circ h_t\in \Gamma(\ker \eta|_{X_t})$ for diffeomorphisms $\{h_{t}\}_{t\in [0,1]}$ of $X$. 
We denote by $\mathcal{N}_{X}^{T}$ the moduli space $\mathcal{M}_{X}^{T}(\psi^{\rm Im}, \omega^{T})$ 
of transverse $(\psi^{\rm Im}, \omega^{T})$-deformations of $X$. 
Then we obtain 

\begin{thm}\label{s4.1.4t1}
The moduli space $\mathcal{N}_{X}^{T}$ is smooth at $0_X$ 
and the tangent space $T_{0_X}\mathcal{N}_{X}^{T}$ 
is isomorphic to $H^{1}(X)$. 
\end{thm}
\begin{proof}
We define $NX^{T}$ by the vector bundle 
\[
NX^T=NX\cap \ker \eta|_{X}
\]
on $X$. 
Let $\mathcal{U}^{T}$ be a sufficiently small neighbourhood of the zero section in $NX^{T}$ and 
$U^T$ the set $\{ v\in \Gamma(NX^{T}) \mid v_{x}\in \mathcal{U}^T, x\in X \}$. 
We define the map $G^{T}:U^{T}\to \wedge^{0}\oplus \wedge^{2}$ by 
\begin{equation*}\label{s4.1.4e1}
G^{T}(v)=(*\exp_{v}^{*}\psi^{\rm Im}, \exp_{v}^{*}\omega^{T}) 
\end{equation*}
for $v\in U^{T}$. 
We remark that if a geodesic is orthogonal to a leaf at one point, then the geodesic is orthogonal to any leaf~\cite{R}. 
It yields that the deformation $\exp_{v}=\{\exp_{tv}\}_{t\in [0,1]}$ satisfies $\frac{d}{d t}\exp_{tv} \in \Gamma(NF|_{\exp_{tv}(X)})$ for each $t$ if $v$ is in $\Gamma(NF)$. 
We can replace $\{\exp_{tv}\}_{t\in [0,1]}$ to a normal deformation $\{\exp_{tv}\circ h'_t\}_{t\in [0,1]}$ by 
a time dependent diffeomorphisms $h'_t$ of $X$ with $h'_0={\rm id}_{X}$ and $\frac{d}{d t}h'_t|_{t=0}=0$. 
Hence $(G^T)^{-1}(0)$ is identified with a neighbourhood of $0_X$ in $\mathcal{N}_{X}^{T}$. 
We have 
\begin{eqnarray*}\label{s4.1.4e2}
d_{0}G^{T}(v)&=&(*(d\iota^{*}(i_{\tilde{v}}\psi^{\rm Im})+
(n+1)\iota^{*}(i_{\tilde{v}}\eta\psi^{\rm Re})-(n+1)\iota^{*}(\eta\wedge i_{\tilde{v}}\psi^{\rm Re})),\  
d\iota^{*}(i_{\tilde{v}}\omega^{T})) \\
&=&(d^{*}\iota^{*}(i_{\tilde{v}}\omega^{T}),\ d\iota^{*}(i_{\tilde{v}}\omega^{T}))
\end{eqnarray*}
for $v\in U^{T}$ where $\tilde{v}$ is an extension of $v$ to $M$.  
In the last equation, we use that 
$\iota^{*}(i_{\tilde{v}}\psi^{\rm Im})=-*\iota^{*}(i_{\tilde{v}}\omega^{T})$ and 
$\iota^{*}(i_{\tilde{v}}\eta)=i_{v}\eta=0$. 
Under the identification 
\[
\Gamma(NX^{T})\simeq \wedge^{1}
\]
given by $v\mapsto i_{v}\omega^{T}$, we identify $d_{0}G^{T}$ with the map 
$D_1:\wedge^{1}\to\wedge^{0}\oplus \wedge^{2}$ defined by  
\begin{equation*}\label{s4.1.4e3}
D_1(\alpha)=(d^{*}\alpha, d\alpha)
\end{equation*}
for $\alpha\in\wedge^{1}$. Then it turns out that 
\begin{equation*}
{\rm Ker}(D_{1})
=\{\alpha\in \wedge^{1} \mid d^{*}\alpha=d\alpha=0\} 
=\mathcal{H}^{1}(X).
\end{equation*}
Now we provide a complex as follows 
\begin{equation*}\label{s4.1.4e4}
0\to \wedge^{1}\stackrel{D_{1}}{\longrightarrow}
\wedge^{0}\oplus \wedge^{2}\stackrel{D_{2}}{\longrightarrow} 
\wedge^{3}\to 0 
\end{equation*}
where the operator $D_{2}$ is given by 
\begin{equation*}
D_{2}(f,\beta)=d\beta
\end{equation*}
for $(f,\beta)\in\wedge^{0}\oplus \wedge^{2}$. 
It is easy to see that  
\begin{eqnarray*}
P_{1}(\alpha)&=&\Delta_{1}\alpha,\\
P_{2}(f,\beta)&=&(\Delta_{0}f, \Delta_{2}\beta).
\end{eqnarray*}
Hence $P_{1}$ and $P_{2}$ are elliptic. 
It follows from ${\rm Im}(G^{T})\subset d^{*}\!\wedge^{1}\oplus d\wedge^{1}$  
that ${\rm Im}(G^{T})$ is perpendicular to 
$\Ker{P_{2}}(=\mathcal{H}^{0}(X)\oplus \mathcal{H}^{2}(X))$ and ${\rm Im}(D^{*}_{2})$. 
Hence we obtain ${\rm Im}(G^{T})\subset {\rm Im}(D_{1})$ by the Hodge decomposition 
$\wedge^{0}\oplus \wedge^{2}=\Ker{P_{2}}\oplus {\rm Im}(D_{1})\oplus {\rm Im}(D_{2}^{*})$. 
Proposition~\ref{s2.1p1} implies that $\mathcal{N}_{X}^{T}$ is smooth at $0_X$ 
with the tangent space ${\rm Ker}(D_{1})=\mathcal{H}^{1}(X)$, 
and hence it completes the proof.  
\end{proof}

\section{Further results}
In this section, we assume that $(M,g)$ is a simply connected and compact Riemannian manifold 
of dimension $2n+1$. 

\subsection{Sasaki manifolds with almost transverse Calabi-Yau structures} 
The metric cone of a Sasaki manifold $(M,g)$ with an almost transverse Calabi-Yau structure is 
not a Calabi-Yau manifold unless $(M,g)$ is Sasaki-Einstein. 
A K\"{a}hler manifold with a non-vanishing holomorphic volume form is called an {\it almost Calabi-Yau manifold}. 
We refer to \cite{J1} for an almost Calabi-Yau manifolds. 
We provide a generalization of Proposition \ref{s3.5p2} as follows 
\begin{prop}\label{s5.1p1}
The Riemannian manifold $(M,g)$ has a contact form $\eta$ and a complex valued $n$-form $\psi$ such that 
$(\psi, \frac{1}{2}d\eta)$ is an almost transverse Calabi-Yau structure on $(M, \mathcal{F})$, 
where $\mathcal{F}$ is the Reeb foliation induced by $\eta$, with $d\psi=\kappa\sqrt{-1}\,\eta\wedge \psi$ for a real constant $\kappa$ 
if and only if the metric cone $(C(M), \overline{g})$ is an almost Calabi-Yau manifold with 
a non-vanishing holomorphic section $\Omega$ of $K_{C(M)}$ and a K\"{a}hler form $\omega$ such that 
\begin{equation}\label{s5.1e1}
\Omega\wedge \overline{\Omega}=r^{2(\kappa -n-1)}c_{n+1}\omega^{n+1}
\end{equation}
and 
\begin{eqnarray*}
&&L_{r\frac{\partial}{\partial r}}\Omega=\kappa\Omega, \label{s2.5eq8}\\
&&L_{r\frac{\partial}{\partial r}}\omega=2\omega. \label{s2.5eq9}
\end{eqnarray*} 
\end{prop}
\begin{proof}
We assume that there exists an almost transverse Calabi-Yau structure $(\psi, \frac{1}{2}d\eta)$ 
with $d\psi=\kappa\sqrt{-1}\,\eta\wedge \psi$. 
Then the K\"{a}hler form $\omega$ is given by $\omega=d(\frac{1}{2}r^{2}\eta)$ and satisfies $L_{r\frac{\partial}{\partial r}}\omega=2\omega$. 
We define $\psi^{\prime}$ as the $n$-form 
\begin{equation*}\label{s5.1e2}
\psi^{\prime}=r^{\kappa}\psi
\end{equation*}
on $C(M)$ and $\Omega$ as the $(n+1)$-form  
\begin{equation*}\label{s5.1e3}
\Omega=(\frac{dr}{r}+\sqrt{-1}\eta)\wedge\psi^{\prime}
\end{equation*}
on $C(M)$. 
Then $\Omega$ is the non-vanishing holomorphic $(n+1)$-form satisfying 
$L_{r\frac{\partial}{\partial r}}\Omega=di_{r\frac{\partial}{\partial r}}\Omega
=d\psi^{\prime}=\kappa\Omega$. The equation (\ref{s3.5e16}) implies that 
\begin{eqnarray*}
\Omega\wedge \overline{\Omega}
&=&2(-1)^{n}\sqrt{-1}\,r^{-1}dr\wedge\eta
\wedge\psi^{\prime}\wedge\overline{\psi^{\prime}}\\
&=&2(-1)^{n}\sqrt{-1}\,r^{2\kappa-1}dr\wedge\eta
\wedge\psi\wedge\overline{\psi}\\
&=&2(-1)^{n}\sqrt{-1}c_{n}r^{2\kappa-1}dr\wedge\eta\wedge(\frac{1}{2}d\eta)^{n}\\
&=&(n+1)c_{n+1}r^{2(\kappa -n-1)}rdr\wedge\eta\wedge(\frac{1}{2}r^{2}d\eta)^{n}\\
&=&r^{2(\kappa -n-1)}c_{n+1}\omega^{n+1}. 
\end{eqnarray*}

Conversely, if there exists a non-vanishing holomorphic section $\Omega$ of $K_{C(M)}$ and a K\"{a}hler form $\omega$ on $C(M)$, 
then the K\"{a}hler form $\omega$ is given by $\omega=d(\frac{1}{2}r^{2}\eta)$ since $L_{r\frac{\partial}{\partial r}}\omega=2\omega$. 
We define $\psi^{\prime}$ as the $n$-form 
\begin{equation*}\label{s5.1e4}
\psi^{\prime}=i_{r\frac{\partial}{\partial r}}\Omega
\end{equation*}
on $C(M)$. 
Then $\psi^{\prime}$ is a transversely $(n,0)$-form on 
$(C(M),\mathcal{F}_{\langle\xi,r\frac{\partial}{\partial r}\rangle})$ such that 
$\Omega=(\frac{dr}{r}+\sqrt{-1}\eta)\wedge\psi^{\prime}$ since $\psi^{\prime}=i_{v}\Omega+i_{\overline{v}}\Omega=i_{v}\Omega$ 
for the holomorphic vector field $v=\frac{1}{2}(r\frac{\partial}{\partial r}-\sqrt{-1}\xi)$. 
The condition $L_{r\frac{\partial}{\partial r}}\Omega=\kappa\Omega$ implies that 
\begin{equation}\label{s5.1e5}
d\psi^{\prime}=L_{r\frac{\partial}{\partial r}}\Omega=\kappa\Omega
=\kappa(\frac{dr}{r}+\sqrt{-1}\eta)\wedge\psi^{\prime}.
\end{equation}
It is straightforward to 
\begin{eqnarray*}
\Omega\wedge \overline{\Omega}&=&2(-1)^{n}\sqrt{-1}\,r^{-1}dr\wedge\eta
\wedge\psi^{\prime}\wedge\overline{\psi^{\prime}},\\
\omega^{n+1}&=&(n+1)rdr\wedge\eta\wedge(\frac{1}{2}r^{2}d\eta)^{n}. 
\end{eqnarray*}
Hence 
\begin{equation}\label{s5.1e6}
\psi^{\prime}\wedge\overline{\psi^{\prime}}=c_{n}r^{2\kappa}(\frac{1}{2}d\eta)^{n}.
\end{equation}
Moreover, we obtain 
\begin{equation}\label{s5.1e7}
\psi^{\prime}\wedge d\eta =-2r^{-2}\sqrt{-1}\,i_{\xi}(\Omega\wedge\omega)=0
\end{equation}
since $\Omega\wedge\omega=\frac{r^{2}}{2}(\frac{dr}{r}+\sqrt{-1}\eta)\wedge d\eta\wedge\psi^{\prime}$. 
We define $\psi$ as the $n$-form 
\begin{equation*}\label{s5.1e8}
\psi=i^{*}\psi^{\prime}
\end{equation*}
on $M$, where $i$ is the inclusion $i:M\to C(M)$. 
Then $\psi$ is a transversely $(n,0)$-form on 
$(M,\mathcal{F}_{\xi})$ such that $d\psi=\kappa\sqrt{-1}\eta\wedge\psi$ by taking the pull-back of the equation (\ref{s5.1e5}) by $i$. 
Moreover, it follows from the equations (\ref{s5.1e6}) and (\ref{s5.1e7}) that 
$\psi\wedge d\eta=\overline{\psi}\wedge d\eta=0$ and $\psi\wedge \overline{\psi}=c_{n}(\frac{1}{2}d\eta)^{n}$ 
Hence $(\psi,\frac{1}{2}d\eta)$ is an almost transverse Calabi-Yau structures, and we finish the proof. 
\end{proof}

\begin{rem}
{\rm 
In Proposition~\ref{s5.1p1}, the holomorphic $(n+1)$-form $\Omega$ induces the calibration $\Omega^{\rm Re}$ with respect to the metric $r^{\kappa^{\prime}}\overline{g}$ 
where $\kappa^{\prime}=2(\frac{\kappa}{n+1}-1)$. 
However, $r^{\kappa^{\prime}}\overline{g}$ is not a K\"{a}hler metric unless $\kappa^{\prime}=0$ which is equal to $\kappa=n+1$. 
}
\end{rem}
Let $(\eta, \xi, \Psi, g)$ be a Sasaki structure on $M$. 
Then we can deform $(\eta, \xi, \Psi, g)$ to another Sasaki structure $(\eta_{a}, \xi_{a}, \Psi_{a}, g_{a})$ 
given by 
\begin{equation*}\label{s5.1e9}
g_{a}=ag+(a^{2}-a)\eta\otimes\eta,\ \eta_{a}=a\eta,\ \xi_{a}=\frac{1}{a}\xi,\ \Psi_{a}=\Psi
\end{equation*}
for $a>0$. These deformations are called the {\it D-homothety transformations}~\cite{Tanno}. 
The D-homothety transformation induces the rescaling $ag^{T}$ of the transverse metric $g^{T}$. 
Indeed, we have $g_{a}^{T}=ag^{T}$. 
\begin{prop}\label{s5.1p2}
Let $\kappa$ be a non-negative constant. 
Then the Riemannian manifold $(M,g)$ has a contact form $\eta$ and 
a complex valued $n$-form $\psi$ such that 
$(\psi, \frac{1}{2}d\eta)$ is an almost transverse Calabi-Yau structure on $(M, \mathcal{F})$, 
where $\mathcal{F}$ is the Reeb foliation induced by $\eta$, with $d\psi=\kappa\sqrt{-1}\,\eta\wedge \psi$ 
if and only if the metric $g$ is Sasakian such that $g^{T}$ is a transverse K\"{a}hler-Einstein metric 
with the Einstein constant $2\kappa$. 
\end{prop}
\begin{proof}
The case of $\kappa=0$ follows from the transverse Yau's Theorem~\cite{EK}. 
Hence it suffices to show the case of $\kappa>0$. 
If $(\psi, \frac{1}{2}d\eta)$ is an almost transverse Calabi-Yau structure 
such that $d\psi=\kappa\sqrt{-1}\,\eta\wedge \psi$. 
We take $a=\frac{\kappa}{n+1}$ and define $\psi_{a}$ and $\eta_{a}$ as $a^{\frac{n}{2}}\psi$ and $a\eta$, respectively. 
Then $(\psi_{a}, \frac{1}{2}d\eta_{a})$ is an almost transverse Calabi-Yau structure 
with respect to the metric $g_{a}=ag$ such that $d\psi_{a}=(n+1)\sqrt{-1}\,\eta_{a}\wedge \psi_{a}$. 
It follows from Proposition~\ref{s3.5p2} that $g_a$ is a Sasaki-Einstein metric. 
Moreover, Proposition~\ref{s3.5p2} implies that ${\rm Ric}_{g_{a}}^{T}=2(n+1)g^{T}_{a}$. 
Then the transverse metric $g^{T}$ of $g$ is transverse K\"{a}hler-Einstein with the Einstein constant $2\kappa$ 
since the Ricci tensor is invariant of the rescaling of the metric and 
${\rm Ric}_{g}^{T}={\rm Ric}_{g_{a}}^{T}=2(n+1)g_{a}^{T}=2\kappa g^{T}$. 

Conversely, we assume that $g$ is a Sasaki metric with ${\rm Ric}^{T}=2\kappa g^{T}$. 
Then the D-homothety transformation $(\eta_{a}, \xi_{a}, \Psi_{a}, g_{a})$ is a Sasaki-Einstein structure for $a=\frac{\kappa}{n+1}$ 
since ${\rm Ric}_{g_{a}}^{T}={\rm Ric}_{g}^{T}=2\kappa g^{T}=2(n+1)g_{a}^{T}$. 
It follows from Proposition~\ref{s3.5p2} that 
there exists an almost transverse ${\rm SL}_{n}(\mathbb{C})$ structure $\psi_{0}$ 
on $(M,\mathcal{F}_{\xi_{a}})$ such that $d\psi_{0}=(n+1)\sqrt{-1}\,\eta_{a}\wedge \psi_{0}$. 
We define $\psi$ as $\psi=a^{-\frac{n}{2}}\psi_{0}$. 
Then $(\psi, \frac{1}{2}d\eta)$ is an almost transverse Calabi-Yau structure  
on $(M,\mathcal{F}_{\xi})$ such that $d\psi=\kappa\sqrt{-1}\,\eta\wedge \psi$, and hence we finish the proof. 
\end{proof}

\begin{rem}
{\rm 
In the proof of Proposition \ref{s5.1p2}, the pair $(\psi, \frac{1}{2}d\eta)$ induces 
the almost transverse Calabi-Yau structure $(\psi_{a}, \frac{1}{2}d\eta_{a})$ on $M$. 
On the cone $C(M)$, the Calabi-Yau structure $(\Omega_{a}, \omega_{a})$ is defined by 
$\Omega_a=(\frac{dr_a}{r_a}+\sqrt{-1}\eta_a)\wedge r^{n+1}_a\psi_a$ and $\omega_{a}=\frac{1}{2}d(r_a^2\eta_a)$ where $r_a$ is the function $r^a$. 
Let $(\Omega, \omega)$ be the almost Calabi-Yau structure corresponding to $(\psi, \frac{1}{2}d\eta)$. 
Then the relations $\Omega_{a}=a^{\frac{n}{2}+1}\Omega$ and $\omega_{a}=\frac{a}{2}d(r^{2a}\eta)$ hold. 
}
\end{rem}
A Sasaki manifold $(M, g)$ is called an {\it $\eta$-Sasaki-Einstein manifold} 
if there exists a constant $\lambda$ such that ${\rm Ric}_{g}=\lambda g + (2n-2-\lambda)\eta\otimes \eta$. 
The condition of $\eta$-Sasaki-Einstein for the constant $\lambda$ is equivalent that 
$g^{T}$ is a transverse K\"{a}hler-Einstein metric with the Einstein constant $\lambda+2$. 
Hence it follows from Proposition~\ref{s5.1p2} that $(M,g)$ has an almost transverse Calabi-Yau structure 
$(\psi, \frac{1}{2}d\eta)$ with $d\psi=\kappa\sqrt{-1}\,\eta\wedge \psi$ for $\kappa\ge 0$ 
if and only if the metric $g$ is $\eta$-Sasaki-Einstein for the constant $\lambda=2\kappa -2$. 

\subsection{The automorphism group ${\rm Aut}(\eta,\psi)$}
Let $(M,g)$ be a Sasaki manifold with a Sasaki structure $(\eta, \xi, \Psi, g)$. 
We assume that there exists a complex valued $n$-form $\psi$ 
such that $(\psi, \frac{1}{2}d\eta)$ is an almost transverse Calabi-Yau structure
on $(M,\mathcal{F}_{\xi})$ with $d\psi=\kappa\sqrt{-1}\,\eta\wedge \psi$ for a real constant $\kappa$. 
In this section, we consider the group 
\begin{equation*}\label{s3.7e1}
{\rm Aut}(\eta,\psi)=\{f\in {\rm Diff}(M)\mid f^{*}\eta=\eta,\ f^{*}\psi=\psi \}
\end{equation*}
of automorphisms preserving $(\eta,\psi)$. 
We also define ${\rm Aut}(\eta,[\psi])$ as the group of diffeomorphisms 
preserving $\eta$ and the conformal class $[\psi]$ : 
\begin{equation*}\label{s3.7e2}
{\rm Aut}(\eta,[\psi])=\{f\in {\rm Diff}(M)
\mid f^{*}\eta=\eta,\ f^{*}\psi=h\psi,\ h\in \wedge^{0}\otimes \mathbb{C} \}. 
\end{equation*}
Then we have 
\begin{lem}\label{s3.7l1}
${\rm Aut}(\eta,[\psi])=
\{f\in {\rm Diff}(M)\mid f^{*}\eta=\eta,\ f^{*}\psi=e^{\sqrt{-1}\, \theta}\psi,\ \theta\in H^0(M) \}$
\end{lem}
\begin{proof}
If $f^{*}\psi=h\psi$ for a function $h$, 
then $d(f^{*}\psi)=dh\wedge\psi+hd\psi=dh\wedge\psi+\sqrt{-1}\,h\eta\wedge \psi$ 
and $d(f^{*}\psi)=f^{*}(d\psi)=\eta\wedge f^{*}\psi$. 
Hence $\overline{\partial}_{T}h=d_{\xi}h=0$ and so $h$ is constant. 
Moreover, the norm $\| h\|$ of $h$ is $1$ by taking 
the pull-back of $\psi\wedge\overline{\psi}=c_{n}(\frac{1}{2}d\eta)^{n}$ by $f$. 
Hence $h=e^{\sqrt{-1}\, \theta}$ for a real constant $\theta$. 
\end{proof} 
From now on, we assume that $M$ is connected. 
Then we can consider ${\rm Aut}(\eta,[\psi])$ as the group of phase changes 
$(\eta, \psi)\to (\eta, e^{\sqrt{-1}\, \theta}\psi)$ for $\theta\in \mathbb{R}$. 
Let ${\rm Aut}(\eta, \xi, \Psi, g)$ denote the group of automorphisms preserving $(\eta, \xi, \Psi, g)$.  
Then we obtain 

\begin{prop}\label{s3.7p1} 
${\rm Aut}(\eta, \xi, \Psi, g)={\rm Aut}(\eta,[\psi])$
\end{prop}
\begin{proof}
We remark that ${\rm Aut}(\eta, \xi, \Psi, g)={\rm Aut}(\eta, \Psi)$. 
In fact, if $f^{*}\eta=\eta$ and $f_{*}\circ \Psi=\Psi\circ f_{*}$, then $f_{*}\xi=\xi$ and $f^{*}g_{D}=g_{D}$. 
It implies that $f^{*}g=g$ since the metric $g$ is given by $g=g_{D}+\eta\otimes\eta$. 
Hence we will prove ${\rm Aut}(\eta, \Psi)={\rm Aut}(\eta,[\psi])$. 
If $f^{*}\eta=\eta$ and $f^{*}\circ\Psi=\Psi\circ f^{*}$, 
then $f$ preserves the vector bundle $D=\Ker\eta$. 
Moreover $f$ also preserves $D^{0,1}$ and $D^{1,0}$ 
since $D^{0,1}$ and $D^{1,0}$ are eigenvalue spaces of $\Psi$. 
Hence $f$ maps any transverse $(n,0)$-form to a transverse $(n,0)$-form by the pull back $f^{*}$. 
In particular, $f^{*}\psi=h\psi$ for a function $h$. Hence $f\in {\rm Aut}(\eta,[\psi])$. 
Conversely, if $f^{*}\eta=\eta$ and $f^{*}\psi=h\psi$, 
then $f$ preserves the vector bundle $D=\Ker\eta$ and $\Ker \psi$. 
Hence $f$ also preserves $D^{0,1}$ and $D^{1,0}$. 
It implies that $f_{*}\circ \Psi=\Psi\circ f_{*}$ 
by the definition of $\Psi$, and hence we finish the proof. 
\end{proof}

It immediately follows that the group ${\rm Aut}(\eta, \psi)$ is the subgroup of ${\rm Aut}(\eta, \xi, \Psi, g)$. 
We define $\mathfrak{aut}(\eta, \psi)$ by 
\begin{equation*}\label{s3.7e3}
\mathfrak{aut}(\eta, \psi)=
\{v\in \Gamma(TM) \mid L_{v}\eta=0,\ L_{v}\psi=0 \}. 
\end{equation*} 
Let $\mathfrak{aut}(\eta, \xi, \Psi, g)$ denote the Lie algebra 
of ${\rm Aut}(\eta, \xi, \Psi, g)$. 
Then we obtain the following relation between $\mathfrak{aut}(\eta,\psi)$ 
and $\mathfrak{aut}(\eta, \xi, \Psi, g)$.  

\begin{prop}\label{s3.7p2}
If $\kappa\neq 0$, then there exists the decomposition 
\begin{equation*}\label{s3.7e4}
\mathfrak{aut}(\eta,\psi)\oplus \langle\xi\rangle_{\mathbb{R}}=\mathfrak{aut}(\eta, \xi, \Psi, g)
\end{equation*} 
where $\langle\xi\rangle_{\mathbb{R}}$ is the $\mathbb{R}$-vector space generated by $\xi$. 
\end{prop}
\begin{proof}
Proposition~\ref{s3.7p1} implies that $\mathfrak{aut}(\eta, \xi, \Psi, g)=
\{v\in \Gamma(TM) \mid L_{v}\eta=0,\ L_{v}\psi=\sqrt{-1}\, c\psi, c\in\mathbb{R} \}$. 
The Reeb vector field $\xi$ satisfies $L_{\xi}\eta=0$ and $L_{\xi}\psi=\sqrt{-1}\kappa\psi$. 
Hence the vector space $\langle\xi\rangle_{\mathbb{R}}$ is a subspace of $\mathfrak{aut}(\eta, \xi, \Psi, g)$ 
and has the trivial intersection with $\mathfrak{aut}(\eta, \psi)$. 
If we take an element $v\in \mathfrak{aut}(\eta, \xi, \Psi, g)$, 
then there exists a real constant $c$ such that $L_{v}\psi=\sqrt{-1}c\psi$. 
The vector field $v-\frac{c}{\kappa}\xi$ is the element of $\mathfrak{aut}(\eta, \psi)$ 
since $L_{v-\frac{c}{\kappa}\xi}\psi=0$. 
Thus $v$ is in $\mathfrak{aut}(\eta, \psi)\oplus\langle\xi\rangle_{\mathbb{R}}$. 
Hence $\mathfrak{aut}(\eta, \xi, \Psi, g)$ coincides with 
$\mathfrak{aut}(\eta, \psi)\oplus\langle\xi\rangle_{\mathbb{R}}$, 
and we finish the proof. 
\end{proof}

We have the identification 
\begin{equation}\label{s3.7e6}
\Gamma(TM)\simeq \wedge^{0}\oplus\wedge_{T}^{1}
\end{equation}
given by $v\mapsto (i_{v}\eta, i_{v}\omega^{T})$ where $\omega^{T}=\frac{1}{2}d\eta$. 

\begin{prop}\label{s3.7p3}
Under the identification (\ref{s3.7e6}), 
the Lie algebra $\mathfrak{aut}(\eta, \psi)$ is given by 
$\{(f,-\frac{1}{2}df)\in \wedge^{0}_{B}\oplus\wedge^{1}_{B} \mid \Delta_{B}f=4\kappa f \}$ 
which is isomorphic to the eigenspace $\Ker (\Delta_{B}-4\kappa)$ of 
the basic Laplacian $\Delta_{B}$ on $\wedge^{0}_{B}$. 
\end{prop}
\begin{proof}
We introduce two operators $*_{T}$ and $*_{\mathbb{C}}$. 
We define an operator 
\begin{equation*}\label{s3.7e8}
*_{T}:\wedge_{T}^{p}\to\wedge_{T}^{2n-p}
\end{equation*}
by the formula 
\begin{equation*}\label{s3.7e9}
*\alpha=(*_{T}\alpha)\wedge \eta
\end{equation*}
for $\alpha\in\wedge_{T}^{p}$ where $*$ is an ordinary Hodge star operator 
with respect to the Riemannian metric $g$ on $M$. 
We can consider $*_{T}$ as an operator 
$\wedge_{T}^{p}\otimes \mathbb{C}\to\wedge_{T}^{2n-p}\otimes \mathbb{C}$ 
by the linearly extension. 
Let ${\rm vol}_{T}$ denote the transverse volume form with respect to the metric $g^{T}$. 
Then $i_{v}{\rm vol}_{T}=*_{T}v^{\sharp}$ for any $v\in\Gamma(Q)$ 
where $v^{\sharp}$ is the transverse $1$-form defined by $v^{\sharp}(w)=g^{T}(v,w)$ for $w\in \Gamma(TM)$. 
The equations ${\rm vol}_{T}=(\omega^{T})^{n}$ and $(Jv)^{\sharp}=i_{v}(\omega^{T})$ implies $i_{Jv}(\omega^{T})^{n}=*_{T}i_{v}\omega^{T}$. 
On the other hand, 
it follows from $i_{Jv}\psi=\sqrt{-1}i_{v}\psi$ that $i_{Jv}(\psi\wedge\overline{\psi})
=\sqrt{-1}(i_{v}\psi\wedge\overline{\psi}-(-1)^{n}\psi\wedge \overline{i_{v}\psi})$. 
It implies that 
\begin{equation}\label{s3.7e12}
\sqrt{-1}c_{n}^{-1}(i_{v}\psi\wedge\overline{\psi}-(-1)^{n}\psi\wedge \overline{i_{v}\psi})
=*_{T}i_{v}\omega^{T}.
\end{equation}
We remark that the equation (\ref{s3.7e12}) holds for any element $v$ of $TM$.  
We define $\overline{*}_{T}$ as $\overline{*}_{T}(\alpha)=\overline{*_{T}\alpha}$ 
for any $\alpha\in\wedge_{T}^{p}\otimes \mathbb{C}$. 
It induces the map $\overline{*}_{T}:\wedge_{T}^{p,q}\to\wedge_{T}^{n-p,n-q}$. 
By taking the $(n-1,n)$-part of the equation (\ref{s3.7e12}), 
we obtain 
\begin{equation}\label{s3.7e13}
\sqrt{-1}c_{n}^{-1}i_{v}\psi\wedge\overline{\psi}
=\overline{*}_{T}(i_{v}\omega^{T})^{1,0}
\end{equation}
for any $v\in TM$. 

We introduce an operator 
\begin{equation*}\label{s3.7e14}
*_{\mathbb{C}}:\wedge_{T}^{p,0}\to\wedge_{T}^{n-p,0}
\end{equation*}
given by the formula 
\begin{equation}\label{s3.7e15}
\overline{*}_{T}\alpha=\sqrt{-1}c_{n}^{-1}(*_{\mathbb{C}}\alpha)\wedge \overline{\psi}
\end{equation}
for $\alpha\in\wedge_{T}^{p,0}$. 
By taking the exterior derivative of the equation (\ref{s3.7e15}), 
we obtain that 
\begin{equation}\label{s3.7e16}
d\overline{*}_{T}\alpha=\sqrt{-1}c_{n}^{-1}(d-\sqrt{-1}\kappa\eta\wedge)(*_{\mathbb{C}}\alpha)\wedge \overline{\psi}
\end{equation}
for $\alpha\in\wedge_{T}^{p,0}$. 
If $\alpha$ is basic, 
then the left hand side of the equation (\ref{s3.7e16}) is the basic $(n-p+1,n)$-form 
$\partial_{B}\overline{*}_{B}\alpha\in\wedge_{B}^{n-p+1,n}$. 
Hence we obtain that 
\begin{eqnarray}
&&(d_{\xi}-\sqrt{-1}\kappa\eta\wedge)*_{\mathbb{C}}\alpha=0, \label{s3.7e17}\\
&&\partial_{T}(*_{\mathbb{C}}\alpha)\wedge \overline{\psi}
=-\sqrt{-1}c_{n}\partial_{B}\overline{*}_{B}\alpha \label{s3.7e17.5}
\end{eqnarray}
for $\alpha\in\wedge_{B}^{p,0}$. 

We start to compute $L_{v}\psi=0$ and $L_{v}\eta=0$. 
By using the operator $*_{\mathbb{C}}$, the equation (\ref{s3.7e13}) is written by $i_{v}\psi=*_{\mathbb{C}}(i_{v}\omega^{T})^{1,0}$
for any $v\in TM$. 
It turns out that 
\begin{eqnarray*}
L_{v}\eta&=&di_{v}\eta+i_{v}d\eta=df+2\theta\\
L_{v}\psi&=&di_{v}\psi+i_{v}d\psi \\
 &=&(d-\sqrt{-1}\kappa\eta\wedge)i_{v}\psi+\sqrt{-1}\kappa i_{v}\eta\psi\\
 &=&(d-\sqrt{-1}\kappa\eta\wedge)*_{\mathbb{C}}\theta^{1,0}
 +\sqrt{-1}\kappa f\psi \\
 &=&(\partial_{T}*_{\mathbb{C}}\theta^{1,0}+\sqrt{-1}\kappa f\psi)
 +\overline{\partial}_{T}*_{\mathbb{C}}\theta^{1,0}
 +(d_{\xi}-\sqrt{-1}\kappa\eta\wedge)*_{\mathbb{C}}\theta^{1,0}
\end{eqnarray*}
where $f=i_{v}\eta$ and $\theta=i_{v}\omega^{T}$. 

If $v$ satisfies $L_{v}\psi=0$ and $L_{v}\eta=0$, 
then $f$ and $\theta$ are basic since $\theta=-\frac{1}{2}df$ is transverse. 
It follows from the equation (\ref{s3.7e17.5}) that $(d_{\xi}-\sqrt{-1}\kappa\eta\wedge)*_{\mathbb{C}}\theta^{1,0}=0$. 
We have 
\begin{eqnarray*}
(\partial_{T}*_{\mathbb{C}}\theta^{1,0}+\sqrt{-1}\kappa f\psi)\wedge \overline{\psi} 
 &=&\partial_{T}*_{\mathbb{C}}\theta^{1,0}\wedge \overline{\psi}
 +\sqrt{-1}\kappa f\psi\wedge\overline{\psi} \\
 &=&-\sqrt{-1}c_{n}\partial_{B}\overline{*}_{B}\theta^{1,0}
 +\sqrt{-1}c_{n}\kappa f\overline{*}_{B}1 \\
 &=&\sqrt{-1}c_{n}\overline{*}_{B}(\partial_{B}^{*}\theta^{1,0}+\kappa f). 
\end{eqnarray*}
It implies that $\partial_{T}*_{\mathbb{C}}\theta^{1,0}+\sqrt{-1}\kappa f\psi
=*_{\mathbb{C}}(\partial_{B}^{*}\theta^{1,0}+\kappa f)$ by the definition of $*_{\mathbb{C}}$. 
Hence we obtain 
\begin{equation*}
L_{v}\psi=-\frac{1}{2}*_{\mathbb{C}}(\partial_{B}^{*}\partial_{B}f-2\kappa f)
-\frac{1}{2}\overline{\partial}_{T}*_{\mathbb{C}}\partial_{B}f
\end{equation*}
since $\theta^{1,0}=-\frac{1}{2}\partial_{B}f$. 
Thus $v$ satisfies $L_{v}\eta=0$ and $L_{v}\psi=0$ if and only if 
the corresponding $(f,\theta)$ satisfies  
\begin{eqnarray}
&&\theta=-\frac{1}{2}df \notag\\
&&\Box_{B}f=2\kappa f \label{s3.7e19.1}\\
&&\overline{\partial}_{T}*_{\mathbb{C}}\partial_{B}f=0 \label{s3.7e19.2}
\end{eqnarray}
where $\Box_{B}$ is the basic complex Laplace operator 
$\partial_{B}^{*}\partial_{B}+\partial_{B}\partial_{B}^{*}
:\wedge_{B}^{p,q}\to\wedge_{B}^{p,q}$. 

In order to see that the equation (\ref{s3.7e19.2}) is induced by (\ref{s3.7e19.1}), automatically, 
we identify $\wedge_{B}^{1,0}$ 
with $\wedge_{B}^{1,n}\otimes K_{D}^{-1}$ 
where $K_{D}$ is the basic canonical bundle $\wedge_{B}^{n,0}$. 
Let $\varpi$ be the basic canonical connection of $K_{D}$ and $\nabla_{\varpi}$ the covariant derivative on $K_{D}^{-1}$. 
The basic $2n$-form $(\omega^T)^{n}$ defines the metric of $K_{D}^{-1}$. 
Then we consider the Laplace operator 
$\Box^{\varpi}_{\partial}:\wedge_{B}^{p,q}\otimes K_{D}^{-1}\to
\wedge_{B}^{p,q}\otimes K_{D}^{-1}$ 
given by $\Box^{\varpi}_{\partial}=\nabla_{\varpi}^{1,0}(\nabla_{\varpi}^{1,0})^{*}+
(\nabla_{\varpi}^{1,0})^{*}\nabla_{\varpi}^{1,0}$ 
where $\nabla_{\varpi}^{1,0}$ means the basic $(1,0)$-part of $\nabla_{\varpi}$ 
and $(\nabla_{\varpi}^{1,0})^{*}$ is the dual operator of $\nabla_{\varpi}^{1,0}$. 
Then we see that 
\begin{equation*}
\Box^{\varpi}_{\partial}=\Box_{B}
\end{equation*}
since $\nabla_{\varpi}^{1,0}$ coincides with the operator $\partial_{B}$ 
under the identification $\wedge_{B}^{1,0}\simeq \wedge_{B}^{1,n}\otimes K_{D}^{-1}$. 
As the same manner, we define the operator 
$\Box^{\varpi}_{\overline{\partial}}:\wedge_{B}^{p,q}\otimes K_{D}^{-1}\to
\wedge_{B}^{p,q}\otimes K_{D}^{-1}$ 
given by $\Box^{\varpi}_{\overline{\partial}}=\nabla_{\varpi}^{0,1}(\nabla_{\varpi}^{0,1})^{*}+
(\nabla_{\varpi}^{0,1})^{*}\nabla_{\varpi}^{0,1}$ 
where $\nabla_{\varpi}^{0,1}$ means the $(0,1)$-part of $\nabla_{\varpi}$. 
Then we obtain 
\begin{equation*}
\Box^{\varpi}_{\partial}=\Box^{\varpi}_{\overline{\partial}}+2\kappa
\end{equation*}
by the Kodaira-Akizuki-Nakano identity on 
the basic vector bundle $\wedge_{B}^{p,q}\otimes K_{D}^{-1}$. 
If we assume $\Box_{B}f=2\kappa f$, 
then $\Box_{B}\partial_{B} f=2\kappa\partial_{B} f$. 
By considering $\partial_{B} f$ as a section of $\wedge_{B}^{1,n}\otimes K_{D}^{-1}$, 
we have  
\begin{equation*}
\Box^{\varpi}_{\overline{\partial}}\partial_{B} f
=(\Box^{\varpi}_{\partial}-2\kappa)\partial_{B} f=(\Box_{B}-2\kappa)\partial_{B} f=0.
\end{equation*}
It implies that 
\begin{equation}\label{s3.7e20}
(\nabla_{\varpi}^{0,1})^{*}\partial_{B} f=0.
\end{equation}
Let $U_{\alpha}$ be a local coordinate of $M$ and 
$\psi_{\alpha}$ a basic and transversely holomorphic local frame of $K_{D}$ over $U_{\alpha}$. 
Hence $\nabla_{\varpi}^{0,1}\psi_{\alpha}=0$. 
Then there exists a function $h_{\alpha}$ such that $
\psi=e^{h_{\alpha}}\psi_{\alpha}$ and $\overline{\partial}_{T}h_{\alpha}=0$. 
By considering $\partial_{B}f$ as 
$\partial_{B} f\wedge \overline{\psi}_{\alpha}
\otimes \overline{\psi_{\alpha}^{-1}}=\partial_{B} f\wedge \overline{\psi}
\otimes e^{-\overline{h}_{\alpha}}\overline{\psi_{\alpha}^{-1}}$, we obtain  
\begin{eqnarray*}
\nabla_{\varpi}^{0,1}\overline{*}_{B}\partial_{B}f
 &=&\nabla_{\varpi}^{0,1}\overline{*}_{T}(\partial_{B}f
 \wedge \overline{\psi}\otimes e^{-\overline{h}_{\alpha}}\overline{\psi_{\alpha}^{-1}}) \\
 &=&\overline{\partial}_{T}\overline{*}_{T}(\partial_{B} f
 \wedge \overline{\psi})\otimes e^{-h_{\alpha}}\psi_{\alpha}^{-1}
 +\overline{*}_{T}(\partial_{B} f \wedge \overline{\psi})\otimes 
 \nabla_{\varpi}^{0,1}(e^{-h_{\alpha}}\psi_{\alpha}^{-1}) \\
 &=&\sqrt{-1}\, 2^{n}c_{n}^{-1}( \overline{\partial}_{T}\overline{*}_{\mathbb{C}}\partial_{B}f )
 \otimes e^{-h_{\alpha}}\psi_{\alpha}^{-1}
\end{eqnarray*}
where the last equation is induced by 
$*_{\mathbb{C}}\alpha=-\sqrt{-1}\, 2^{-n}c_{n}\overline{*}_{T}(\alpha\wedge \overline{\psi})$ 
for $\alpha\in\wedge_{T}^{p,0}$. 
Thus the equation (\ref{s3.7e20}) implies $\overline{\partial}_{T}\overline{*}_{\mathbb{C}}\partial_{B}f=0$. 
Hence (\ref{s3.7e19.1}) implies (\ref{s3.7e19.2}). 

Under the identification (\ref{s3.7e6}), an element $v\in \mathfrak{aut}(\eta, \psi)$ corresponds to 
$(f,\theta)$ such that $\theta=-\frac{1}{2}df$ and $\Box_{B}f=2\kappa f$ :
\begin{eqnarray*}
\mathfrak{aut}(\eta, \psi)
 &=&\{(f,\theta)\in \wedge^{0}\oplus\wedge^{1}_{T} 
 \mid \theta=-\frac{1}{2}df,\ \Box_{B}f=2\kappa f \} \\
 &=&\{(f,-\frac{1}{2}df)\in \wedge^{0}_{B}\oplus\wedge^{1}_{B}
 \mid \Delta_{B}f=4\kappa f \},
\end{eqnarray*}
and hence it completes the proof. 
\end{proof}
\begin{rem}
{\rm 
Futaki, Ono and Wang introduced a Hamiltonian holomorphic vector field on Sasaki manifolds 
in order to consider an obstruction to the existence of transverse K\"{a}hler-Einstein metric~\cite{FOW}. 
They showed that the complex vector space of normalized Hamiltonian holomorphic vector fields is isomorphic to $\Ker (\Box_{B}-2(n+1))$ on a Sasaki-Einstein manifold $M$. 
Hence $\mathfrak{aut}(\eta, \psi)$ is isomorphic to 
the real part of the space of normalized Hamiltonian holomorphic vector fields on $M$. 
}
\end{rem}

Proposition~\ref{s3.7p2} and Proposition~\ref{s3.7p3} yield the following  
\begin{cor}
\begin{enumerate}
\item[(i)] If $\kappa>0$, 
then $\mathfrak{aut}(\eta,\psi)\oplus \langle\xi\rangle_{\mathbb{R}}=\mathfrak{aut}(\eta, \xi, \Psi, g)$. 
\item[(ii)] If $\kappa<0$, 
then $\mathfrak{aut}(\eta,\psi)=\{0\}$ and $\mathfrak{aut}(\eta, \xi, \Psi, g)=\langle\xi\rangle_{\mathbb{R}}$.
\item[(iii)] If $\kappa=0$, 
then $\mathfrak{aut}(\eta,\psi)=\langle\xi\rangle_{\mathbb{R}}$. $\hfill\Box$
\end{enumerate}
\end{cor}

\begin{rem}
{\rm 
In the case $\kappa\le 0$, it is known that $\mathfrak{aut}(\eta, \xi, \Psi, g)=\langle\xi\rangle_{\mathbb{R}}$ in Theorem 8.1.14~\cite{BG}. 
Hence we can see that $\mathfrak{aut}(\eta,\psi)=\langle\xi\rangle_{\mathbb{R}}=\mathfrak{aut}(\eta, \xi, \Psi, g)$ if $\kappa=0$. 
}
\end{rem}

\subsection{Special Legendrian submanifolds in Sasaki manifolds with almost transverse Calabi-Yau structures}
Let $(M, g)$ be a Sasaki manifold with a Sasaki structure $(\xi, \eta, \Psi, g)$. 
We assume that there exists a complex valued $n$-form $\psi$ 
such that $(\psi, \frac{1}{2}d\eta)$ is an almost transverse Calabi-Yau structure on $(M, \mathcal{F}_{\xi})$ 
with $d\psi=\sqrt{-1}\,\kappa\eta\wedge\psi$ for a real number $\kappa$. 
Then $(\psi, \frac{1}{2}d\eta)$ induces an almost Calabi-Yau structure $(\Omega, \omega)$ 
on the metric cone $(C(M), \overline{g})$ as in Proposition~\ref{s5.1p1}. 
An $(n+1)$-dimensional submanifold $\widetilde{X}$ in $C(M)$ is called {\it a special Lagrangian submanifold} 
if $\tilde{\iota}^{*}\Omega^{\rm Re}=\tilde{\iota}^{*}\omega=0$ where $\tilde{\iota}$ is the embedding $\tilde{\iota}:\widetilde{X}\hookrightarrow C(M)$. 
We consider such submanifolds of cone type. 
\begin{defi}\label{s5.3d1}
{\rm 
A submanifold $X$ in $M$ is {\it special Legendrian} 
if the cone $C(X)$ is a special Lagrangian submanifold in $C(M)$. 
}
\end{defi}
Then we obtain 
\begin{prop}\label{s5.3p1}
Any special Legendrian submanifold is a minimal submanifold in $(M,g)$. 
\end{prop}
\begin{proof}
Let $X$ be a special Legendrian submanifold. 
Then the special Lagrangian cone $C(X)$ is a calibrated submanifold with respect to 
the metric $r^{\kappa^{\prime}}\overline{g}$ where $\kappa^{\prime}=2(\frac{\kappa}{n+1}-1)$. 
In fact, the holomorphic $(n+1)$-form $\Omega$ induces the calibration $\Omega^{\rm Re}$ with respect to the metric $r^{\kappa^{\prime}}\overline{g}$. 
Then special Lagrangian submanifolds are given by calibrated submanifolds. 
Therefore $C(X)$ is minimal with respect to $r^{\kappa^{\prime}}\overline{g}$ and the mean curvature vector field $\widetilde{H}^{\prime}$ vanishes. 
Let $\widetilde{H}$ be the mean curvature vector field of $C(X)$ with respect to $\overline{g}$. 
Then $\widetilde{H}$ also vanishes since $\widetilde{H}^{\prime}=r^{-\kappa^{\prime}}\widetilde{H}$. 
Hence $C(X)$ is minimal with respect to $\overline{g}$, and $X$ is a minimal submanifold with respect to $g$. 
\end{proof}

We provide a characterization of special Legendrian submanifolds : 

\begin{prop}\label{s5.3p2}
An $n$-dimensional submanifold $X$ in $M$ is a special Legendrian submanifold 
if and only if $\iota^{*}\psi^{\rm Im}=\iota^{*}\eta=0$. 
\end{prop}
\begin{proof}
By repeating the argument of the proof of Proposition~\ref{s4.1.1p2}, 
we can show that the embedding $\tilde{\iota}:C(X)\hookrightarrow C(M)$ is special Lagrangian if and only if $\tilde{\iota}^*\Omega^{\rm Re}=\tilde{\iota}^*\eta=0$. 
The condition $\tilde{\iota}^*\Omega^{\rm Re}=\tilde{\iota}^*\eta=0$ is equivalent to $\iota^{*}\psi^{\rm Im}=\iota^{*}\eta=0$ 
since $\Omega=(\frac{dr}{r}+\sqrt{-1}\eta)\wedge r^{\kappa}\psi$. 
Hence $X$ is a special Legendrian submanifold if and only if $\iota^{*}\psi^{\rm Im}=\iota^{*}\eta=0$. 
\end{proof}

Let $X$ be a compact connected special Legendrian submanifold and 
$\mathcal{M}_{X}$ the moduli space of special Legendrian deformations of $X$. 

\begin{thm}\label{s5.3p3}
The infinitesimal deformation space of $X$ is 
isomorphic to the eigenspace ${\rm Ker}(\Delta_{0}-2\kappa)$ of $\Delta_{0}$ with eigenvalue $2\kappa$. 
If $\kappa=0$, then $X$ is rigid and $\mathcal{M}_{X}$ is a $1$-dimensional manifold. 
If $\kappa<0$, then $X$ does not have any non-trivial deformation and $\mathcal{M}_{X}=\{0_X\}$. 
\end{thm}
\begin{proof}
Proposition~\ref{s5.3p2} implies that $\mathcal{M}_{X}$ is 
the moduli space $\mathcal{M}_{X}(\psi^{\rm Im},\eta)$ of $(\psi^{\rm Im},\eta)$-deformations of $X$. 
We take the set $U$ as in $(\ref{s2.1e2})$ and define the map $F:U\to \wedge^{0}\oplus \wedge^{1}$ by 
\begin{equation*}\label{s5.3e1}
F(v)=(*\exp_{v}^{*}\psi^{\rm Im},\exp_{v}^{*}\eta)
\end{equation*}
for $v\in U$. 
We can regard $F^{-1}(0)$ as a set of special Legendrian submanifolds 
in $M$ which is close to $X$ in $C^1$ sense. . 
Then we have 
\begin{eqnarray*}\label{s5.3e2}
d_{0}F(v)&=&(*\iota^{*}L_{\tilde{v}}\psi^{\rm Im}, \iota^{*}L_{\tilde{v}}\eta)\\
&=&(*d\iota^{*}(i_{\tilde{v}}\psi^{\rm Im})+\kappa\iota^{*}(i_{\tilde{v}}\eta\psi^{\rm Re}), 
d\iota^{*}(i_{\tilde{v}}\eta)+\iota^{*}(i_{\tilde{v}}d\eta))\\
&=&(\frac{1}{2}d^{*}\iota^{*}(i_{\tilde{v}}d\eta)+\kappa\iota^{*}(i_{\tilde{v}}\eta), 
d\iota^{*}(i_{\tilde{v}}\eta)+\iota^{*}(i_{\tilde{v}}d\eta))
\end{eqnarray*}
for $v\in U$ where $\tilde{v}$ is an extension of $v$ to $M$. 
In the last equation, we use that 
$\iota^{*}(i_{\tilde{v}}\psi^{\rm Im})=-\frac{1}{2}*\iota^{*}(i_{\tilde{v}}d\eta)$. 
Under the identification 
\[
\Gamma(NX)\simeq \wedge^{0}\oplus\wedge^{1}
\]
given by $v\mapsto (i_{v}\eta, \frac{1}{2}i_{v}d\eta)$, 
we can consider $d_{0}F$ as the map 
$d_{0}F:\wedge^{0}\oplus\wedge^{1}
\to\wedge^{0}\oplus \wedge^{1}$ defined by  
\begin{equation*}\label{s5.3e4}
d_{0}F(f,\alpha)=(d^{*}\alpha+\kappa f, df+2\alpha)
\end{equation*}
for $(f,\alpha)\in\wedge^{0}\oplus\wedge^{1}$. 
Then it turns out that 
\begin{eqnarray*}
{\rm Ker}(d_{0}F)&=&\{(f,\alpha)\in\wedge^{0}\oplus\wedge^{1} 
\mid d^{*}\alpha+\kappa f=0, 2\alpha+df=0\} \\
&=&\{(f,-\frac{1}{2}df)\in\wedge^{0}\oplus\wedge^{1} \mid (\Delta_{0}-2\kappa)f=0 \} 
\end{eqnarray*}
and hence the infinitesimal deformation space of $X$ is  
isomorphic to the space ${\rm Ker}(\Delta_{0}-2\kappa)$. 
In the case of $\kappa=0$, we already proved that the moduli space $\mathcal{M}_{X}$ 
is a $1$-dimensional smooth manifold in Proposition~\ref{s2.4e1}. 
Then the $1$-dimensional deformation space is given by 
$\langle\xi\rangle_{\mathbb{R}}=\mathfrak{aut}(\eta,\psi)$. Thus $X$ is rigid. 
If $\kappa<0$, then ${\rm Ker}(\Delta_{0}-2\kappa)=\{0\}$ and $X$ has only the trivial deformation. 
Hence we finish the proof. 
\end{proof}
Let $\omega^{T}$ be the transverse $2$-form $\frac{1}{2}d\eta$ on $M$. 
We denote by $\mathcal{N}_{X}$ the moduli space $\mathcal{M}_{X}(\psi^{\rm Im}, \omega^{T})$ of 
$(\psi^{\rm Im}, \omega^{T})$-deformations of $X$. 
We fix an integer $s\ge 3$ and a real number $\alpha$ with $0<\alpha<1$ and set $\mathcal{N}_{X}^{s,\alpha}$ 
as the moduli space $\mathcal{M}_{X}^{s,\alpha}(\psi^{\rm Im}, \omega^{T})$ of $(\psi^{\rm Im}, \omega^{T})$-deformations of $C^{s, \alpha}$-class. 

\begin{prop}\label{s5.3p4}
The moduli space $\mathcal{N}_{X}^{s,\alpha}$ is smooth at $0_X$. 
If $\kappa\neq 0$, then the tangent space $T_{0_X}\mathcal{N}_{X}^{s,\alpha}$ 
is isomorphic to 
$\{(-\frac{1}{\kappa}d^{*}\alpha,\alpha)\in C^{s,\alpha}(\wedge^{0}\oplus\wedge^{1}) \mid d\alpha=0 \}$. 
If $\kappa= 0$, then $T_{0_X}\mathcal{N}_{X}^{s,\alpha}$ 
is isomorphic to $C^{s,\alpha}(\wedge^{0})\oplus H^{1}(X)$. 
\end{prop}
\begin{proof}
We take the set $U$ as in $(\ref{s2.1e2})$ and define the map $G:U\to \wedge^{0}\oplus \wedge^{2}$ by 
\begin{equation*}
G(v)=(*\exp_{v}^{*}\psi^{\rm Im}, \exp_{v}^{*}\omega^{T})
\end{equation*}
for $v\in U$. It follows that 
\begin{eqnarray*}
d_{0}G(v)&=&(*(d\iota^{*}(i_{\tilde{v}}\psi^{\rm Im})+
(n+1)\iota^{*}(i_{\tilde{v}}\eta\psi^{\rm Re})-\kappa\iota^{*}(\eta\wedge i_{\tilde{v}}\psi^{\rm Re}),\  
d\iota^{*}(i_{\tilde{v}}\omega^{T})) \\
&=&(d^{*}\iota^{*}(i_{\tilde{v}}\omega^{T})+
\kappa\iota^{*}(i_{\tilde{v}}\eta),\ d\iota^{*}(i_{\tilde{v}}\omega^{T}))
\end{eqnarray*}
for $v\in U$ where $\tilde{v}$ is an extension of $v$ to $M$.  
In the last equation, we use that 
$\iota^{*}(i_{\tilde{v}}\psi^{\rm Im})=-\frac{1}{2}*\iota^{*}(i_{\tilde{v}}d\eta)$ and 
$\iota^{*}\psi^{\rm Re}={\rm vol}(X)$. 
Under the identification $\Gamma(NX)\simeq \wedge^{0}\oplus\wedge^{1}$ 
given by $v\mapsto (i_{v}\eta, i_{v}\omega^{T})$, we identify $d_{0}G$ with the map 
$D_1:\wedge^{0}\oplus\wedge^{1}
\to\wedge^{0}\oplus \wedge^{2}$ defined by  
\begin{equation*}
D_1(f,\alpha)=(d^{*}\alpha+\kappa f, d\alpha) 
\end{equation*}
for $(f,\alpha)\in\wedge^0\oplus \wedge^1$. 
Now we provide a complex as follows 
\begin{equation*}
0\to \wedge^{0}\oplus\wedge^{1}\stackrel{D_{1}}{\longrightarrow}
\wedge^{0}\oplus \wedge^{2}\stackrel{D_{2}}{\longrightarrow} 
\wedge^{3}\to 0 
\end{equation*}
where the operator $D_{2}$ is given by 
\begin{equation*}
D_{2}(f,\beta)=d\beta
\end{equation*}
for $(f,\beta)\in\wedge^0\oplus \wedge^2$. It is easy to see that  
\begin{eqnarray*}
P_{1}(f,\alpha)&=&(\kappa^{2}f+\kappa d^{*}\alpha, \Delta_{1}\alpha+\kappa df),\\
P_{2}(f,\beta)&=&((\Delta_{0}+\kappa^{2})f, \Delta_{2}\beta).
\end{eqnarray*}
Hence $P_{2}$ is the elliptic operator. 
It follows from ${\rm Im}(G)\subset \wedge^{0}\oplus d\wedge^{1}$  
that ${\rm Im}(G)$ is perpendicular 
to $\Ker{P_{2}}(=\{0\}\oplus \mathcal{H}^{2}(X))$ and ${\rm Im}(D^{*}_{2})$. 
Hence we obtain ${\rm Im}(G)\subset {\rm Im}(D_{1})$ by the Hodge decomposition 
$\wedge^{0}\oplus \wedge^{2}=\Ker{P_{2}}\oplus {\rm Im}(D_{1})\oplus {\rm Im}(D_{2}^{*})$. 
Proposition~\ref{s2.1p1} implies that $\mathcal{N}_{X}^{s,\alpha}$ is smooth at $0_X$ 
with the tangent space ${\rm Ker}(D_{1}^{s,\alpha})$. 
If $\kappa\neq 0$, then it turns out that 
\begin{eqnarray*}
{\rm Ker}(D_{1})&=&\{(f,\alpha)\in\wedge^{0}\oplus\wedge^{1} 
\mid d^{*}\alpha+\kappa f=0, d\alpha=0\} \\
&=&\{(-\frac{1}{\kappa}d^{*}\alpha,\alpha)\in\wedge^{0}\oplus\wedge^{1} \mid d\alpha=0 \}. 
\end{eqnarray*}
If $\kappa=0$, then 
\begin{eqnarray*}
{\rm Ker}(D_{1})&=&\{(f,\alpha)\in\wedge^{0}\oplus\wedge^{1} 
\mid d^{*}\alpha=d\alpha=0\} \\
&=&\{(f,\alpha)\in\wedge^{0}\oplus\wedge^{1} \mid \Delta_{1}\alpha=0 \} \\
&=&\wedge^{0}\oplus\mathcal{H}^{1}(X). 
\end{eqnarray*}
Hence we finish the proof. 
\end{proof}

Let $\mathcal{L}_{X}$ be the moduli space of Legendrian deformations of $X$ of $C^{\infty}$-class 
and $\mathcal{L}_{X}^{s,\alpha}$ that of $C^{s, \alpha}$-class. 
The following is a generalization of Theorem \ref{s4.3t1}. 
\begin{thm}\label{s5.3t1}
The moduli space $\mathcal{M}_{X}$ is 
the intersection $\mathcal{N}_{X}\cap\mathcal{L}_{X}$ 
where $\mathcal{N}_{X}^{s,\alpha}$ and $\mathcal{L}_{X}^{s,\alpha}$ are smooth. 
\end{thm} 
\begin{proof}
The moduli space $\mathcal{M}_{X}$ is the intersection $\mathcal{N}_{X}\cap \mathcal{L}_{X}$
since $\mathcal{M}_{X}(\psi^{\rm Im}, \eta, \omega^{T})=\mathcal{M}_{X}(\psi^{\rm Im}, \omega^{T})\cap\mathcal{M}_{X}(\eta)$. 
The space $\mathcal{L}_{X}^{s,\alpha}$ is smooth by Proposition~\ref{s2.2.4p1}. 
It follows from Proposition~\ref{s5.3p4} that $\mathcal{N}_{X}^{s,\alpha}$ is smooth at $\mathcal{N}_{X}\cap \mathcal{L}_{X}$. 
Hence it completes the proof. 
\end{proof}

Let $\mathcal{N}_{X}^{T}$ denote the moduli space $\mathcal{M}_{X}^{T}(\psi^{\rm Im}, \omega^{T})$ of 
transverse $(\psi^{\rm Im}, \omega^{T})$-deformations of $X$ of $C^{\infty}$-class. 
Then we obtain 
\begin{thm}\label{s5.3t2}
The moduli space $\mathcal{N}_{X}^{T}$ is smooth at $0_X$ and 
the tangent space $T_{0_X}\mathcal{N}_{X}^{T}$ is isomorphic to $H^{1}(X)$. 
\end{thm}
\begin{proof}
We define $NX^{T}$ by the vector bundle $NX^T=NX\cap \ker \eta|_{X}$ on $X$. 
Let $\mathcal{U}^{T}$ be a sufficiently small neighbourhood of the zero section in $NX^{T}$ and 
$U^T$ the set $\{ v\in \Gamma(NX^{T}) \mid v_{x}\in \mathcal{U}^T, x\in X \}$. 
We define the map $G^T:U^{T}\to \wedge^{0}\oplus \wedge^{2}$ by 
\begin{equation*}
G^{T}(v)=(*\exp_{v}^{*}\psi^{\rm Im}, \exp_{v}^{*}\omega^{T}) 
\end{equation*}
for $v\in U^T$. Then we have 
\begin{eqnarray*}
d_{0}G^{T}(v)&=&(*(d\iota^{*}(i_{\tilde{v}}\psi^{\rm Im})+
\kappa\iota^{*}(i_{\tilde{v}}\eta\psi^{\rm Re})-\kappa\iota^{*}(\eta\wedge i_{\tilde{v}}\psi^{\rm Re})),\  
d\iota^{*}(i_{\tilde{v}}\omega^{T})) \\
&=&(d^{*}\iota^{*}(i_{\tilde{v}}\omega^{T}),\ d\iota^{*}(i_{\tilde{v}}\omega^{T}))
\end{eqnarray*}
for $v\in U^{T}$ where $\tilde{v}$ is an extension of $v$ to $M$.  
In the last equation, we use that 
$\iota^{*}(i_{\tilde{v}}\psi^{\rm Im})=-*\iota^{*}(i_{\tilde{v}}\omega^{T})$ and 
$\iota^{*}(i_{\tilde{v}}\eta)=i_{v}\eta=0$. 
Under the identification 
\[
\Gamma(NX^{T})\simeq \wedge^{1}
\]
given by $v\mapsto i_{v}\omega^{T}$, we can consider $d_{0}G^{T}$ as the map 
$D_1:\wedge^{1}\to\wedge^{0}\oplus \wedge^{2}$ given by  
\begin{equation*}
D_1(\alpha)=(d^{*}\alpha, d\alpha) 
\end{equation*}
for $\alpha\in \wedge^{1}$. Then it turns out that 
\begin{equation*}
{\rm Ker}(D_{1})
=\{\alpha\in \wedge^{1} \mid d^{*}\alpha=d\alpha=0\} 
=\mathcal{H}^{1}(X).
\end{equation*}
Now we provide a complex as follows 
\begin{equation*}
0\to \wedge^{1}\stackrel{D_{1}}{\longrightarrow}
\wedge^{0}\oplus \wedge^{2}\stackrel{D_{2}}{\longrightarrow} 
\wedge^{3}\to 0 
\end{equation*}
where the operator $D_{2}$ is given by 
\begin{equation*}
D_{2}(f,\beta)=d\beta
\end{equation*}
for $(f,\beta)\in\wedge^0\oplus \wedge^2$. It is easy to see that  
\begin{eqnarray*}
P_{1}(\alpha)&=&\Delta_{1}\alpha,\\
P_{2}(f,\beta)&=&(\Delta_{0}f, \Delta_{2}\beta).
\end{eqnarray*}
Hence $P_{1}$ and $P_{2}$ are elliptic. 
It follows from ${\rm Im}(G^{T})\subset d^{*}\!\wedge^{1}\oplus d\wedge^{1}$  
that ${\rm Im}(G^{T})$ is perpendicular to 
$\Ker{P_{2}}(=\mathcal{H}^{0}(X)\oplus \mathcal{H}^{2}(X))$ and ${\rm Im}(D^{*}_{2})$. 
Hence we obtain ${\rm Im}(G^{T})\subset {\rm Im}(D_{1})$ by the Hodge decomposition 
$\wedge^{0}\oplus \wedge^{2}=\Ker{P_{2}}\oplus {\rm Im}(D_{1})\oplus {\rm Im}(D_{2}^{*})$. 
Proposition~\ref{s2.1p1} implies that $\mathcal{N}_{X}^{T}$ is smooth at $0_X$ 
with the tangent space ${\rm Ker}(D_{1})=\mathcal{H}^{1}(X)$, 
and hence it completes the proof.  
\end{proof}

\subsection{Minimal Legendrian deformations of special Legendrian submanifolds}
We assume that there exists an almost transverse Calabi-Yau structure $(\psi, \frac{1}{2}d\eta)$ on $M$ 
with $d\psi=\sqrt{-1}\,\kappa\eta\wedge\psi$ for a real number $\kappa$. 
Let $(\Omega, \omega)$ be the corresponding almost Calabi-Yau structure on $C(M)$. 
For a real constant $\theta$, the phase change $\Omega\to e^{\sqrt{-1}\,\theta}\Omega$ induces 
the almost Calabi-Yau structure $(e^{\sqrt{-1}\,\theta}\Omega, \omega)$ on $C(M)$. 
An $(n+1)$-dimensional submanifold $\widetilde{X}$ in $C(M)$ is called 
{\it a $\theta$-special Lagrangian submanifold} for a phase $\theta$ 
if $\tilde{\iota}^{*}(e^{\sqrt{-1}\,\theta}\Omega)^{\rm Im}=\tilde{\iota}^{*}\omega=0$. 
Hence any $\theta$-special Lagrangian submanifold is a special Lagrangian submanifold 
with respect to the almost Calabi-Yau structure $(e^{\sqrt{-1}\,\theta}\Omega, \omega)$. 
\begin{defi}\label{s5.4d1}
{\rm 
A submanifold $X$ in $M$ is {\it $\theta$-special Legendrian} 
if the cone $C(X)$ is a $\theta$-special Lagrangian submanifold in $C(M)$ for a phase $\theta$. 
}
\end{defi}
Any $\theta$-special Legendrian submanifold is a special Legendrian submanifold 
in the Sasaki manifold $(M, g)$ with the almost transverse Calabi-Yau structure $(e^{\sqrt{-1}\,\theta}\psi, \frac{1}{2}d\eta)$. 
Hence any $\theta$-special Legendrian submanifold is minimal. Moreover the converse is true in the case $\kappa>0$. 
In order to see it, we consider the relation between a mean curvature vector field and a phase of $(\psi, \frac{1}{2}d\eta)$. 
From now on, we assume that $\kappa>0$. 
Let $X$ be a connected oriented Legendrian submanifold in $M$ and $H$ the mean curvature vector field of $X$ with respect to $g$. 
Then there exists a $\mathbb{R}/\mathbb{Z}$-valued function $\theta$ on $X$ such that $*\iota^{*}\psi=e^{-\sqrt{-1}\,\theta}$ 
where $*$ is the Hodge operator with respect to the metric $\iota^*g$ on $X$. 
\begin{lem}\label{s5.4l1}
Let $\omega^{T}$ be a 2-form $\frac{1}{2}d\eta$. 
Then $d\theta=\iota^{*}(i_{H}\omega^{T})$. 
\end{lem}
\begin{proof}
We take $a=\frac{\kappa}{n+1}$ and define $\psi_{a}$ and $\eta_{a}$ as $a^{\frac{n}{2}}\psi$ and $a\eta$, respectively. 
Then $(\psi_{a}, \frac{1}{2}d\eta_{a})$ is an almost transverse Calabi-Yau structure 
with respect to the metric $g_{a}=ag$ such that $d\psi_{a}=(n+1)\sqrt{-1}\,\eta_{a}\wedge \psi_{a}$ 
and $*_{g_a}\iota^{*}\psi_a=e^{-\sqrt{-1}\,\theta}$ where $*_{g_a}$ is the Hodge operator with respect to $\iota^*g_a$ on $X$. 
Let $H_a$ be the mean curvature vector field of $X$ with respect to $g_{a}$. 
Then it follows from the equation (\ref{s4.1e2}) in Proposition~\ref{s3.5p2} that $d\theta=\iota^{*}(i_{H_a}\omega_a^{T})$ where $\omega_a^{T}=\frac{1}{2}d\eta_a$. 
Hence the equation $H_a=\frac{1}{a}H$ implies that $d\theta=\iota^{*}(i_{\frac{1}{a}H}a\omega^{T})=\iota^{*}(i_{H}\omega^{T})$. 
\end{proof}

\begin{prop}\label{s5.4p1}
Let $X$ be a connected oriented $n$-dimensional submanifold in $M$. 
Then the following conditions are equivalent. 
\begin{enumerate}
\item[(i)] $X$ is minimal Legendrian. 

\item[(ii)] $X$ is $\theta$-special Legendrian for a phase $\theta$. 

\item[(iii)] $d*\iota^{*}\psi^{\rm Im}=\iota^{*}\eta=0$. 
\end{enumerate}
\end{prop}
\begin{proof}
By Proposition~\ref{s5.3p1}, (ii) implies (i). 
We may assume that $X$ is Legendrian and there exists a $\mathbb{R}/\mathbb{Z}$-valued function $\theta$ on $X$ such that $*\iota^{*}\psi=e^{-\sqrt{-1}\,\theta}$. 
If the condition (i) holds, then it follows from Lemma~\ref{s5.4l1} that $\theta$ is constant and $d*\iota^{*}\psi^{\rm Im}=d(*\iota^{*}\psi)^{\rm Im}=0$. 
Hence (iii) holds. 
If we assume the condition (iii), then the imaginary part of $e^{-\sqrt{-1}\,\theta}=*\iota^{*}\psi$ is constant, and $e^{-\sqrt{-1}\,\theta}$ is also constant. 
Hence $*\iota^{*}(e^{\sqrt{-1}\,\theta}\psi)=e^{\sqrt{-1}\,\theta}*\iota^{*}\psi=1$ and $\iota^{*}(e^{\sqrt{-1}\,\theta}\psi)^{\rm Im}=0$. 
It yields that $X$ is a special Legendrian submanifold in the Sasaki manifold with the almost transverse Calabi-Yau structure $(e^{\sqrt{-1}\,\theta}\psi, \frac{1}{2}d\eta)$. 
Therefore the condition (ii) holds. Hence we finish the proof. 
\end{proof}

Let $X$ be a compact connected special Legendrian submanifold in $M$. 
We consider minimal Legendrian deformations of $X$. 
\begin{prop}\label{s5.4p2}
Then the infinitesimal deformation space of $X$ as a minimal submanifold is  
isomorphic to ${\rm Ker}(\Delta_{0}-2\kappa)\oplus \mathbb{R}$. 
\end{prop}
\begin{proof}
We take the set $U$ as in $(\ref{s2.1e2})$ and define the map $F:U\to \wedge^{1}\oplus \wedge^{1}$ by 
\begin{equation*}\label{s5.4e1}
F(v)=(d*\exp_{v}^{*}\psi^{\rm Im},\exp_{v}^{*}\eta)
\end{equation*}
for $v\in U$. 
Then Proposition~\ref{s5.4p1} implies that $F^{-1}(0)$ is regarded as a set of minimal Legendrian submanifolds 
in $M$ which is close to $X$. 
It follows that 
\begin{eqnarray*}
d_{0}F(v)&=&(d*\iota^{*}L_{\tilde{v}}\psi^{\rm Im}, \iota^{*}L_{\tilde{v}}\eta)\\
&=&(d(*d\iota^{*}(i_{\tilde{v}}\psi^{\rm Im})+\kappa *\iota^{*}(i_{\tilde{v}}\eta\psi^{\rm Re})), 
d\iota^{*}(i_{\tilde{v}}\eta)+\iota^{*}(i_{\tilde{v}}d\eta))\\
&=&(d(d^{*}\iota^{*}(\frac{1}{2}i_{\tilde{v}}d\eta)+\kappa \iota^{*}(i_{\tilde{v}}\eta)), 
d\iota^{*}(i_{\tilde{v}}\eta)+\iota^{*}(i_{\tilde{v}}d\eta))
\end{eqnarray*}
for $v\in U$ where $\tilde{v}$ is an extension of $v$ to $M$. 
Under the identification $\Gamma(NX)\simeq \wedge^{0}\oplus\wedge^{1}$ given by $v\mapsto (i_{v}\eta, \frac{1}{2}i_{v}d\eta)$, 
we identify $d_{0}F$ with the map 
$d_{0}F:\wedge^{0}\oplus\wedge^{1}
\to\wedge^{1}\oplus \wedge^{1}$ defined by 
\begin{equation*}\label{s5.4e5}
d_{0}F(f,\alpha)=(d(d^{*}\alpha+\kappa f), df+2\alpha)
\end{equation*}
for $(f,\alpha)\in\wedge^{0}\oplus\wedge^{1}$. 
Then it turns out that 
\begin{eqnarray*}
{\rm Ker}(d_{0}F)&=&\{(f,\alpha)\in\wedge^{0}\oplus\wedge^{1} 
\mid d(d^{*}\alpha+\kappa f)=0, 2\alpha+df=0\} \\
&=&\{(f,-\frac{1}{2}df)\in\wedge^{0}\oplus\wedge^{1} \mid d(\Delta_{0}-2\kappa )f=0 \}.
\end{eqnarray*}
Since ${\rm Ker}(d(\Delta_{0}-2\kappa ))={\rm Ker}(\Delta_{0}-2\kappa )\oplus{\rm Ker}(d|_{\wedge^{0}})$ for $\kappa\neq 0$, 
the space ${\rm Ker}(d_{0}F)$ is isomorphic to the space ${\rm Ker}(\Delta_{0}-2\kappa)\oplus\mathbb{R}$. 
Hence it completes the proof. 
\end{proof}
Proposition~\ref{s5.4p2} was shown by Ohnita from the point of view of 
minimal Legendrian submanifolds in $\eta$-Sasaki-Einstein manifolds~\cite{O1}. 
In the same paper, he also provided some examples of rigid minimal Legendrian submanifolds 
in the standard $7$-sphere $S^7$ and the real Steifel manifold $V_2(\mathbb{R}^5)$ 
and a non-rigid minimal Legendrian submanifold in $S^7$ (he said that any deformation of $X$ was trivial instead that $X$ was rigid). 
We can consider another rigidity condition for special Legendrian submanifolds. 
We recall that a special Legendrian submanifold $X$ is \textit{rigid} if 
any special Legendrian deformation of $X$ is induced by the automorphism group ${\rm Aut}(\eta, \psi)$ of diffeomorphisms of $M$ preserving $\eta$ and $\psi$. 
Any element $f$ of ${\rm Aut}(\eta, \psi)$ induces the diffeomorphism $\widetilde{f}$ of $C(M)$ by $\widetilde{f}(r, x)=f(x)$ for $(r, x)\in \mathbb{R}_{>0}\times M=C(M)$. 
Then $\widetilde{f}$ is the element of the group ${\rm Aut}(\Omega, \omega, r)$ of diffeomorphisms of $C(M)$ preserving $\Omega$, $\omega$ and $r$. 
Conversely, ${\rm Aut}(\Omega, \omega, r)$ induces ${\rm Aut}(\eta, \psi)$ by the restriction to the hypersurface $M=\{r=1\}$. 
Under the correspondence, a special Legendrian submanifold $X$ is rigid if and only if 
any deformations of the special Lagrangian cone $C(X)$ is induced by the group ${\rm Aut}(\Omega, \omega, r)$. 
Hence there exist two rigidity conditions for special Legendrian submanifolds. We show that these conditions are equivalent : 
\begin{thm}\label{s5.4t1}
Let $X$ be a compact connected special Legendrian submanifold. 
Then $X$ is \textit{rigid} as a special Legendrian submanifold if and only if 
it is rigid as a minimal Legendrian submanifold. 
\end{thm}
\begin{proof} 
It follows from Proposition~\ref{s5.3p3} and Proposition~\ref{s5.4p2} that the infinitesimal deformation space of $X$ as a minimal Legendrian submanifold is 
the sum of that of $X$ as a special Legendrian submanifold and the $1$-dimensional vector space generated by the Reeb vector field $\xi$ on $X$. 
Since the group ${\rm Aut}(\eta, g)$ is the automorphism group ${\rm Aut}(\eta, \xi, \Psi, g)$ of the Sasaki manifold (cf. Proposition 8.1.1 \cite{BG}), 
Proposition~\ref{s3.7p2} implies that $\mathfrak{aut}(\eta, g)=\mathfrak{aut}(\eta, \psi)\oplus \langle\xi\rangle_{\mathbb{R}}$. 
Hence $X$ is rigid as a minimal Legendrian submanifold if and only if $X$ is also rigid as a special Legendrian submanifold. 
It completes the proof. 
\end{proof}

\vspace{0.5\baselineskip}
\noindent
\textbf{Acknowledgements}. 
The author would like to thank the referee for his useful comments. 
This work was partially supported by GCOE `Fostering top leaders in mathematics', 
Kyoto University and by Grant-in-Aid for Young Scientists (B) $\sharp$21740051 from JSPS. 

\begin{center}

\end{center}

\begin{flushright}
\begin{tabular}{l}
\textsc{Takayuki Moriyama}\\
Department of Mathematics\\
Mie University\\
Mie 514-8507, Japan\\
E-mail: takayuki@edu.mie-u.ac.jp
\end{tabular}
\end{flushright}


\begin{thebibliography}{11}

\bibitem{BG}
C.P. Boyer and K. Galicki, 
\textit{Sasakian Geometry}, 
Oxford Mathematical Monographs, 
Oxford University Press, Oxford, (2008). 





\bibitem{EK}
A. El Kacimi-Alaoui,
\textit{Op\'{e}rateurs transversalement elliptiques 
sur un feuilletage riemannien et applications}, 
Compositio Math. \textbf{73} (1990), no. 1, 57--106. 

\bibitem{FHY}
A. Futaki, K. Hattori and H. Yamamoto, 
\textit{Self-similar solutions to the mean curvature 
flows on Riemannian cone manifolds and special Lagrangians on toric Calabi-Yau cones},
Osaka J. Math. \textbf{51} (2014), no. 4, 1053--1079.

\bibitem{FOW}
A. Futaki, H. Ono and G. Wang,
\textit{Transverse K\"{a}hler geometry of 
Sasaki manifolds and toric Sasaki-Einstein manifolds}, 
J. Differential Geom. \textbf{83} (2009), no. 3, 585--635.




\bibitem{HL}
R. Harvey and H. B. Lawson,
\textit{Calibrated geometries}, 
Acta Math. \textbf{148} (1982), 47--157. 

\bibitem{H}
M. Haskins,
\textit{The geometric complexity of special Lagrangian $T^2$-cones}, 
Invent. Math. \textbf{157} (2004),  no. 1, 11--70. 





\bibitem{J3}
D. D. Joyce,
\textit{Special Lagrangian submanifolds with 
isolated canonical singularities. I. Regularity}, 
Ann. Global Anal. Geom. \textbf{25} (2004), no. 3, 201--251. 

\bibitem{J2}
D. D. Joyce,
\textit{Special Lagrangian submanifolds with isolated conical singularities. II. Moduli spaces.}, 
Ann. Global Anal. Geom. \textbf{25} (2004), no. 4, 301--352. 

\bibitem{J1}
D. D. Joyce,
\textit{Riemannian holonomy groups and calibrated geometry}. 
Oxford Graduate Texts in Mathematics, 12. 
Oxford University Press, (2007). 






\bibitem{Mc}
R. C. McLean,
\textit{Deformation of calibrated submanifolds},
Comm. Anal. Geom. \textbf{6} (1998), 705--749.

\bibitem{M} 
T. Moriyama, 
\textit{The moduli space of transverse Calabi-Yau structures 
on foliated manifolds}, 
Osaka J. Math. \textbf{48} (2011), no. 2, 383--413. 

\bibitem{Mo} 
C. B. Morrey, 
\textit{Second-order elliptic systems of differential equations}, 
Contributions to the theory of partial differential equations, 101--159. 
Annals of Mathematics Studies, no. 33. Princeton University Press, Princeton, N. J., 1954. 



\bibitem{O2}
Y. Ohnita, 
\textit{Stability and rigidity of special Lagrangian cones over certain minimal Legendrian orbits}, 
Osaka J. Math. \textbf{44} (2007), no. 2, 305--334. 


\bibitem{O1}
Y. Ohnita, 
\textit{On deformation of 3-dimensional certain minimal Legendrian submanifolds}, 
Proceedings of The Thirteenth International Workshop on Diff. Geom., \textbf{13}, (2009), 71--87.



\bibitem{R}
B. Reinhart,
\textit{Foliated manifolds with bundle-like metrics}, 
Ann. of Math.(2) \textbf{69} (1959), 119--132.


\bibitem{S}
J. Sparks,
\textit{Sasaki-Einstein manifolds}, 
Surveys in differential geometry. Volume XVI. 
Geometry of special holonomy and related topics, (2012), 265--324.


\bibitem{Tanno}
S. Tanno, 
\textit{The topology of contact Riemannian manifolds}, 
Illinois. J. Math. \textbf{12} (1968), 700--717.


\bibitem{TY}
R. P. Thomas and S.-T. Yau, 
\textit{Special Lagrangians, stable bundles and mean curvature flow}, 
Comm. Anal. Geom. \textbf{10} (2002), no. 5, 1075--1113.


\bibitem{TV}
A. Tomassini and L. Vezzoni, 
\textit{Contact Calabi-Yau manifolds and special Legendrian submanifolds}, 
Osaka. J. Math. \textbf{45} (2008), no. 1, 127--147.

\end{thebibliography}
\end{document}